\documentclass[10pt, article]{amsart}

\usepackage{tikz}
\usetikzlibrary{calc}
\usepackage{ae} 
\usepackage[T1]{fontenc}
\usepackage[cp1250]{inputenc}
\usepackage{amsmath}
\usepackage{amssymb, amsfonts,amscd,verbatim}

\usepackage[normalem]{ulem}
\usepackage{hyperref}
\usepackage{indentfirst}
\usepackage{latexsym}
\input pstricks.tex
\input xy
\xyoption{all}

\usepackage{amsmath}    

\theoremstyle{plain}
\newtheorem{Pocz}{Poczatek}[section]
\newtheorem{Proposition}[Pocz]{Proposition}

\newtheorem{Theorem}[Pocz]{Theorem}
\newtheorem{Corollary}[Pocz]{Corollary}

\newtheorem{Lemma}[Pocz]{Lemma}
\newtheorem{Axiom}[Pocz]{Axiom}
\newtheorem{Condition}[Pocz]{Condition}
\newtheorem{Observation}[Pocz]{Observation}

\newtheorem{Question}[Pocz]{Question}

\newtheorem{Example}[Pocz]{Example}

\theoremstyle{definition}
\newtheorem{Definition}[Pocz]{Definition}

\theoremstyle{remark}
\newtheorem{Remark}[Pocz]{Remark}

\numberwithin{equation}{section}
\title[Axiomatization of geometry employing group actions]
{Axiomatization of geometry employing group actions}

\author{Jerzy Dydak}
\address{University of Tennessee, Knoxville, TN 37996, USA}
\email{jdydak@utk.edu}

\date{ \today
}


\begin{document}
\maketitle
\begin{center}
\today
\end{center}

\begin{abstract}
The aim of this paper is to develop a new axiomatization of planar geometry by reinterpreting the original axioms of Euclid. The basic concept is still that of a line segment but its equivalent notion of betweenness is viewed as a topological, not a metric concept. That leads quickly to the notion of connectedness without any need to dwell on the definition of topology. In our approach line segments must be connected. Lines and planes are unified via the concept of separation: lines are separated into two components by each point, planes contain lines that separate them into two components as well. We add a subgroup of bijections preserving line segments and establishing unique isomorphism of basic geometrical sets, and the axiomatic structure is complete. Of fundamental importance is the Fixed Point Theorem that allows for creation of the concepts of length and congruency of line segments.
The resulting structure is much more in sync with modern science than other axiomatic approaches to planar geometry. For instance, it leads naturally to the Erlangen Program in geometry. Our Conditions of Homogeneity and Rigidity have two interpretations. In physics, they correspond to the basic tenet that independent observers should arrive at the same measurement and are related to boosts in special relativity. In geometry, they mean uniqueness of congruence for certain geometrical figures.

Euclid implicitly assumes the concepts of length and angle measure in his axioms. Our approach is to let both of them emerge from axioms. Euclid obfuscates the fact that to compare lengths of line segments one needs rigid motions beforehand. Our system of axioms of planar geometry rectifies that defect of all current axiomatic approaches to planar geometry.

Another thread of the paper is the introduction of boundary at infinity, an important concept of modern mathematics, and linking of Pasch Axiom to endowing boundaries at infinity with a natural relation of betweenness. That way spherical geometry can be viewed as geometry of boundaries at infinity.

\end{abstract}

\tableofcontents

\section{Introduction}

At the end of the 19th century it became apparent that Euclid's axioms (see \cite{Eucl} or \cite{Stah}) of planar geometry are incomplete. See \cite{Gree1} for a short account and read \cite{Gree2} for an extensive account of historical developments aimed at rectifying axioms of Euclid.

Currently, there are three sets of axioms that describe geometry adequately. Those are (in historical order): Hilbert's axioms \cite{Hilb}, Tarski's axioms \cite{Tars}, and Birkhoff's axioms \cite{Birk}
(see \cite{SSTa} and \cite{Szmi} for discussions of axiomatic approaches to geometry). Of the three, Birkhoff's axioms are the most concise, yet they assume existence of reals and do not explain reals geometrically. The other two sets of axioms are difficult to memorize as they lack a clear scaffolding structure \cite{Scaf}.
Additionally, the post-Euclid sets of axioms have not gained a permanent foothold in the classroom - most teachers prefer using Euclid's axioms (see \cite{Stah} or \cite{Vene}  for examples of widely used textbooks where Euclid's axioms are prominent) and the other axiomatic approaches to geometry are considered more or less a curiosity. There is an interesting new textbook of Anton Petrunin \cite{Petr} using a variant of Birkhoff's axioms. Also, there is an axiomatic system of A. D. Alexandrov \cite{Alex} from 1994.

In this paper we develop a new axiomatization of planar geometry by reinterpreting the original Euclid's axioms. The basic concept is still that of a line segment but its equivalent notion of betweenness is viewed as a topological, not a metric concept. That leads quickly to the notion of connectedness without any need to dwell on the definition of topology. In our approach line segments must be connected. Lines and planes are unified via the concept of separation: lines are separated into two components by each point, planes contain lines that separate them into two components as well. We add a subgroup of bijections preserving line segments that establishes unique isomorphism of basic geometrical sets, and the axiomatic structure is complete. Of fundamental importance is the Fixed Point Theorem that allows for creation of the concepts of length and congruency of line segments.
The resulting structure is much more in sync with modern mathematics than other axiomatic approaches to planar geometry. For instance, it leads naturally to the Erlangen Program in geometry \cite{Erla}. Notice most results in modern geometry/topology start with group actions by isometries (see \v Svarc-Milnor Lemma in \cite{Roe lectures}  on p.8, \cite{BH}, \cite{BS}, or \cite{NY}). Our approach provides the missing link of constructing metrics for certain actions.

Let's look at

\begin{Axiom}[Axiom 1 of Euclid]\label{Axiom 1 of Euclid}
A straight line segment can be drawn joining any two points.
\end{Axiom}

If one overlooks the word "drawn", then Axiom \ref{Axiom 1 of Euclid} seems to simply say that for any two points $A$ and $B$, the line segment $AB$ exists. However, that is a very narrow interpretation.
The question is: What does it mean to draw a line segment? 
To answer it let's apply the relativity idea from physics: from the point of view of a pen it is the plane that is moving, not the pen. That implies the plane is equipped with motions called \textbf{translations}. For every two points $A$ and $B$ there is a unique translation $\tau_{AB}$ that sends point $A$ to point $B$:
$$\tau_{AB}(A)=B$$
Additional implication is that the line segment $AB$ ought to consist of points between $A$ and $B$. Yet another implication is that the segment $AB$ ought to contain no holes: drawing is a continuous motion. That leads to the concept of connectivity of line segments.

Similarly, one can reinterpret Axiom 3 of Euclid:

\begin{Axiom}[Axiom 3 of Euclid]\label{Axiom 3 of Euclid}
 Given any straight line segment, a circle can be drawn having the segment as radius and one endpoint as center.
\end{Axiom}

Again, relativistically speaking, from the point of view of a pen, it is the plane that is rotating, not the pen. That means the plane is equipped with motions called \textbf{rotations}. Translations and rotations are examples of \textbf{rigid motions} and their existence ought to be part of axiomatization of Euclidean geometry.

Let's look at
\begin{Axiom}[Axiom 2 of Euclid]\label{Axiom 2 of Euclid}
Any straight line segment can be extended indefinitely in a straight line.
\end{Axiom}

One possible interpretation is that the line joining two points $A$ and $B$ arises by applying the translation $\tau_{AB}$ to the line segment joining $A$ and $B$ over and over. A better way is to look at lines topologically. Namely, each point separates them into two components. That point of view can be extended to planes: lines of Euclid ought to separate planes into two components as well. It turns out this is equivalent to Pasch's Axiom. Thus, the correct way to interpret Axiom \ref{Axiom 2 of Euclid} is to say that each line segment is contained in a line (can be extended to a line) separating the plane into two components. That way the indefiniteness is the opposite of definiteness: a definite extension of a line segment does not separate the plane.

To summarize: Euclid's axioms are intuitive but incomplete. Hilbert's axioms are non-intuitive but complete. The same can be said of Tarski's axioms.
The problem is that Euclid assumes the concept of length of line segments and explains congruence via rigid motions yet he does not use rigid motions as part of the axiomatic system. Perhaps it is because it seems one needs the concept of length to define rigid motions. It is so if one concentrates on bounded geometrical figures only. In case of infinite figures one can define rigid motions without the concept of length. The way to arrive at it is to think how measurements occurs. In that way, a new system of axioms emerges which is both intuitive and complete.

Euclid assumes the concept of length in his axioms. Our approach is to let the concept of length emerge from axioms. Euclid obfuscates the fact that to compare lengths of line segments one needs rigid motions beforehand. Our system of axioms of planar geometry rectifies that defect of all current axiomatic approaches to planar geometry. Also, our system makes it possible to explain vectors without confusing them with points. Those are simply translations of a Euclidean plane (see \ref{CharOfEuclideanPlanesThm}). Notice we point out vectors do not exist for non-Euclidean geometries (see \ref{CharOfEuclideanPlanesThm}). 

Here is our scaffolding structure (or a Jacob's ladder in the sense of Marcel Berger \cite{Berg1} - see \cite{Berg2}, \cite{Berg3} for implementations of it) of axiomatizing of the Euclidean plane:
\begin{enumerate}
\item the basic concept is betweenness (or existence of line segments),
\item line segments ought to be connected,
\item lines are spaces that are separated by each point into two components,
\item Erlangen lines are lines with a group of isomorphisms that establish unique equivalence of any two maximal rays,
\item Erlangen lines have the concept of congruency of line segments and the concept of length of line segments,
\item all Erlangen lines are isomorphic,
\item circles are spaces with antipodal points that separate them into two lines,
\item Erlangen circles are circles with a group of isomorphisms that establish unique equivalence of any two maximal rays,
\item all Erlangen circles are isomorphic,
\item planes are spaces with every two points contained in a line separating the plane into two components,
\item Erlangen planes are planes with a group of isomorphisms that establish unique equivalence of any two half-planes,
\item Euclidean planes are Erlangen planes satisfying Axiom 5 of Euclid,
\item all Euclidean planes are isomorphic,
\item there exist non-Euclidean Erlangen planes.
\end{enumerate}

There are several references to real numbers prior to their formal introduction right after \ref{CreatingConnectedLines}. Obviously, those can be skipped if one wants a pure introduction of reals via geometry only. Similarly, one can omit more advanced concepts (linear vector spaces, linear order, etc).

The other thread of the paper is interpreting Axiom 2 of Euclid as one that leads to the concept of maximal rays, which in turn leads to the concept of the boundary at infinity of 
spaces with maximal rays. It is natural to seek conditions for the boundary at infinity to have a relation of betweenness induced by the betweenness of the original space. It turns out that natural condition is a variant of Pasch Axiom which guarantees the boundary at infinity to be a space with maximal rays that also satisfies Pasch Axiom (we call those Pasch spaces). In that setting Pasch spaces whose boundary at infinity consists of two points are exactly lines or circles and Pasch spaces whose boundary at infinity are circles are exactly planes. That thread provides another scaffolding structure (Jacob's ladder of Marcel Berger \cite{Berg1}) to geometry.

Foundations of geometry on one hand can be compared to Big Bang of physics since Euclid's Elements began modern mathematics. On the other hand Big Bang lasted mini-seconds and the search for the most effective way of presenting geometry is already taking more than 2300 years. The aim of this paper is to offer different rungs to Jacob's ladder of geometry with the goal of creating a comprehensive yet concise introduction to the subject: one lower rung is adding discussion of lines (as opposed to using them as an undefined concept), and another mobile rung is using group actions as a major tool that encodes chunks of mathematics.

The author is grateful to Nik Brodskiy, Chuck Collins, and Joan Lind for weekly conversations about philosophies of teaching and to Matic Cencelj, Brendon Labuz, and Maria Moszy\' nska for offering comments and suggestions that improved the exposition of the paper.

The author would like to thank Universidad Complutense de Madrid for hospitality during the summer of 2014 when the ideas of the paper started to emerge.

\section{Betweenness}
This section provides the background material for a rigorous discussion of Axiom 1 of Euclid \ref{Axiom 1 of Euclid}.

Euclid chose line segments as his basic concept. Both Hilbert and Tarski use the concept of betweenness which is connected to the physical world, hence less abstract. However, line segments are much more useful in developing a theory than betweenness, a nod to the wisdom of Euclid.

Betweenness is the first geometrical concept we encounter.

\textbf{Practical Definition}: A point $M$ is \textbf{between} points $A$ and $B$ if, when viewing the plane (or space) from point $A$, only point $M$ is visible.

\textbf{Example}: During a \textbf{lunar eclipse} the Earth is between the Moon and the Sun. 

\textbf{Example}: During a \textbf{solar eclipse} the Moon is between the Sun and the Earth. 

\textbf{Practical Definition}: A \textbf{line segment} with ends $A$ and $B$ is the set of all points between $A$ and $B$ including both $A$ and $B$.

\textbf{Example}: An arrow is a good practical example of a line segment.

A practical way to detect if a physical object is straight (i.e. it is a line segment) is to hold one of its ends in front of one's eye and look along the object.

\begin{Definition}\label{BetweennessDef}
A \textbf{betweenness} relation on a set $\Pi$ is a ternary relation satisfying the following conditions:\\
a. each point is between itself and any point of $\Pi$,\\
b. if a point $A$ is between points $B$ and $C$, then $A$ is between $C$ and $B$ as well,\\
c. Let $s[E,D]$ be the set of all points between points $E$ and $D$.
If $C$ is between $A$ and $B$, then $s[A,B]=s[A,C]\cup s[C,B]$ and $s[A,C]\cap s[C,B]=s[C,C]=\{C\}$. 
\end{Definition}

\begin{Remark}
Observe that both Hilbert and Tarski define betweenness quite differently from us. Their primary interest is extending line segments. Trying to derive \ref{BetweennessDef} from Hilbert's or Tarski's axioms is quite a challenge.
\end{Remark}

\begin{Definition}
$s(A,B):=s[A,B]\setminus \{A,B\}$, the set of points \textbf{strictly between} $A$ and $B$, will be called an \textbf{open line segment} and $s[A,B]$ will be called a \textbf{line segment} or a \textbf{closed line segment}.

A \textbf{half-open} line segment $s[A,B)$ is defined as $s[A,B):=s[A,B]\setminus \{B\}$.
Similarly, the half-open line segment $s(A,B]$ is defined as $s(A,B]:=s[A,B]\setminus \{A\}$.
\end{Definition}

Notice that two major mathematical structures carry implicitly a betweenness relation
and for the third one (metric spaces) the situation is a bit more complicated.

\begin{Example} A linear order $\leq$ on $\Pi$ induces betweenness as follows: $A$ is between $B$ and $C$ if $B\leq A\leq C$ or $C\leq A\leq B$.
\end{Example}

\begin{Example} A structure of a real vector space on $\Pi$ induces betweenness as follows: $A$ is between $B$ and $C$ if there is $0\leq t\leq 1$ such that $A=t\cdot B+(1-t)\cdot C$.
\end{Example}

\begin{Proposition}\label{SplittingOfClosedSegments}
Let $\Pi$ be a set with a relation of betweenness. If $S\subset s(A,B)$ is a finite set of points, then there is exactly one ordering $A_0=A, A_1,\ldots, A_n=B$ of $S\cup \{A,B\}$
such that $s[A,B]$ is the union of line segments $s[A_i,A_{i+1}]$ for $0\leq i < n$ and
if two segments $s[A_i,A_{i+1}]$, $s[A_j,A_{j+1}]$ have a non-empty intersection
for $i < j$, then $j=i+1$ and the intersection equals $\{A_{i+1}\}$.
\end{Proposition}
\begin{proof}
Notice there is exactly one ordering $A_0=A, A_1,\ldots, A_n=B$ of points in $S\cup \{A,B\}$ such that
$s[A,A_i]$ is a proper subset of $s[A,A_{i+1}]$ for all $i < n$.
Consequently, line segments $s[A_i,A_{i+1}]$ and $s[A_j,A_{j+1}]$ are disjoint if $i+1 < j$ and the fact that $s[A,B]$ is the union of line segments $s[A_i,A_{i+1}]$ for $0\leq i < n$ can be shown by induction.

If there is another ordering $B_0=A, B_1,\ldots, B_n=B$ of $S\cup \{A,B\}$
with the same property, then $A_1=B_1$ as otherwise (consider $A_1\in s(A,B_1)$ for simplicity) $s[B_0,B_1]$ intersects $s[B_j,B_{j+1}]$ (here $j$ is the index satisfying $A_1=B_j$) in its interior point.
By induction one can complete the proof.
\end{proof}

\begin{Corollary}\label{InclusionOfSegmentsProp}
Let $\Pi$ be a set with a relation of betweenness. If $C,D\in s[A,B]$, then $s[C,D]\subset s[A,B]$.
\end{Corollary}
\begin{proof}
Apply \ref{SplittingOfClosedSegments} to $S=\{C,D\}\setminus \{A,B\}$.
\end{proof}

\begin{Corollary}\label{BetwennessOfTripleProp}
Let $\Pi$ be a set with a relation of betweenness. 
Given three different points $A,B,C$ of a line segment $s[X,Y]$ in $\Pi$ there is exactly one pair such that the remaining point is between them.
\end{Corollary}
\begin{proof}
Apply \ref{SplittingOfClosedSegments} to $S=\{A,B,C\}\setminus \{X,Y\}$.
\end{proof}

\begin{Proposition}\label{IntersectionOfClosedSegments}
Let $\Pi$ be a set with a relation of betweenness. If a finite family $\mathcal{F}$ of closed line subsegments of a line segment $s[A,B]$ has non-empty intersection $I$, then that intersection $I$ is a closed line segment.
\end{Proposition}
\begin{proof}
Impose the linear order on the set of all endpoints of the family
$\mathcal{F}$ as in \ref{SplittingOfClosedSegments}. Notice $I=s[L,M]$, where $L$ is the maximum of left endpoints of segments in $\mathcal{F}$ and $M$ is the minimum of right endpoints of segments in $\mathcal{F}$.
\end{proof}

\subsection{Preserving betweenness}

Given a mathematical structure (betweenness in our case) on sets, it is natural to define morphisms and isomorphisms corresponding to that structure. There are two ways of looking at morphisms in the case of betweenness.

\begin{Definition}
Let $\Pi_1$ and $\Pi_2$ be two sets, each with its own relation of betweenness. A function
$f:\Pi_1\to \Pi_2$ \textbf{preserves betweenness} if $f(C)$ is between $f(A)$ and $f(B)$
whenever $C$ is between $A$ and $B$. In other words, $f(s[A,B])\subset s[f(A),f(B)]$ for all points $A,B$ of $\Pi_1$.
\end{Definition}

\begin{Definition}
Let $\Pi_1$ and $\Pi_2$ be two sets, each with its own relation of betweenness. A function
$f:\Pi_1\to \Pi_2$ \textbf{preserves closed line segments} if $f(s[A,B])=s[f(A),f(B)]$ for all points $A,B$ of $\Pi_1$.
\end{Definition}

Notice preserving of closed line segments implies preserving betweenness but the converse does not hold.

\begin{Example}
Consider the unit interval $[0,1]$ with the standard relation of betweenness (induced by the order on $[0,1]$). One can give $[0,1]$ the trivial relation of betweenness: each closed interval $s[A,B]$ is declared to be equal to $\{A,B\}$. Notice the identity function from the trivial $[0,1]$ to the standard $[0,1]$ preserves betweenness but does not preserve closed line segments.
\end{Example}

The above example justifies the following definitions:

\begin{Definition}\label{MorphismIsomorphismDef}
Let $\Pi_1$ and $\Pi_2$ be two sets, each with its own relation of betweenness. 
A \textbf{morphism} from $\Pi_1$ to $\Pi_2$ is a function $f:\Pi_1\to \Pi_2$ preserving closed line segments.
An \textbf{isomorphism} from $\Pi_1$ to $\Pi_2$ is a bijection $f:\Pi_1\to \Pi_2$ preserving closed line segments.
\end{Definition}

Notice that in the definition of an isomorphism in \ref{MorphismIsomorphismDef} we do not have to assume $f^{-1}$ preserves closed line segments. It follows automatically from the fact $f$ preserves closed line segments as shown in the next proposition.

\begin{Proposition}\label{BijectionPreservingBetweennessProp}
Let $f:\Pi_1\to \Pi_2$ be a bijection of sets, each equipped with its own relation of betweenness. \\
a. If $f$ preserves closed line segments, then so does $f^{-1}$.\\
b. If both $f$ and $f^{-1}$ preserve betweenness, then they preserve closed line segments as well.
\end{Proposition}
\begin{proof}
a. Suppose $f(C)\in s[f(A),f(B)]$. Since $f(s[A,B])=s[f(A),f(B)]$ there is $D\in s[A,B]$
satisfying $f(D)=f(C)$. Because $f$ is a bijection, $C=D$ and $f^{-1}(s[f(A),f(B)])=
s[A,B]$.\\
b.
Suppose $A,B\in \Pi_1$ and $C\in s[f(A),f(B)]$. Choose $D\in \Pi_1$ with $f(D)=C$.
Since $f^{-1}$ preserves betweenness, $D\in s[A,B]$. That shows $f(s[A,B])=s[f(A),f(B)]$.
\end{proof}

\begin{Proposition}
Let $\Pi$ be a set with a relation of betweenness. Self-isomorphisms
$f:\Pi\to\Pi$ form a group.
\end{Proposition}

\section{Connectedness and convexity}\label{Convexity and connectedness}

This section contains a discussion of two concepts that are of primary importance
for the paper: connectedness and convexity. Notice connectedness is called continuity in foundations of geometry - see \cite{Gree1}.

The first attempt at attempt at connectedness arises when measuring line segments: a person trying to figure out the length of a line segment in feet is carefully planting consecutive steps so that there is no visible space between them. Mathematically it amounts to the requirement that a closed line segment cannot be expressed as the union of two disjoint closed line segments.

Geometrically, the first attempt at connectedness is that each non-empty closed line segment cannot be expressed as a half-open line segment.

The ultimate concept of connectedness of a closed line segment amounts to saying that there are no holes in it. From the point of view of physics it means the finality of results of measurements: making measurements with increasing precision should converge to a concrete amount.

The following result summarizes the first attempt at connectedness.
\begin{Proposition}\label{FirstApproximationOfConnectedness}
Let $\Pi$ be a set with a relation of betweenness. The following conditions are equivalent:\\
1. If $s[A,B]=s[C,D]\cup s[E,F]$, then $s[C,D]\cap s[E,F]\ne\emptyset$ or $s[A,B]=\emptyset$,\\
2. $s(A,B)$ is not empty if $A\ne B$,\\
3. Equality $s[A,B]=s[C,D)$ is not possible.
\end{Proposition}
\begin{proof}
1.$\implies$2. Suppose $s(A,B)$ is empty and $A\ne B$. Now, 
$s[A,B]=s[A,A]\cup s[B,B]$ and the two line segments are disjoint, a contradiction.

2.$\implies$3. Suppose $s[A,B]=s[C,D)$ for some points $A, B, C, D$ of $\Pi$. 
We may assume, using \ref{SplittingOfClosedSegments}, that points
$A, B, C, D$ are ordered as $A_0=C$, $A_1=A$, $A_2=B$, and $A_4=D$,
with the possibility of some of the first three points being equal. However, $B\ne D$,
which implies $s(B,D)=\emptyset$, a contradiction.

3.$\implies$1.
Suppose $s[A,B]=s[C,D]\cup s[E,F]$ is non-empty and $s[C,D]\cap s[E,F]=\emptyset$.
One of the points $C, D, E, F$ must be equal to $A$ and one of them must be equal to $B$. Without loss of generality, we may assume $D=A$ and $F=B$.
We may assume, using \ref{SplittingOfClosedSegments}, that points
$A, B, C, E$ are ordered as $A_0=A$, $A_1=C$, $A_2=E$, and $A_4=B$,
with the possibility that $A=C$ or $E=B$. In any case, $s(C,E)=\emptyset$
and $s[C,C]=s[C,E)$, a contradiction.
\end{proof}

To express the final attempt at connectedness we need the concept of convexity.

\begin{Definition}\label{ConvexDef}
Let $\Pi$ be a set with a relation of betweenness. A subset $S$ of $\Pi$
is \textbf{convex} if for all $A,B\in S$ the closed line segment $s[A,B]$ is contained in $S$.
\end{Definition}

\begin{Definition}\label{ConvConnectednessDef}
Let $\Pi$ be a set with a relation of betweenness. A closed line segment $s[A,B]$ is \textbf{connected} if, whenever expressed as a union of two disjoint non-empty convex subsets, precisely one of them is a closed line segment.
\end{Definition}

\begin{Example}
In the set $\mathbb{Q}$ of rational numbers with the standard relation of betweenness, the line segment $s[0,4]$ is not connected. Indeed, the sets $\{t\in s[0,4]\mid t^2 < 2\}$ and $\{t\in s[0,4]\mid t^2 > 2\}$ are both convex,
their union is $s[0,4]$, and none of them is a closed line segment.
\end{Example}

Our final attempt at connectedness is stronger than the first attempt.
\begin{Lemma}
Let $\Pi$ be a set with a relation of betweenness. If a closed line segment $s[A,B]$
is connected, then the open line segment $s(C,D)$ is non-empty for any two different points $C, D$ of $s[A,B]$.
\end{Lemma}
\begin{proof}
Assume $C < D$ as in \ref{SplittingOfClosedSegments}. If $s(C,D)$ is empty, then $s[A,B]=s[A,C]\cup s[D,B]$
is a union of two non-empty closed line segments that are disjoint, a contradiction.
\end{proof}

\begin{Proposition}\label{ConnectednessOfSubsegments}
Let $\Pi$ be a set with a relation of betweenness. If $C,D\in s[A,B]$ and the line segment $s[A,B]$ is connected, then $s[C,D]$ is connected as well.
\end{Proposition}
\begin{proof}
Assume $s[C,D]\subset s[A,D]$.
Suppose $s[C,D]$ is the union $U\cup V$ of two non-empty convex subsets that are disjoint. Without loss of generality, assume $C\in U$ and $D\in V$. Notice $s[A,C]\cup U$ and $V\cup s[D,B]$
are both convex in $s[A,B]$ and disjoint. Hence exactly one of them, say $V\cup s[D,B]$
is a closed line segment. Thus $V\cup s[D,B]=s[E,B]$
and $V=(V\cup s[D,B])\cap s[C,D]=s[E,B]\cap s[C,D]=s[E,D]$
is a closed line segment. It is easy to see $U$ is not a closed line segment.
\end{proof}

\begin{Proposition}\label{ConnectednessOfUnionOfSegments}
Let $\Pi$ be a set with a relation of betweenness. If $C\in s[A,B]$ and the line segments $s[A,C]$ and $s[C,B]$ are connected, then $s[A,B]$ is connected as well.
\end{Proposition}
\begin{proof}
Suppose $s[A,B]$ is the union $U\cup V$ of two non-empty convex subsets that are disjoint. Without loss of generality, assume $A\in U$ and $B\in V$. Also, we may assume $C\in V$ which implies $s[C,B]\subset V$.
Notice $s[A,C]\cap U$ and $V\cap s[A,C]$
are both convex in $s[A,C]$ and disjoint. Hence exactly one of them, say $V\cap s[A,C]$
is a closed line segment. Thus $V\cap s[A,C]=s[E,C]$ for some $E\in s[A,C]$
and $V=(V\cap s[A,C])\cup s[C,B]=s[E,C]\cup s[C,B]=s[E,B]$
is a closed line segment. It is easy to see $U$ is not a closed line segment.
\end{proof}

\begin{Proposition}\label{IntersectionOfInfinitelyClosedSegments}
Let $\Pi$ be a set with a relation of betweenness. If a family $\mathcal{F}$ of closed line subsegments of a connected line segment $s[A,B]$ has non-empty intersection $I$, then that intersection $I$ is a closed line segment.
\end{Proposition}
\begin{proof}
Notice $I$ is convex. Impose the linear order on $s[A,B]$ with $X\leq Y$ meaning $s[A,X]\subset s[A,Y]$. 
Our strategy is to show $I$ has both a minimum and a maximum in which case $I$ is a closed segment due to its convexity.

If both $A$ and $B$ belong to the intersection $I$, then $I=s[A,B]$ and we are done.
Assume $A\notin I$. If $B\in I$, then the complement $I'$ of $I$ is convex and the worst case is $I'=s[A,C]$ for some $C$. Since $C\notin I$, there is a closed line segment
$s[A',B]\in \mathcal{F}$ with $C \notin s[A',B]$. As $C < A'$ there is $D\in s(C,A')$
that belongs to $I'$, a contradiction.

The remaining case is when both $A$ and $B$ are outside of $I$ and $I$ separates $s[A,B]$ into two convex sets: $S_A$ and $S_B$, with $A \in S_A$ and $B \in S_B$. As above, $S_A$ cannot be a closed line segment, so $I\cup S_B$ is a closed line segment.
Similarly, $I\cup S_A$ is a closed line segment and their intersection, equal $I$,
is a closed line segment.
\end{proof}

\begin{Corollary}\label{IntersectionOfDecreasingSegments}
Let $\Pi$ be a set with a relation of betweenness. If $s[A_n,B_n]$ is a connected line segment in $\Pi$ for $n\ge 1$ and $s[A_n,B_n]\subset s[A_m,B_m]$
for all $n > m$, then the intersection of all line segments $s[A_n,B_n]$ is a non-empty closed line segment.
\end{Corollary}
\begin{proof}
Impose the linear order on $s[A_1,B_1]$ with $X\leq Y$ meaning $s[A_1,X]\subset s[A_1,Y]$. 
Without loss of generality we may assume $A_n < B_n$ for all $n$: if $A_k=B_k$ for some $k$, then clearly the intersection of all segments is $s[A_k,B_k]$. Consider $U=\bigcup\limits_{i=1}^\infty [B_i,B_1]$
and $V=\bigcup\limits_{i=1}^\infty s[A_1,A_i]$. Both are convex, disjoint, and none is a closed line segment. Hence there is $C$ in the complement of $U\cup V$
and it belongs to the intersection of all intervals $s[A_n,B_n]$, $n\ge 1$.
\end{proof}

The following is a fundamental result for the whole paper.

\begin{Theorem}[Fixed Point Theorem]\label{FixedPointTheorem}
Let $\Pi$ be a set with a relation of betweenness and let $f:s[A,B]\to s[A,B]$ be a one-to-one function such that $f(s[C,D])=s[f(C),f(D)]$ for all points $C,D\in s[A,B]$.
If the line segment $s[A,B]$ is connected, then $f$ has a fixed point.
\end{Theorem}
\begin{proof}
Assume $f$ has no fixed points.\\
\textbf{Special Case}: $f(s[A,B])=s[A,B]$. That implies $f(A)=B$ and $f(B)=A$
as $A$ and $B$ are the only points of $s[A,B]$ that are not between different points.

Consider $U=\{C\mid C\in s[A,f(C)]\}$ and $V= \{C\mid C\in s[B,f(C)]\}$. 
Notice $A\in U$ and $s[A,C]\subset U$ if $C\in U$. That means $U$ is convex.
Indeed, $X\in s[A,C]$ implies $f(X)\in s[B,f(C)]$, so $X\in [A,f(X)]$.
Similarly, $V$ is convex.

If $C\in s[A,f(C)]$, then $f(C)\in s[B,f^2(C)]$, so $f(U)\subset V$. Similarly, $f(V)\subset U$.
Therefore $f(U)=V$ and $f(V)=U$ as $U$ and $V$ are disjoint.

$U$ cannot be a closed interval, as in that case $V$ is a closed interval as well contradicting connectedness of $s[A,B]$.
Therefore $V$ is a closed interval resulting in $U$ being a closed interval, a contradiction again.

\textbf{General Case}:
If $f(s[A,B])\ne s[A,B]$, consider the intersection of all $f^n(s[A,B])$, $n\ge 1$.
It is equal to $s[A',B']$ for some $A'$, $B'$ by \ref{IntersectionOfDecreasingSegments}.
Notice $f(s[A',B'])=s[A',B']$, so $f$ has a fixed point by the Special Case.
\end{proof}

\subsection{Another view of connectedness}

This part is to show that our definition of connectedness coincides with the standard concept of connectedness in the topology induced by a relation of betweenness.

\begin{Definition}\label{NbhdDef}
Let $\Pi$ be a set with a relation of betweenness. A subset $U$ of a closed line segment $s[A,B]$ is called a \textbf{neighborhood} of $X\in U$ in $s[A,B]$ if either\\
(i) there are points $D,E$ such that $X \in s(D,E) \subset U$\\
or\\
(ii) $X=A$ (or $X=B$) and there is $D \in s(A,B)$ such that $s[X,E) \subset U$.
\end{Definition}

\begin{Definition}\label{OpenDef}
Let $\Pi$ be a set with a relation of betweenness. A subset $U$ of a closed line segment $s[A,B]$ is called \textbf{open} in $s[A,B]$ if it is a neighborhood in $s[A,B]$ of its every point.
\end{Definition}

\begin{Theorem}\label{ConnectednessConvexityThm}
Let $\Pi$ be a set with a relation of betweenness. A closed line segment $s[A,B]$ in $\Pi$ is connected if and only if it cannot be expressed as a union of two disjoint non-empty open subsets.
\end{Theorem}
\begin{proof}
If there is a decomposition of $s[A,B]$ into two disjoint non-empty convex subsets
$U$ and $V$
such that either both of them are closed line segments or none of them is a closed line segment, then both $U$ and $V$ are open.

Let $s[A,B]$ be the union of two disjoint non-empty open subsets $U$ and $V$. Let $S_1$ be the union of all closed line segments $s[A,X]$ such that $X$ belongs to $U$. Then $S_1$ is a convex subset of $s[A,B]$ and its complement $S_2$ is convex as well. It cannot happen that exactly one of sets $S_1$, $S_2$ is a closed line segment. Indeed, if $S_1=s[A,E]$, then $E\in U$
and there is $F$ such that $s[E,F)\subset U$. In this case $F\in V$
(otherwise $s[A,F]\subset S_1$, a contradiction) and $V=s[F,B]$.
A similar argument works if $S_2$ is a closed line segment.
\end{proof}

\section{Spaces with maximal rays}

In this section we formalize spaces that satisfy Axiom 2 of Euclid \ref{Axiom 2 of Euclid}. However, our point of view is to extend line segments to rays rather than straight lines (notice that Euclid did not consider infinite lines, so our approach can still be considered as the one following Euclid's ideas).
It is so because we plan to draw on a basic idea from physics of a space being isotropic (i.e. being the same in all directions) and rays correspond to directions in our system.
Also, rays lead naturally to the concept of boundary at infinity, a fundamental idea of contemporary mathematics (see \cite{Grom} and \cite{BH}).

\begin{Definition}\label{RayDef}
Let $\Pi$ be a set with a relation of betweenness. A \textbf{ray} $r$ in $\Pi$ \textbf{emanating} from $A$ is a convex subset
satisfying the following conditions:\\
1. $A\in r$ but $A$ is not between any two other points of $r$,\\
2. every three points of $r$ are contained in a closed line segment which is a subset of $r$,\\
3. $r$ is not equal to any closed line segment.

A \textbf{maximal ray} emanating from $A$ is a ray that is not a proper subset of another ray emanating from $A$.
\end{Definition}

\begin{Example}
$s[A,B)$ is a ray if $A\ne B$ and $s[A,B)$ is not a closed line segment.
\end{Example}

\begin{Observation}\label{OrderingOfRays}
$r$ is a ray emanating from $A$ if and only if it is convex and the order on $r$, defined by $B \leq C$ meaning $s[A,B]\subset s[A,C]$, is linear with $A$ being the minimum, and $r$ has no maximum.
\end{Observation}

\begin{Proposition}\label{MaximalRaysProp}
Let $\Pi$ be a set with a relation of betweenness. Assume that if line segments $s[A,B]$ and $s[A,C]$ have a common interior point, then one of them is contained in the other.
If a closed line segment $s[A,B]$, where $A\ne B$, is a subset of a ray $r$ emanating from $A$, then the union $r'$ of all rays emanating from $A$ and containing $B$ is the maximal ray emanating from $A$ and containing $s[A,B]$.
\end{Proposition}
\begin{proof}
Notice that, given two rays emanating from $A$ and containing $B$, one of them is contained in the other. Therefore $r'$ cannot have a maximal element, is connected, and every three points of $r'$ are contained in a closed line segment which is a subset of $r'$. That means $r'$ is a ray.
\end{proof}

The next definition summarizes \ref{MaximalRaysProp} and provides the background material for one way of understanding of Axiom 2 of Euclid \ref{Axiom 2 of Euclid}.

\begin{Definition}\label{SpaceWithMaximalRays}
$\Pi$ is a space with \textbf{maximal rays} if it is a set with a relation of betweenness
satisfying the following conditions:\\
1.
if two line segments $s[A,B]$ and $s[A,C]$ have a common interior point, then one of them is contained in the other.\\
2. if $A\ne B$ and $s[A,B]$ is connected, then there is a connected line segment $s[C,D]$ containing $s[A,B]$ in its interior.
\end{Definition}

In spaces with maximal rays we have a natural concept of rays being antipodal.

\begin{Definition}\label{AntipodalRaysDef}
Let $\Pi$ be a space with maximal rays. Two maximal rays $r_1$ and $r_2$ emanating from the same point $A$ are \textbf{antipodal} if there are points $B\in r_1$ and $C\in r_2$ such that $A\in s(B,C)$.
Equivalently, $r_1\cup r_2$ is convex.
\end{Definition}

Notice that each maximal ray has at most one antipodal maximal ray and every connected maximal ray (that means its subsegments are connected) has an antipodal maximal ray.

\section{Boundary at infinity and Pasch Axiom}

Given a space $\Pi$ with maximal rays we can interpret Euclidean angles as unordered pairs of connected maximal rays in $\Pi$ emanating from the same point.
The modern point of view is that connected maximal rays in $\Pi$ emanating from $A$ form the \textbf{boundary at infinity} $\partial(\Pi,A)$. Notice rays are spaces with the simplest boundary at infinity, namely one-point space.

\begin{Definition}
Given two different points $A$ and $B$ of a space $\Pi$ with maximal rays such that $s[A,B]$ is connected, $ray[A,B]$ is defined as the maximal ray emanating from $A$ and containing $B$.
\end{Definition}

It is natural to seek a betweenness structure on $\partial(\Pi,A)$. It turns out Pasch Axiom is exactly what is needed. 

There are two versions of Pasch Axiom used in the literature:\\
1. If $C$ is between $A$ and $B$ and $E$ is between $A$ and $D$, then line segments $s[C,D]$ and $s[B,E]$ intersect.\\
2. A line intersecting one side of a triangle must intersect another side as well.

We need a variant of Pasch's Axiom that fits our purposes better as it applies to non-planar cases as well. It is right in the middle of the two above axioms.

\begin{Lemma} Let $\Pi$ be a space with maximal rays. Consider the following conditions:\\
1. (Weak Pasch Condition) If $C$ is between $A$ and $B$ and $E$ is between $A$ and $D$, then line segments $s[C,D]$ and $s[B,E]$ intersect.\\
2. If $C$ is between $A$ and $B$ and $E$ is between $A$ and $D$, then for any $G$ between $B$ and $D$ there is
$H\in s[C,E]$ that is between $A$ and $G$.\\
3. (Medium Pasch Condition) If $C$ is between $A$ and $B$ and $E$ is between $A$ and $D$, then for any $G$ between $B$ and $D$ there is
$H\in s[C,E]$ that is between $A$ and $G$ and for any $H\in s[C,E]$ there is $G\in s[B,D]$ such that $H$ is between $A$ and $G$.\\
4. (Pasch Condition) Given a subset $\{r_1,r_2,r_3\}$ of $\partial(\Pi,A)$ disjoint with its antipodal set and given
points $B,C\in r_1$ and $D\in r_3$, if $r_2$ intersects $s[B,D]$, then it also intersects $s[C,D]$.\\
5. (Strong Pasch Condition) A maximal line intersecting one side of a triangle must intersect another side as well.

Conditions 1. and 2. are equivalent. Conditions 3. and 4. are equivalent. Higher numbered conditions imply lower numbered conditions.
\end{Lemma}

\begin{proof}
2.$\implies$1. Suppose $C$ is between $A$ and $B$ and $E$ is between $A$ and $D$. Look at the situation from the point of view of the point $D$. Condition 2 says that $s[D,C]$ must intersect $s[B,E]$.

1.$\implies$2. Suppose $C$ is between $A$ and $B$, $E$ is between $A$ and $D$, and $G$ is between $B$ and $D$.
Condition 1 (viewed from $D$) says that $s[A,G]$ must intersect $s[B,E]$
at some point $G'$. Condition 1 (viewed from $B$) says that $s[A,G']$ must intersect $s[C,E]$
at some point $H$. Hence, there is
$H\in s[C,E]$ that is between $A$ and $G$.

Obvously, Condition 3 implies Condition 2.

2. If $C$ is between $A$ and $B$ and $E$ is between $A$ and $D$, then for any $G$ between $B$ and $D$ there is
$H\in s[C,E]$ that is between $A$ and $G$.

3. (Medium Pasch Condition) If $C$ is between $A$ and $B$ and $E$ is between $A$ and $D$, then for any $G$ between $B$ and $D$ there is
$H\in s[C,E]$ that is between $A$ and $G$ and for any $H\in s[C,E]$ there is $G\in s[B,D]$ such that $H$ is between $A$ and $G$.

3.$\implies$4. Suppose a subset $\{r_1,r_2,r_3\}$ of $\partial(\Pi,A)$ is disjoint with its antipodal set and there are
points $B,C\in r_1$ and $D\in r_3$. Pick a point $B'\in r_1$ such that both $B$ and $C$ are in $s[A,B']$.
Applying Condition 3 one notices that $r_2$ intersects $s[D,B']$ if and only if it intersects $s[D,B]$.
Applying Condition 3 again, one notices that $r_2$ intersects $s[D,B']$ if and only if it intersects $s[D,C]$. Therefore,
if $r_2$ intersects $s[B,D]$, then it also intersects $s[C,D]$.

4.$\implies$3. Suppose $C$ is between $A$ and $B$ and $E$ is between $A$ and $D$. Let $r_1:=ray[A,B]$
and $r_3:=ray[A,D]$. 
Given any $G$ between $B$ and $D$ define $r_2$
as $ray[A,G]$. If any two of the rays $r_1$, $r_2$, $r_3$ are antipodal, then all points are on the same line and existence of $H$ on $s[A,G]\cap s[C,E]$ follows from facts about line segments. Otherwise, Condition 4 says $r_2$ intersects $s[B,E]$ at some point $G'$. 

Notice $G'$ is between $A$ and $G$. Indeed, for the same reason $ray[B,E]$ intersects $s[A,G]$ and that point must be $G'$ as otherwise all points $A$, $B$, $D$ are on the same line which we already dealt with.

Apply the same reasoning again to conclude $s[A,G']$ intersects $s[C,E]$ at some point $H$.

The second case (for any $H\in s[C,E]$ there is $G\in s[B,D]$ such that $H$ is between $A$ and $G$) has analogous proof. 

5.$\implies$4. This follows from the fact that spaces satisfying the Strong Pasch Condition are planes (see subsequent sections) and planes satisfy the Medium Pasch Condition.
\end{proof}

\begin{Definition}\label{PaschCondition}
A space $\Pi$ with maximal rays satisfies \textbf{Pasch Condition}
if, given a subset $\{r_1,r_2,r_3\}$ of $\partial(\Pi,A)$ disjoint with its antipodal set and given
points $B,C\in r_1$ and $D\in r_3$, if $r_2$ intersects $s[B,D]$, then it also intersects $s[C,D]$.
\end{Definition}

We are ready to formulate the class of spaces that have boundaries at infinity at each point and those boundaries have the antipodal map.  
\begin{Definition}\label{PaschSpacesDef}
A \textbf{Pasch space} is a space with maximal rays satisfying Pasch Condition.

We are going to consider only two kinds of Pasch spaces: a \textbf{spherical Pasch space} $\Pi$ is one equipped with the antipodal map $a$ (i.e. $a^2=id$ and $a(B)\ne B$ for all $B\in \Pi$) preserving closed line segments
satisfying the following conditions:\\
1. each pair of non-antipodal points $A,B\in \Pi$, the segment $s[A,B]$ is connected,\\
2. $s[A,B]$ consists of exactly two points if $a(B)=A$,\\
3. if $r_1\cup r_2$ contains a pair of antipodal points, then $r_1$ and $r_2$ are antipodal if they emanate from the same point.

A \textbf{regular Pasch space} is one in which all line segments are connected.
\end{Definition}

\begin{Proposition}\label{StrongPaschCondition}
Let $\Pi$ be a regular Pasch space with three points $A,B,D$ not lying in a closed line segment. \\
a. If $C\in s(A,B)$, then for every $F\in s[B,D]$ there is unique $E\in s[D,C]\cap s[A,F]$.\\
b.  If $C\in s(A,B)$, then for every $E\in s[C,D]$ there is unique $F$ such that $s[A,F]\cap s[C,D]$ contains $E$.
\end{Proposition}
\begin{proof}
Notice $ray[A,B]$ and $ray[A,D]$ are not antipodal.\\
1. There is $E\in s[C,D]$ that belongs to $ray[A,F]$. The only issue is whether
$E\in s[A,F]$. However, $ray[D,C]$ also intersects $s[A,F]$ at some point $E'$.
If $E'\ne E$, then points $A, D, C$ lie on one segment resulting in points  
$A,B,D$ lying in a closed line segment, a contradiction. For the same reason $E$ is unique.\\
2. Pick $F\in s[B,D]$ lying on $ray[A,E]$. By 1., $E\in s[A,F]$.
\end{proof}

\begin{Proposition}\label{BoundaryAtInfinityProp}
If $\Pi$ is a Pasch space, then for every point $A$ of $\Pi$ the boundary of infinity $\partial(\Pi,A)$ has a natural relation of betweenness defined as follows: \\
1. if $r_1$ and $r_2$ are antipodal or equal, then $s[r_1,r_2]$ is defined to be $\{r_1,r_2\}$,\\
2. if $r_1$ and $r_2$ are not antipodal and different, then
$r_3$ is declared to be between $r_1$ and $r_2$ if for every choice of $B\in r_1$ and $C\in r_2$, the ray $r_3$ intersects $s[B,C]$.
\end{Proposition}
\begin{proof}
Using Pasch Condition observe that 2. can be reformulated in a weaker form: 
$r_3$ is between non-antipodal $r_1$ and $r_2$ if for some choice of $B\in r_1\setminus \{A\}$ and $C\in r_2\setminus \{A\}$, the ray $r_3$ intersects $s[B,C]$.

Suppose $r_3, r_4$ are between rays $r_1$ and $r_2$. In order to show $s[r_1,r_2]=s[r_1,r_3]\cup s[r_3,r_2]$ assume $r_4$ is not between $r_1$ and $r_3$. 
Choose $B\in r_1$ and $C\in r_3$ such that $r_4$ does not intersect $s[B,C]$.
Also, choose $D\in r_2\setminus \{A\}$. Let $E$ be in the intersection of $r_3$ and $s[B,D]$.
Using Pasch Condition notice $r_4$ cannot intersect $s[B,E]$. Since $r_4$
intersects $s[B,D]$, it must intersect $s[E,D]$ which shows that
$r_4$ is between $r_3$ and $r_2$.
\end{proof}

\begin{Lemma}\label{PaschBetweennessLemma}
Let $\Pi$ be a Pasch space. If $r_2\in \partial(\Pi,A)$ is strictly between $r_1\ne r_3$, then $r_3$ is between $r_2$ and the antipodal ray $a(r_1)$ to $r_1$.
\end{Lemma}
\begin{proof}
It suffices to consider the case $r_3\ne r_2$. Choose points $B\in r_1\setminus \{A\}$, $C\in r_3\setminus \{A\}$ and let $D$ belong to the intersection of $r_2$ and $s[B,C]$. Choose $B'\in a(r_1)\setminus \{A\}$. The maximal ray emanating from $B'$ and containing $D$ must intersect $s[A,C]$ at some point $E$ as it intersects $s[C,B]$ unless
two of rays $ray[B',B]$, $ray[B',C]$, and $ray[B',D]$ are antipodal and existence of $E$ is not guaranteed.
If so, the only interesting case is that of $ray[B',C]$, and $ray[B',D]$ being antipodal, since the other cases lead to $r_2=a(r_1)$ or $r_3=a(r_1)$
and $r_3$ is between $r_2$ and $a(r_1)$ as $r_3=a(r_1)$ is not possible.
If $ray[B',C]$ and $ray[B',D]$ are antipodal, then $A\in s[D,C]$ and $r_2=r_1$
contradicting $r_2$ being strictly between $r_1\ne r_3$.
\end{proof}

\begin{Theorem}
If $\Pi$ is a Pasch space, then its boundary at infinity $\partial(\Pi,A)$ is a spherical Pasch space for every point $A$ of $\Pi$.
\end{Theorem}
\begin{proof}
In order to show $\partial(\Pi,A)$ is a space with maximal rays,
we need to check that assumptions in \ref{SpaceWithMaximalRays} are satisfied.
Therefore assume two line segments $s[r_1,r_2]$ and $s[r_1,r_3]$ have a common interior point $r_4$. That implies neither $r_2$ nor $r_3$ is antipodal to $r_1$. 
By \ref{PaschBetweennessLemma} both $r_2$ and $r_3$ are between $r_4$ and
$a(r_1)$. Therefore we may assume $r_2$ is between $r_4$ and $r_3$.
That implies $r_2$ is between $r_1$ and $r_3$ resulting in  
$s[r_1,r_2]\subset s[r_1,r_3]$.

If $r_1$ is not antipodal to $r_2$, we choose points $B\in r_1\setminus \{A\}$, $C\in r_2\setminus \{A\}$, extend $s[B,C]$ to $s[D,E]$ containing $s[B,C]$ in its interior, and observe $s[ray[A,D],ray[A,E]]$ contains $s[r_1,r_2]$ in its interior.

The antipodal map on $ \partial(\Pi,A)$ is defined as follows:
$a(ray[r_1,r_2]):=ray[r_1,a(r_2)]$.

To verify the validity of our definition we need to show
$$ray[r_1,r_2]=ray[r_1,r_3]$$
$$\Downarrow$$
$$ray[r_1,a(r_2)]=ray[r_1,a(r_3)]$$

It suffices to consider the special case of $r_2$ being strictly between $r_1$ and $r_3$. Now,
it follows by applying \ref{PaschBetweennessLemma} twice. 

Indeed, \ref{PaschBetweennessLemma} converts a triple of points of $\partial(\Pi,A)$
$$U, V, W$$
to the triple
$$V, W, a(U)$$
when the notation means the middle ray is strictly between the other two. Indeed, $W=a(U)$ is not possible as in that case the open line segment $s(U,W)$ is empty contrary to it containing $V$.

Applying that move twice gives
$$r_3, r_2, r_1$$
$$\Downarrow$$
$$r_2, r_1, a(r_3)$$
$$\Downarrow$$
$$r_1, a(r_3), a(r_2)$$
$$\Downarrow$$
$$ray[r_1,a(r_2)]=ray[r_1,a(r_3)]$$
Suppose $r,r_1,r_2,r_3\in \partial(\Pi,A)$ such that $ray[r_1,r_2]$, $ray[r_1,r_3]$ exist and $r\in ray[r_1,r_2]$, $a(r)\in ray[r_1,r_3]$.
In that case $ray[r_1,r_2]=ray[r_1,r]$ and $ray[r_1,r_3]=ray[r_1,a(r)]$
which means $ray[r_1,r_2]$ is antipodal to $ray[r_1,r_3]$ as needed.

To show $\partial(\Pi,A)$ satisfies Pasch Condition \ref{PaschCondition} consider $r_A, r_B, r_C, r_D, r_E$ such that $r_C$ is strictly between $r_A$ and $r_B$, $r_E$ is strictly between $r_D$ and $r_C$, and 
$$\{ray[r_A,r_B],ray[r_A,r_D],ray[r_A,r_E]\}$$
 is disjoint with its antipodal image.
Choose points $A'\in r_A$, $B'\in r_B$, $D'\in r_D$ different from $A$ and let $C'\in r_C\cap s[A',B']$, $E'\in r_E\cap s[C',D']$. 

From the point of view of $A'$ the ray $ray[A',E']$  must intersect $s[B',D']$ at $F'$.
In that case $ray[r_A,r_E]$ intersects $s[r_B,r_D]$ at $ray[A,F']$.

The proof of the other remaining case (there is $r_F$ being strictly between $r_D$ and $r_B$ with the goal of finding $r_E$ so that $ray[r_A,r_F]$ intersects $s[r_C,r_D]$) is similar.

The final item to show is that $a(r_3)\in s[a(r_1),a(r_2)]$ if 
$r_3\in s(r_1,r_2)$. This is not obvious only if $r_3$ is strictly between $r_1$ and $r_2$ and follows by applying \ref{PaschBetweennessLemma} three times. 

Indeed, \ref{PaschBetweennessLemma} converts a triple of points of $\partial(\Pi,A)$
$$U, V, W$$
to the triple
$$V, W, a(U)$$
when the notation means the middle ray is strictly between the other two. Indeed, $W=a(U)$ is not possible as in that case the open line segment $s(U,W)$ is empty contrary to it containing $V$.

Applying that move three times gives
$$U, V, W$$
$$\Downarrow$$
$$V, W, a(U)$$
$$\Downarrow$$
$$W, a(U), a(V)$$
$$\Downarrow$$
$$a(U), a(V), a(W)$$
\end{proof}

\section{Lines}

In this section we describe lines intrinsically. The advantage of this approach is that when describing models of hyperbolic geometry one avoids confusing students. For example, it is customary to describe the Klein model (see \ref{Klein plane}) as follows (verbatim from \url{http://www.geom.uiuc.edu/docs/forum/hype/model.html}): \\
\emph{In the Klein model of the hyperbolic plane, the "plane" is the unit disk; in other words, the interior of the Euclidean unit circle. We call Euclidean points the "points" for our model. We call the portions of Euclidean lines which intersect the disk "lines."} \\
Perhaps this description is understandable to a seasoned mathematician but it is quite confusing to students.

Notice that most textbooks devote very little space to a discussion of lines. In our approach understanding lines is the most important issue to which we devote almost half of the whole paper. Once the concept of lines is understood and most importantly, the concept of length of line segments in lines, the rest of the material follows quite naturally.

\begin{Definition}\label{LineDef}
A \textbf{line} is a set $l$ with a relation of betweenness satisfying the following conditions:\\
a. $l$ consists of at least two points.\\
b. for every $A\in l$ the complement $l\setminus\{A\}$ of $A$ in $l$ can be expressed as the union of two disjoint non-empty convex sets $c_1$ and $c_2$ such that $A\in s[B,C]$ whenever $B\in c_1$ and $C\in c_2$.

If, in addition, every closed line segment in $l$ is connected, then $l$ is called a \textbf{connected line}.
\end{Definition}

\begin{Example}
The set of integers $\mathbb{Z}$ with the standard betweenness is a line but is not a connected line.
\end{Example}

\begin{Example}
Any one-dimensional real vector space $l$ is a connected line. Indeed, define betweenness in $l$ as follows: $C$ is between $A$ and $B$
if there is $t\ge 0$ such that $C-A=t(B-A)$ or $B=A$.. Notice that if $A\ne B$, then $s[A,B]=\{A+t(B-A)\mid t\in [0,1]\}$ and $l\setminus \{A\}$ has two components:
$r_1=\{A+t(B-A)\mid t > 0\}$
and $r_2=\{A+t(B-A)\mid t < 0\}$.
\end{Example}

\begin{Remark}
At a higher level one can see that any topological line has unique betweenness relation
which induces the same topology and makes it a connected line in the sense of \ref{LineDef}.
\end{Remark}

\begin{Example}
If $A$ and $B$ are two different points of a space with betweenness and $s[A,B]$ is connected, then the open segment $s(A,B)$ is a connected line.
\end{Example}

\begin{Lemma}\label{IntersectionOfSegmentsLine}
Let $l$ be a connected line.
The intersection of any two closed line segments with a common endpoint is a closed line segment. If the intersection is not a singleton, then one of the line segments is contained in the other.
\end{Lemma}
\begin{proof}
Suppose $s[A,C]\cap s[A,B]$ is not a closed line segment. Since it is a convex subset of
$s[A,C]$ and its complement is convex as well, it follows that $s[A,C]\cap s[A,B]=s(D,A]$
for some $D\in s[A,C]$ (we are using our definition of connectivity here). Similarly, $s[A,C]\cap s[A,B]=s(E,A]$ for some $E\in s[A,B]$.
Notice $E\ne D$ since $s[A,C]\cap s[A,B]$ does not contain $D$. 

Choose any $M\in s(E,D)$. $M$ cannot miss $s[A,B]\cup s[A,C]$ as $M\notin s[A,B]\cup s[A,C]$ implies $E$ and $D$ belong to the same component of $l\setminus \{M\}$.
Also, $M\in s[A,B]\cap s[A,C]$ cannot happen because
in that case one of points $D$, $E$ has to be on the same side of $M$ as $A$ contradicting one of equalities
$s[A,C]\cap s[A,B]=s(D,A]$, $s[A,C]\cap s[A,B]=s(E,A]$
in view of $s[A,M]\subset s[A,C]\cap s[A,B]$.

Thus, $M$ belongs to exactly one of the segments
$s[A,B]$, $s[A,C]$. Without loss of generality,
we may assume $M\in s[A,C]\setminus s[A,B]$. 
$M\notin s[A,D)$ means $M\in s(D,C]$ which creates a contradiction: $D$ is on the same side of $M$ as $A$
and $E$ is on the same side of $M$ as $A$, hence $D$ is on the same side of $M$ as $E$.
\end{proof}

\begin{Corollary}\label{AddingClosedSegments}
If $l$ is a connected line, then any finite union of closed line segments is a closed line segment if there is a common point to all of them.
\end{Corollary}
\begin{proof}
It suffices to consider the case of all line segments having a common endpoint
by replacing each line segment $s[A,B]$ by the union $s[A,I]\cup s[B,I]$, where $I$ is a common point to all line segments.

By \ref{IntersectionOfSegmentsLine} we can get rid of all line segments contained in other line segments and arrive either at only one line segment or two line segments intersecting only at the common endpoint, say $A$. In that case the other endpoints, say $B$ and $C$, are in different components of $l\setminus \{A\}$ which means $A$ is between $B$ and $C$ resulting in $s[A,B]\cup s[A,C]$ being $s[B,C]$.
\end{proof}

\begin{Lemma}\label{BetwennessCorollary}
Given three different points of a connected line $l$ there is exactly one pair such that the remaining point is between them.
\end{Lemma}
\begin{proof}
It follows from \ref{BetwennessOfTripleProp} as $s[A,B]\cup s[A,C]$ is a closed line segment by \ref{AddingClosedSegments}.
\end{proof}

\begin{Corollary}\label{LinesAsPaschSpaces}
Every connected line is a Pasch space such that each boundary at infinity consists of two points.
\end{Corollary}

\subsection{Rays in lines}

\begin{Observation}
Given two different points $A$ and $B$ of a line $l$, $ray[A,B]$ defined as the union of $\{A\}$
and the component of $l\setminus \{A\}$ containing $B$ is a maximal ray emanating from $A$ and containing $B$.
\end{Observation}

The next result provides another way of interpreting of Axiom 2 of Euclid \ref{Axiom 2 of Euclid}.

\begin{Corollary}\label{MaximalLineCor}
Let $\Pi$ be a set with betweenness such that each closed line segment is connected.
If $\Pi$ is a space with maximal rays and each connected closed line segment is contained in the interior of another connected closed line segment, then each non-trivial line segment $s[A,B]$ is contained in a maximal line denoted by $l(A,B)$.
\end{Corollary}
\begin{proof}
Pick $C\in s(A,B)$ and notice the maximal rays containing $s[C,A]$ and $s[C,B]$, respectively, add up to a line $l$. Every line containing $s[A,B]$  is a subset of $l$.
\end{proof}

\begin{Theorem}\label{LineCharThm2}
Let $l$ be a Pasch space such that every line segment is connected.
$l$ is a line if and only if its boundary at infinity at some point consists of two points.
\end{Theorem}
\begin{proof}
One direction follows from \ref{LinesAsPaschSpaces}.
If the boundary at infinity consists of two points at some point $A$, then those two rays must be antipodal and every point separates $l$. 
\end{proof}

\begin{Theorem}\label{LineCharThm}
Let $l$ be a Pasch space such that every line segment is connected.
$l$ is a line if and only if every three points of $l$ are contained in the interior of a closed line segment.
\end{Theorem}
\begin{proof}
By \ref{IntersectionOfSegmentsLine} every three points of a line are contained in the interior of a closed line segment. This takes care of one direction of the proof.

To show the other direction notice $l$ is a space with maximal rays,
so pick $A\ne B$ in $l$ and consider the maximal line $l(A,B)$ containing
$s[A,B]$ using \ref{MaximalLineCor}. Notice $l(A,B)$ must be equal to $l$ by \ref{MaximalRaysProp}.
\end{proof}

\subsection{Orienting rays and lines}

Each ray $r$ has two orders inducing its relation of betweenness: one making the initial point $O$ of $r$ its minimum and the other one making $O$ the maximum of $r$.
The first order is defined by $A\leq B$ if and only $s[O,A]\subset s[O,B]$ and the second order is defined by $A\leq B$ if and only $s[O,B]\subset s[O,A]$. We will give preference to the first order and we will call it the \textbf{positive order} on $r$ in analogy to the order on the positive reals.

Orienting a line $l$ amounts to choosing a maximal ray $r$ and declaring every other ray to be either positive (relative to $r$) or negative. A maximal ray $r'$ is \textbf{positive} if 
$r\cap r'$ is a maximal ray emanating from some point. Otherwise, $r'$ is \textbf{negative}. Equivalently, $r'$ is positive relative $r$ if their positive orders agree on the overlap $r\cap r'$ of rays.

Notice one gets two possible ways of orienting the line $l$ and each orientation induces a linear order on $l$ such that $s[A,B]=\{X\in l \mid A\leq X\leq B\}$ or $s[A,B]=\{X\in l \mid B\leq X\leq A\}$. Namely, $A\leq B$ if and only if $A=B$ or $A\ne B$ and $ray[A,B]$ is a positive maximal ray.

\subsection{Dedekind cuts}

Dedekind cuts can be viewed as a method of creating connected lines out of lines that are not connected.

The set of integers $\mathbb{Z}$ with the standard betweenness is a line such that any closed line segment $s[m,n]$ containing at least two points is disconnected in a strong sense: it is the union of two closed line subsegments that are nonempty and disjoint.

From $\mathbb{Z}$ we move to the set of rational numbers $\mathbb{Q}$ which is connected to some degree. Namely, no closed line segment in $\mathbb{Q}$ 
is the union of two closed line subsegments that are nonempty and disjoint.

To create a connected line out of $\mathbb{Q}$ one uses Dedekind cuts \cite{Dede}. We are going to be slightly more general than that.

\begin{Definition}\label{DedekindCut}
Let $l$ be a line such that no closed line segment is the union of two closed line subsegments that are nonempty and disjoint.
Given a maximal ray $r$ in $l$, a \textbf{Dedekind cut} is a pair $(\Sigma_1,\Sigma_2)$
of non-empty, disjoint and convex subsets of $l$ whose union is $l$, $\Sigma_1$ is not a maximal ray at any point but it contains a negative maximal ray.
\end{Definition}

\begin{Example}
For every rational number $q$ the pair $(\Sigma_1,\Sigma_2)$,
where $\Sigma_1=\{x\in \mathbb{Q}\mid x < q\}$ and $\Sigma_2=\{x\in \mathbb{Q}\mid x \ge q\}$, is a Dedekind cut.
\end{Example}

\begin{Example}
$(\Sigma_1,\Sigma_2)$,
where $\Sigma_1=\{x\in \mathbb{Q}\mid x\leq 0 \mbox{ or } x^2 < 2\}$ and $\Sigma_2=\{x\in \mathbb{Q}\mid x^2 > 2 \mbox{ and } x > 0\}$, is a Dedekind cut in which none of the sets is a maximal ray.
\end{Example}

\begin{Proposition}\label{CreatingConnectedLines}
Let $l$ be a line such that no closed line segment is the union of two closed line subsegments that are nonempty and disjoint.
The set of Dedekind cuts is a connected line if the betweenness is induced by the linear order defined as follows: $(\Sigma_1,\Sigma_2)\leq (\Sigma'_1,\Sigma'_2)$ if and only if $\Sigma_1\subset \Sigma'_1$
\end{Proposition}
\begin{proof}
The only part requiring a non-trivial proof is the connectedness of line segments in the space of Dedekind cuts. Given a convex set $\Sigma$ in $s[(\Sigma_1,\Sigma_2),(\Sigma'_1,\Sigma'_2)]$ containing $(\Sigma_1,\Sigma_2)$ let's add all the first coordinates of elements of $\Sigma$ and obtain $A$. If it is not a maximal ray, then
the point $(A,l\setminus A)$ is a Dedekind cut and the complement of 
$\Sigma$ in $s[(\Sigma_1,\Sigma_2),(\Sigma'_1,\Sigma'_2)]$ is a closed line segment.
If $A$ is a maximal ray, then $\Sigma$ is a closed line segment in $s[(\Sigma_1,\Sigma_2),(\Sigma'_1,\Sigma'_2)]$.
\end{proof}

Notice the connected line constructed in \ref{CreatingConnectedLines} contains an isomorphic copy of $l$. In the particular case of $l=\mathbb{Q}$, the constructed connected line is the set of reals $\mathbb{R}$.

\subsection{Bijections between lines}

\begin{Proposition}\label{BijectionsBetweenSpacesMaxRays}
If $f:\Pi_1\to\Pi_2$ is a bijection between spaces with maximal rays that preserves connected line segments, then $f(l)$ is a connected line in $\Pi_2$ for every connected line $l$ in $\Pi_1$.
\end{Proposition}
\begin{proof}
$f(l)$ is a convex subset of $\Pi_2$.
In view of \ref{BijectionPreservingBetweennessProp} it suffices to show 
$f^{-1}:f(l)\to l$ preserves betweenness. Suppose $f(C)$ is between $f(A)$ and $f(B)$ in $f(l)$ but $C$ is not between $A$ and $B$. We may assume $A$ is between
$C$ and $B$. Therefore $f(A)$ is between $f(C)$ and $f(B)$, a contradiction to \ref{BetwennessCorollary}.
\end{proof}

\begin{Proposition}
Let $l_1,l_2$ be two connected lines. If $f:l_1\to l_2$ is a bijection preserving betweenness, then $f$ preserves closed line segments.
\end{Proposition}
\begin{proof}
Suppose $C\in s[f(A),f(B)]\setminus f(s[A,B])$ for some points $A,B$ of $l_1$.
There is $D\in l_1$ with $f(D)=C$. As $D\notin s[A,B]$, we may assume $A$ is between
$D$ and $B$. Therefore $f(A)$ is between $C$ and $f(B)$, a contradiction to \ref{BetwennessCorollary}.
\end{proof}

\begin{Lemma}\label{TranslationLemma}
Let $l$ be a connected line and let $\tau:l\to l$ be a bijection preserving 
betweenness.
If $\tau$ has no fixed points, then for each $A\in l$ $ray[A,\tau(A)]$
is sent by $\tau$ to a proper subset of itself. In particular, $\tau(A)$ is between
$A$ and $\tau^2(A)$.
\end{Lemma}
\begin{proof}
$\tau(ray[A,\tau(A)])$ does not contain $A$. Indeed, suppose $A=\tau(B)$
and $B\in ray[A,\tau(A)]$. $B$ cannot be between $A$ and $\tau(A)$, as in that case
$\tau^{-1}$ sends $s[A,\tau(A)]$ to itself and $\tau^{-1}$ has a fixed point.
Therefore $\tau(A)=\tau^2(B)$ is between $A$ and $B$ in which case
$\tau$ sends $s[B,\tau(B)]$ to itself resulting in a fixed point of $\tau$, a contradiction.

Since $\tau(ray[A,\tau(A)])$ is a convex subset of $l$ missing $A$ and containing $\tau(A)$,
it is contained in $ray[A,\tau(A)]$.
\end{proof}

\begin{Corollary}\label{PeriodicityLemma}
Let $l$ be a connected line and let $\tau:l\to l$ be a bijection preserving 
betweenness.
If $\tau$ has no fixed points, then $A\ne \tau^n(A)$ for any point $A$ of the line $l$
and any integer $n\ne 0$.
\end{Corollary}
\begin{proof}
Apply \ref{TranslationLemma} to see that $\tau^n(ray[A,\tau(A)])$ misses $A$.
\end{proof}

\subsection{Final comments}

The reason we chose \ref{LineDef} as a definition of lines is because it is generalizable to circles and planes. However, \ref{LineDef} works best for connected lines as illustrated by the following:

\begin{Example}
Consider three disjoint rays emanating from points $A$, $B$, and $C$. Define betweenness by extending betweenness of each ray to require that $s[X,Y]$, for $X$ and $Y$ belonging to different rays, is defined
as the union of $s[X,O_X]\cup s[Y,O_Y]$, where $O_X$ is the initial point of the ray containing $X$ and $O_Y$ is the initial point of the ray containing $Y$. Notice we get a line in the sense of \ref{LineDef} which is counterintuitive. 
\end{Example}

If one wants a definition of lines that works correctly for non-connected lines as well, then the condition from \ref{LineCharThm} seems the most useful. Namely, the requirement that every three points are contained in the interior of a closed line segment.

\section{Spaces with rigid motions}\label{Spaces with rigid motions}

In this section we introduce a general framework for discussing congruence and length of line segments. Our approach is motivated by ideas from physics.

\begin{Definition}\label{IsotropicSpaceDef}
An \textbf{isotropic space} is a pair $(\Pi,\mathcal{I})$ consisting of a space $\Pi$ with maximal rays and a subgroup $\mathcal{I}$ of the group of isomorphisms of $\Pi$ satisfying the following condition:\\
For any two maximal rays $r_1$ and $r_2$ in $\Pi$ there is $f\in\mathcal{I}$ such that $f(r_1)= r_2$.
\end{Definition}

The condition above means, in the language of physics, that $\Pi$ is the same in all directions.

Our next task is to define rigidity without using the concept of length.

\begin{Definition}\label{SpaceWithRigidMotionsDef}
A \textbf{space with rigid motions} is an isotropic space $(\Pi,\mathcal{I})$ such that each $f\in\mathcal{I}$ is a \textbf{rigid motion}, i.e. if $f(s[A,B])\subset s[A,B]$ for some $f\in \mathcal{I}$ and some line segment $s[A,B]$, then $f(s[A,B])= s[A,B]$.
\end{Definition}

The condition above mean that each $f$ in $\mathcal{I}$ is \textbf{rigid}, i.e. it cannot contract or expand line segments.
Indeed, if $f(s[A,B])\supset s[A,B]$, then $f^{-1}(s[A,B])\subset s[A,B]$ and $f^{-1}(s[A,B])= s[A,B]$
resulting in $f(s[A,B])= s[A,B]$.

The following is the most interesting case of spaces with rigid motions.
\begin{Proposition}\label{NilpotentCaseOfRM}
Let $(\Pi,\mathcal{I})$ be an isotropic space.
 If $f(s[A,B])\subset s[A,B]$ for some $f\in \mathcal{I}$ and some line segment $s[A,B]$ implies that $f$ restricted to $s[A,B]$ is of finite order (i.e. there is a natural $n$ satisfying $f^n=id$ on $s[A,B]$), then $(\Pi,\mathcal{I})$ is a space with rigid motions.
\end{Proposition}
\begin{proof}
Suppose $f(s[A,B])\subset s[A,B]$ for some $f\in \mathcal{I}$ and some line segment $s[A,B]$. Let $n$ be a natural number satisfying $f^n=id$ on $s[A,B]$.
By applying $f$ to $f(s[A,B])\subset s[A,B]$ repeatedly one gets
$f^k(s[A,B])\subset f(s[A,B])$ for all $k\ge 1$. In particular, for $k=n$,
one gets $s[A,B]=f^n(s[A,B])\subset f(s[A,B])$ resulting in $f(s[A,B])= s[A,B]$.
\end{proof}

\begin{Proposition}\label{RigidMotionsViaRaysProp}
Let $(\Pi,\mathcal{I})$ be an isotropic space.
$(\Pi,\mathcal{I})$ is a space with rigid motions if and only if $f(r)=r$ for some maximal ray $r$ implies $f|r$ is the identity.
\end{Proposition}
\begin{proof}
If $(\Pi,\mathcal{I})$ is a space with rigid motions and $f(r)=r$ for some maximal ray $r$ emanating from $A$, then for every point $B\in r$
we have either $f(s[A,B])\subset s[A,B]$ or $f^{-1}(s[A,B])\subset s[A,B]$.
Hence $f(s[A,B])=s[A,B]$ or $f^{-1}(s[A,B])=s[A,B]$ resulting in
$f(B)=B$.

Suppose $f(r)=r$ for some maximal ray $r$ implies $f|r$ is the identity
and $f(s[A,B])\subset s[A,B]$ for some points $A\ne B$. By \ref{FixedPointTheorem} there is a fixed point $C\in s[A,B]$ of $f$.
Consider the two antipodal rays $r$ and $a(r)$ emanating from $C$ so that
$s[A,B]\subset r\cup a(r)$. If $f$ preserves them, then $f$ is the identity on $s[A,B]$. If $f$ reverses them, then $f^2$ preserves them and $f^2$ is the identity on $s[A,B]$. By \ref{NilpotentCaseOfRM}, $(\Pi,\mathcal{I})$ is a space with rigid motions.
\end{proof}

\subsection{Congruence}
\begin{Definition}
Let $(\Pi,\mathcal{I})$ be a space with rigid motions.
Two line segments $s[A,B]$ and $s[C,D]$ of $\Pi$ are \textbf{congruent} if there is $\tau\in \mathcal{I}$ such that
$$\tau(s[A,B])= s[C,D].$$
\end{Definition}

Notice congruence is an equivalence relation and if
$s[C,D]$ is a proper subset of $s[A,B]$, then the two line segments are not congruent.

\subsection{Length of line segments}

In this part we introduce the concept of length of line segments in spaces with rigid motions. Our approach is very similar to that of Euclid: we do not consider length as a real-valued function. Instead, it is an equivalence relation with addition, multiplication, and proportions. Also, we have subtraction of a smaller length from a larger length.

\begin{Definition}
Let $(\Pi,\mathcal{I})$ be a space with rigid motions.
Given two line segments $s[A,B]$ and $s[C,D]$ of an Erlangen line $(l,\mathcal{I})$ we say the \textbf{length} of 
$s[A,B]$ is \textbf{at most the length} of $s[C,D]$ (notation: $|AB|\leq |CD|$) if there is $\tau\in \mathcal{I}$ such that
$\tau(s[A,B])\subset s[C,D]$. If, in addition, $s[A,B]$ is not congruent to $s[C,D]$, then we say say the \textbf{length} of 
$s[A,B]$ is \textbf{smaller than the length} of $s[C,D]$ (notation: $|AB| < |CD|$).
\end{Definition}

\begin{Lemma}\label{LengthInequalityTransitivityLine}
Let $(\Pi,\mathcal{I})$ be a space with rigid motions.
If $|AB|\leq |CD|$ and $|CD\leq |EF|$, then $|AB|\leq |EF|$.
\end{Lemma}
\begin{proof}
Pick $\tau_1, \tau_2\in \mathcal{I}$ so that $\tau_1(s[A,B])\subset s[C,D]$
and $\tau_2(s[C,D])\subset s[E,F]$. Put $\tau=\tau_2\circ \tau_1$ and observe
$\tau(s[A,B])\subset s[E,F]$.
\end{proof}

The following lemma implies congruence of line segments is the same as equality of their lengths.
\begin{Lemma}
Let $(\Pi,\mathcal{I})$ be a space with rigid motions.
If $|AB|\leq |CD|$ and $|CD|\leq |AB|$, then the line segments $s[A,B]$ and $s[C,D]$ are congruent.
\end{Lemma}
\begin{proof}
Pick isomorphisms $\tau, \rho$ in $\mathcal{I}$ such that
$\tau(s[A,B])\subset s[C,D]$ and $\rho(s[C,D])\subset s[A,B]$. Let
$\phi=\rho\circ \tau$. Notice $\phi(s[A,B])\subset s[A,B]$, hence $s[A,B]=\phi(s[A,B])$. Similarly, for $\psi=\tau\circ \rho$, one has
$s[C,D]=\psi(s[C,D])$ resulting in $\tau(s[A,B])= s[C,D]$.
\end{proof}

\begin{Lemma}
Let $(\Pi,\mathcal{I})$ be a space with rigid motions.
Given two connected line segments $s[A,B]$ and $s[C,D]$ of $\Pi$ one has
$|AB|\leq |CD|$ or $|CD|\leq |AB|$.
\end{Lemma}
\begin{proof}
It is clearly so if $A=B$ or $C=D$, so assume $A\ne B$ and $C\ne D$.
Pick an isomorphism $\rho$ sending the ray $ray[A,B]$ onto the ray $ray[C,D]$.
Either the image of $s[A,B]$ is contained in $s[C,D]$ and $|AB|\leq |CD|$
or the image of $s[A,B]$ contains $s[C,D]$ and $|AB|\ge |CD|$.
\end{proof}

\subsection{Algebra of lengths of line segments}

\begin{Definition}\label{SumOfLengthsDef}
Let $(\Pi,\mathcal{I})$ be a space with rigid motions. The length of a line segment $s[A,B]$
is equal to the \textbf{sum of lengths} $|CD|$ and $|EF|$ if there is a point $M$ in $s[A,B]$
such that $|AM|=|CD|$ and $|MB|=|EF|$.
\end{Definition}

\begin{Lemma}
Let $(\Pi,\mathcal{I})$ be a space with rigid motions.
If the length of a connected line segment $s[A,B]$
is equal to the sum of lengths $|CD|$ and $|EF|$ and the length of a connected line segment $s[A',B']$
is equal to the sum of lengths $|EF|$ and $|CD|$, then $|AB|=|A'B'|$.
\end{Lemma}
\begin{proof}
Pick $M\in s[A,B]$ and $M'\in s[A',B']$ satisfying
$|AM|=|B'M'|=|CD|$ and $|BM|=|A'M'|=|EF|$. Choose the isomorphism
$\tau$ sending $ray[A,B]$ onto $ray[B',A']$. Notice $\tau(M)=M'$ as
$|B'M'|=|B'\tau(M)|$. By the same reason $\tau(B)=A'$ resulting in $|AB|=|A'B'|$.
\end{proof}

\begin{Lemma}\label{AddingLengthsLemma}
Let $(\Pi,\mathcal{I})$ be a space with rigid motions.
If $|AB|\leq |CD|$, $|A'B'|\leq |C'D'|$, and $|AB|+|A'B'|$ exists for connected line segments, then $|AB|+|A'B'|\leq |CD|+|C'D'|$.
If, in addition, $|AB|+|A'B'|=|CD|+|C'D'|$, then $|AB|= |CD|$ and $|A'B'|= |C'D'|$
\end{Lemma}
\begin{proof}
Similar to the one above.
\end{proof}

\begin{Definition}\label{DifferenceOfLengthsDef}
Let $(\Pi,\mathcal{I})$ be a space with rigid motions. The length of a line segment $s[A,B]$
is equal to the \textbf{difference of lengths} $|CD|$ and $|EF|$ if $|AB|+|EF|=|CD|$.
\end{Definition}

\begin{Lemma}\label{ExistenceUniquenessOfDifference}
Let $(\Pi,\mathcal{I})$ be a space with rigid motions.
For every two connected lengths $|AB|\ge |CD|$ there is a line segment whose length equals the difference of lengths and its length is unique.
\end{Lemma}
\begin{proof}
Send $ray[C,D]$ by $\tau\in \mathcal{I}$ onto $ray[A,B]$.
$|\tau(D)B|$ is the desired difference of lengths.
\end{proof}

\subsection{Dividing of lengths of line segments}

\begin{Definition}
Let $(\Pi,\mathcal{I})$ be a space with rigid motions.
Given two line segments $s[A,B]$ and $s[C,D]$ of an Erlangen line we say the \textbf{length} of 
$s[A,B]$ is \textbf{twice the length} of $s[C,D]$ (notation: $|AB|=2\cdot |CD|$) if
$s[A,B]$ contains a point $M$ such that $|AM|=|MB|=|CD|$.
\end{Definition}

\begin{Lemma}\label{MidpointLemma}
Let $(\Pi,\mathcal{I})$ be a space with rigid motions. For every connected line segment $s[A,B]$ there
is exactly one midpoint of $s[A,B]$, i.e. a point $M\in s[A,B]$ satisfying
$|AM|=|MB|$.
\end{Lemma}
\begin{proof}
Obvious if $A=B$. First, let's concentrate on the existence of $M$.  
Pick an isomorphism $\rho$ sending $ray[A,B]$ onto $ray[B,A]$. It sends $s[A,B]$ onto itself, so it has a fixed point $M$ which must be in $s(A,B)$.

Uniqueness of $M$ follows from \ref{AddingLengthsLemma}.
\end{proof}

\begin{Definition}
Let $(\Pi,\mathcal{I})$ be a space with rigid motions.
Let $n > 0$ be an integer.
Given two line segments $s[A,B]$ and $s[C,D]$ of an Erlangen line we say the \textbf{length} of 
$s[A,B]$ is \textbf{$n$ times the length} of $s[C,D]$ (notation: $|AB|=n\cdot |CD|$) if
$s[A,B]$ contains points $C_0,\ldots, C_n$ such that $C_0=A$, $C_n=B$,
$|C_iC_{i+1}|=|CD|$ for all $0\leq i\leq n-1$, and the line segments $s[C_i,C_{i+1}]$,
$s[C_j,C_{j+1}]$ have at most one common point if $i\ne j$.
\end{Definition}

\begin{Lemma}\label{DividingLemma}
Let $(\Pi,\mathcal{I})$ be a space with rigid motions and let $n > 1$ be a natural number.
For each connected length $L$ of line segments there is a unique length $L_n$ of line segments
such that $n\cdot L_n=L$.
\end{Lemma}
\begin{proof}
Uniqueness of $L_n$ follows from \ref{AddingLengthsLemma}. Existence of $L_n$ for $n$ being a power of two follows from \ref{MidpointLemma}.
Given $A\ne B$ consider two subsets of $s[A,B]$:
$S_1=\{X\in s[A,B]\mid n\cdot |AX| < |AB|\}$ and 
$S_2=\{X\in s[A,B]\mid n\cdot |AX| > |AB|\}$. Both are convex, non-empty and neither can be a closed line segment (see below). Therefore their union is not the whole $s[A,B]$ and the point $X$ outside of that union satisfies $n\cdot |AX|=|AB|$.

$S_1$ is not empty because we can pick an integer $k$ with $m=2^k > n$.
Now, the point $X\in s[A,B]$ satisfying $m\cdot |AX|=|AB|$ belongs to $S_1$.
For the same reason $S_1$ is not a closed interval: if $n\cdot |AX| < |AB|$
we can find $Y\in s[X,B]$ satisfying $m\cdot |XY|=|XB|$ and observe
$n\cdot |AY| < |AB|$. The last observation hinges on the distributivity of
multiplication with respect to addition: $n\cdot(|XY|+|YZ|)=n\cdot |XY|+n\cdot |YZ|$ if $Y\in s[X,Z]$, which is easy to prove.
\end{proof}

\section{Erlangen lines}

\begin{Definition}\label{ErlangenLineDef}
An \textbf{Erlangen line} is a pair $(l,\mathcal{I})$ 
 satisfying the following conditions:\\
a. $l$ is a connected line,\\
b. $\mathcal{I}$ is a subgroup of the group of isomorphisms of $l$ satisfying Condition \ref{Homogeneity and rigidity of the line}.
\end{Definition}

\begin{Condition}[Homogeneity and rigidity of the Erlangen line]\label{Homogeneity and rigidity of the line}
For every ordered pair of maximal rays in $l$ there is exactly one isomorphism in $\mathcal{I}$ sending the first ray onto the other.
\end{Condition}

Condition \ref{Homogeneity and rigidity of the line} has two interpretations: it corresponds to the basic tenet of physics that independent observers should arrive at the same measurement. Namely, the group $\mathcal{I}$ is used to make measurements. From the point of view of geometry Condition
\ref{Homogeneity and rigidity of the line} means uniqueness of congruence for maximal rays.

\begin{Example}[The real line]\label{The real line}
Let $l$ be the set of reals with the standard relation of betweenness and let $\mathcal{I}$ be the set of all functions of the form $f(x)=m\cdot x+b$, where $m=\pm 1$.

The pair $(l,\mathcal{I})$ is an Erlangen line.
\end{Example}

Our next three examples of Erlangen lines will help us understand three models of hyperbolic geometry later on.

\begin{Example}[The hyperbolic line I]\label{The hyperbolic line I}
Let $l=(0,\infty)$ with the standard relation of betweenness and let $\mathcal{I}$ be the set of all functions of the form $f(x)=c\cdot x^k$, where $c > 0$ and $k=\pm 1$.
\end{Example}

\begin{Example}[The hyperbolic line II]\label{The hyperbolic line II}
Let $l=(-1,1)$ with the standard relation of betweenness and let $\mathcal{I}$ be the set of all bijections $f: (-1,1) \rightarrow(-1,1)$ that can be expressed as
 $f(x)=\frac{a\cdot x+b}{c\cdot x+d}$ for some real numbers $a$, $b$, $c$, and $d$. Equivalently, $f$ is the restriction of a rational function such that $f(0)\in (-1,1)$ and either $f(1)=1$ and $f(-1)=-1$ or $f(1)=-1$ and $f(-1)=1$.

Indeed, it suffices to show two facts since $\mathcal{I}$ is obviously a subgroup of bijections of $(-1,1)$:\\
1. for every $w\in l$ there is a unique $f\in\mathcal{I}$ such that $f(0)=w$ and $f(1)=1$.\\
2. for every $w\in l$ there is a unique $f\in\mathcal{I}$ such that $f(0)=w$ and $f(1)=-1$.\\

Fact 2. follows from 1.: given $f, g\in \mathcal{I}$ satisfying
$f(0)=w$, $f(1)=-1$, $g(0)=w$, and $g(1)=-1$, the functions $x\to f(-x)$ and $x\to g(-x)$
fix $1$ and send $0$ to $w$, hence must be equal.

To show Fact 1. observe that $f(x)=\frac{a\cdot x+b}{c\cdot x+d}$ must satisfy
$a+b=c+d$, $-a+b=-(-c+d)$, and $w=b/d$. Therefore $b=wd$, $b=c$, and $a=d$
resulting in $f(x)=\frac{d\cdot x+wd}{wd\cdot x+d}= \frac{x+w}{w\cdot x+1}$.
\end{Example}

\begin{Example}[The hyperbolic line III]\label{The hyperbolic line III}
Let $l$ be the the upper part of the hyperbola $\{(t,x)\mid t^2-x^2=1\}$ (i.e. $t > 0$) and let $\mathcal{I}$ be the orthochronous group $O^+(1,1)$
of Lorentz transformations. It is the group of linear transformations of the Descartes plane $\mathbb{R}^2$ preserving the quadratic form $t^2-x^2$ and preserving the orientation of $t$.

The betweenness of $l$ determined by the order on $x$. 

The element of $\mathcal{I}$ are related to the well-known \textbf{boosts} in the $x$-direction from special relativity. Such boosts are given by the formulae (we are using a unit system in which the speed of light is $1$):
$$t'=\lambda(t-v\cdot x), x'=\lambda(x-v\cdot t)$$
where $(t',x')=h(t,x)$, $v$ is a real constant of absolute value less than $1$,
and $\lambda=\frac{1}{\sqrt{1-v^2}}$.

Analyzing linear transformations $h(t,x)=(a\cdot t+ b\cdot x, c\cdot x+d\cdot t)$
preserving the quadratic form $t^2-x^2$ and preserving the orientation of $t$
leads to $a > 0$, $c=\pm a$, $d=\pm b$, and $a^2-b^2=1$. Thus, $\mathcal{I}$
consists of linear transformations
$$ h(t,x)=(a\cdot t+ b\cdot x, \pm (a\cdot x+ b\cdot t))$$
where $a > 0$ and $a^2-b^2=1$.

To show $(l,\mathcal{I})$ is an Erlangen line it suffices to prove the following two facts:\\
1. for every $w\in l$ there is a unique $f\in\mathcal{I}$ such that $f(1,0)=w$ and $f$ preserves the orientation of $l$.\\
2. for every $w\in l$ there is a unique $f\in\mathcal{I}$ such that $f(1,0)=w$ and $f$ reverses the orientation of $l$.

If $w=(a,b)$, then $f$ in 1. must be $f(t,x)=(a\cdot t+ b\cdot x, a\cdot x+b\cdot t)$
and $f$ in 2. must be $f(t,x)=(a\cdot t- b\cdot x, -a\cdot x+b\cdot t)$.

Notice that boosts from special relativity correspond to our translations (see \ref{TranslationDefLine} later on).

\end{Example}

\begin{Proposition}\label{InvolutionAndIdLines}
Let $(l,\mathcal{I})$ be an Erlangen line.
An isomorphism $\tau$ in $\mathcal{I}$ that has a fixed point is an \textbf{involution}, i.e. $\tau^2$ is the identity. 
An isomorphism in $\mathcal{I}$ that has two fixed points is the identity. 
\end{Proposition}
\begin{proof}
If $\tau$ has a fixed point $A$, then it either preserves the rays emanating from $A$ (in which case $\tau=id$ as $id$ also preserves both rays emanating from $A$) or it permutes the rays. In the latter case $\tau^2$ preserves both rays emanating from $A$ and $\tau^2=id$.

If $\tau$ has two fixed points $A$ and $B$, then $\tau$ preserves $ray[A,B]$, so $\tau=id$.
\end{proof}

\begin{Corollary}\label{ErlangenLinesViaRigidMotionsCor}
a. Every Erlangen line $(l,\mathcal{I})$ is a space with rigid motions.\\
b. Every space with rigid motions $(l,\mathcal{I})$ is an Erlangen line if $l$ is a connected line.
\end{Corollary}
\begin{proof}
a. Suppose $f(s[A,B])\subset s[A,B]$ for some $f\in \mathcal{I}$. By \ref{FixedPointTheorem} $f$ has a fixed point and by \ref{InvolutionAndIdLines}
$f^2=id$. By \ref{NilpotentCaseOfRM} $(l,\mathcal{I})$ is a space with rigid motions.\\
b. It suffices to show $f=id$ if $f(r)= r$ for some $f\in \mathcal{I}$ and some 
maximal ray $r$. Let $A$ be the initial point of $r$ and let $B\in r$.
Either $f(s[A,B])\subset s[A,B]$ or $s[A,B]\subset f(s[A,B])$. In the former case $f(s[A,B])= s[A,B]$ which implies $B=f(B)$ and in the latter case
$f^{-1}(s[A,B])\subset s[A,B]$ which implies $B=f(B)$ as well.
\end{proof}

\begin{Definition}\label{SymmetryDefLine}
Given an Erlangen line $(l,\mathcal{I})$ and a point $A$ of $l$, the \textbf{reflection} $i_A$ in $A$
is the isomorphism in $\mathcal{I}$ sending each maximal ray emanating from $A$ to the other maximal ray emanating from $A$.
\end{Definition}

\begin{Definition}\label{TranslationDefLine}
Given an Erlangen line $l$ and two different points $A, B$ of $l$, the \textbf{translation} 
$\tau_{AB}$ from $A$ to $B$ is the isomorphism in $\mathcal{I}$ sending the ray $ray[A,B]$ to the maximal ray emanating from $B$ that does not contain $A$.
If $A=B$, we define $\tau_{AB}$ as the identity function.
\end{Definition}

\begin{Proposition}\label{CharOfSymmetriesLine}
Let $(l,\mathcal{I})$ be an Erlangen line and let $\tau$ be a non-trivial isomorphism in $\mathcal{I}$. The following conditions are equivalent:\\
a. $\tau$ is the reflection in a point,\\
b. $\tau$ has a fixed point,\\
c. for every maximal ray $r$ of $l$ neither $r$ is contained in $\tau(r)$ nor
$\tau(r)$ is contained in $r$,\\
d. there is a maximal ray $r_0$ of $l$ such that neither $r_0$ is contained in $\tau(r_0)$ nor
$\tau(r_0)$ is contained in $r_0$.
\end{Proposition}
\begin{proof}
a.$\implies$b. is obvious.\\
b.$\implies$c. By \ref{InvolutionAndIdLines} $\tau$ is an involution.
If $r\subset \tau(r)$ or $\tau(r)\subset r$ for some ray $r$, then applying $\tau$ to
those inclusions one gets $\tau(r)\subset r$ or $r\subset \tau(r)$ resulting in
$r=\tau(r)$. Therefore $\tau=id$, a contradiction.\\
c.$\implies$d. is obvious.\\
d.$\implies$a. Let $A$ be the initial point of $r_0$. Suppose $\tau$ has no fixed points. By \ref{TranslationLemma} for each $A\in l$ the ray $ray[A,\tau(A)]$
is sent by $\tau$ to a proper subset of itself. In particular, $\tau(A)$ is between
$A$ and $\tau^2(A)$. Therefore $A$ is between $\tau^{-1}(A)$ and $\tau(A)$ (by applying $\tau^{-1}$) and $r_0$ has to be equal to $ray[A,\tau^{-1}(A)]$.
Applying \ref{TranslationLemma} to $\tau^{-1}$ we see $\tau^{-1}(r_0)\subset r_0$.
Hence $r_0\subset \tau(r_0)$, a contradiction.

Thus $\tau$ has a fixed point, say $B$, and $\tau$ must be the reflection in $B$
as $\tau\ne id$.
\end{proof}

\begin{Proposition}\label{CharOfTanslationsLine}
Let $(l,\mathcal{I})$ be an Erlangen line and let $\tau$ be a non-trivial isomorphism in $\mathcal{I}$. The following conditions are equivalent:\\
a. $\tau$ has no fixed points,\\
b. $\tau$ is the translation from $A$ to $\tau(A)$ for each $A$ in $l$,\\
c. $\tau$ is a translation,\\
d. for every maximal ray $r$ of $l$ either $r$ is contained in $\tau(r)$ or
$\tau(r)$ is contained in $r$,\\
e. there is a maximal ray $r_0$ of $l$ such that either $r_0$ is contained in $\tau(r_0)$ or
$\tau(r_0)$ is contained in $r_0$.
\end{Proposition}
\begin{proof}
a., d., and e. are equivalent by \ref{CharOfSymmetriesLine}.\\
a.$\implies$b. follows from
\ref{TranslationLemma} which says that for each $A\in l$ $ray[A,\tau(A)]$
is sent by $\tau$ to a proper subset of itself. \\
b.$\implies$c. is obvious.\\
c.$\implies$a. Suppose $\tau=\tau_{AB}$ and $\tau$ has a fixed point $C$.
Since $\tau$ is an involution by \ref{InvolutionAndIdLines}, $\tau(B)=A$
contradicting the fact $\tau$ sends $ray[A,B]$ to the ray at $B$ not containing $A$.

\end{proof}

\begin{Corollary}\label{TranslationsOfLineSubgroup}
Let $(l,\mathcal{I})$ be an Erlangen line.
Translations form a subgroup of $\mathcal{I}$.
\end{Corollary}
\begin{proof}
Consider two non-trivial translations $\tau_1$ and $\tau_2$. We need to show $\tau_1\circ\tau_2$ is a translation. 
For $\tau_1\circ\tau_2$ not to be a translation it is necessary that it has a fixed point $A$. Let $B=\tau_2(A)$. Notice $\tau_1(B)=A$. By \ref{CharOfTanslationsLine},
$\tau_1=\tau_{BA}$ and $\tau_2=\tau_{AB}$. Therefore $\tau_1\circ\tau_2$ is the identity and is indeed a translation contrary to our assumption.
\end{proof}

\section{Length of line segments in Erlangen lines}\label{LengthOfLineSegmentsSec}

In this part we expand on Section \ref{Spaces with rigid motions} and analyze the concept of length of line segments in Erlangen lines. From a big picture point of view, lengths of non-trivial line segments form a line and lengths of line segments form a ray.

\begin{Lemma}
Let $(l,\mathcal{I})$ be an Erlangen line.
For every two lengths there is a line segment whose length equals the sum of lengths.
\end{Lemma}
\begin{proof}
Given four points $A, B, C, D$ of $l$ apply the translation $\tau_{CA}$. If $E:=\tau_{CA}(D)$ is on the other side of $A$ than $B$, then $|BE|$ equals the sum of $|AB|$ and $|CD|$.
Otherwise apply the reflection $i_A$ to $E$ obtaining $F$ with
$|BF|$ equal to the sum of $|AB|$ and $|CD|$.
\end{proof}

\begin{Proposition}\label{TranslationSizeProp}
Let $(l,\mathcal{I})$ be an Erlangen line. For any translation $\tau$ and any two points $A$ and $B$ of $l$ the lengths $|A\tau(A)|$ and $|B\tau(B)|$ are equal.
\end{Proposition}
\begin{proof}
It is clearly so if $\tau=id$ or $A=B$, hence consider the case of $\tau\ne id$ and $A\ne B$.
If $s[A,\tau(A)]$ and $s[B,\tau(B)]$ do not overlap on any interior point, we may assume
$s[B,\tau(B)]\subset ray[A,\tau(A)]$.  Since $|AB|=|\tau(A)\tau(B)|$
and $|A\tau(A)|=|AB|-|\tau(A)B|$, $|B\tau(B)|=|\tau(A)\tau(B)|-|\tau(A)B|$, we get
$|A\tau(A)|=|B\tau(B)|$ by
\ref{ExistenceUniquenessOfDifference}.

Suppose $s[A,\tau(A)]$ and $s[B,\tau(B)]$ do overlap on an interior point. Now, we may assume
$B\in s[A,\tau(A)]$. Since $|AB|=|\tau(A)\tau(B)|$
and $|A\tau(A)|=|AB|+|\tau(A)B|$, $|B\tau(B)|=|\tau(A)\tau(B)|+|\tau(A)B|$, we get
$|A\tau(A)|=|B\tau(B)|$.
\end{proof}

\begin{Corollary}\label{TranslationsOfLineAbSubgroup}
Let $(l,\mathcal{I})$ be an Erlangen line.
Translations form an Abelian subgroup of $\mathcal{I}$.
\end{Corollary}
\begin{proof}
By \ref{TranslationsOfLineSubgroup} translations form a subgroup of $\mathcal{I}$.
Suppose $A\in l$ and $\tau_1, \tau_2$ are two non-trivial translations.
Pick $A\in l$ and assume $\tau_2(\tau_1(A))$ is on the other side of $\tau_1(A)$ than $A$.
The length of $s[A,\tau_2(\tau_1(A))]$ equals the sum of $|A\tau_1(A)|+|A\tau_2(A)|$
by \ref{TranslationSizeProp}.
Similarly, the length of $s[A,\tau_1(\tau_2(A))]$ equals the sum of $|A\tau_1(A)|+|A\tau_2(A)|$ resulting in $\tau_2(\tau_1(A))=\tau_1(\tau_2(A))$ and $\tau_2\circ\tau_1=\tau_1\circ \tau_2$ due to existence of a fixed point of the commutator of the two translations.

If $\tau_2(\tau_1(A))$ in the same side of $\tau_1(A)$ as $A$, we switch to $\tau_2^{-1}$ and obtain its commutativity with $\tau_1$.
\end{proof}

\begin{Theorem}
Let $(l,\mathcal{I})$ be an Erlangen line. The Grothendieck group of the monoid of lengths of $l$ is isomorphic to the group of translations.
\end{Theorem}
\begin{proof}
Given a commutative monoid, i.e. a set $S$ with commutative and associative addition, one creates the Grothendieck group $G(S)$ of $S$ (see \cite{Grot}) as equivalence sets of ordered pairs $(m,n)$ of elements of $S$. Namely, $(m,n)\equiv (m',n')$ if $m+n'=m'+n$. Notice all pairs $(m,m)$ are equivalent and form the neutral element of $G(S)$ if the addition in $G(S)$ is defined via $(m,n)+(k,p)=(m+k,n+p)$. 

Pick a maximal ray $r$ in $l$. Orient the line by requiring that $A < B$ if and only if $ray[A,B]$ is a positive maximal ray.
Given two closed line segments label their endpoints $A$, $B$, $C$, and $D$ so that
$A\leq B$ and $D\leq C$.
Assign $\tau_{CD}\circ \tau_{AB}$ to the pair $(|AB|,|CD|)$. That creates an isomorphism from the group of translations onto the Grothendieck group of the monoid of lengths of $l$.
\end{proof}

\begin{Proposition}\label{DividingLengthViaTranslationsProp}
Let $n > 0$ be an integer and let $s[A,B]$, $s[C,D]$ be two line segments of an Erlangen line.
The length of $s[A,B]$ is $n$ times the length of $s[C,D]$ if and only if
$\tau_{CD}^n=\tau_{AB}$ or $\tau_{CD}^{-n}=\tau_{AB}$.
\end{Proposition}
\begin{proof}
If $\tau_{CD}^{-n}=\tau_{AB}$, then $\tau_{DC}^{n}=\tau_{AB}$,
so assume $\tau_{CD}^n=\tau_{AB}$. Put $C_i= \tau_{CD}^i(A)$ for $0\leq i\leq n$.
Using \ref{TranslationSizeProp} observe $|C_iC_{i+1}|=|CD|$ for all $0\leq i\leq n-1$, and the line segments $s[C_i,C_{i+1}]$,
$s[C_j,C_{j+1}]$ have at most one common point if $i\ne j$.

Conversely, assume $s[A,B]$ contains points $C_0,\ldots, C_n$ such that $C_0=A$, $C_n=B$,
$|C_iC_{i+1}|=|CD|$ for all $0\leq i\leq n-1$, and the line segments $s[C_i,C_{i+1}]$,
$s[C_j,C_{j+1}]$ have at most one common point if $i\ne j$.
Notice the translation from $C_i$ to $C_{i+1}$ equals the translation from $C_j$ to $C_{j+1}$ for all $i,j$ and those translations equal either $\tau_{CD}$ or $\tau_{DC}$.
\end{proof}

\begin{Lemma}[Archimedes' Axiom]\label{Archimedes' Axiom}
Given points $A$, $B$, $C$, and $D$ of an Erlangen line such that $C\ne D$ there
is a natural number $n$ satisfying $n\cdot |CD| > |AB|$.
\end{Lemma}
\begin{proof}
$n=1$ works if $|CD| > |AB|$, so consider the case of $|CD|\leq |AB|$. On the line segment $s[A,B]$ look at all points $X$ satisfying $|AX|=k\cdot |CD|$ for some natural $k$ and add all line segments $s[A,X]$ together. It is a non-empty convex subset $\Sigma$ of $s[A,B]$ and either $\Sigma=s[A,E]$ for some $E$ or its complement
is of the form $s[E,B]$ for some $E$. In the first case extending $s[A,E]$ towards $B$ by adding a line segment of length $s[C,D]$ creates a line segment of length $n\cdot |CD| > |AB|$ for some $n$. In the second case one looks at the point $F$ between $A$ and $E$ satisfying $|FE|=|CD|$. There is $X\in s[F,E)$ with $|AX|=k\cdot |CD|$ and adding a line segment of length $s[C,D]$ to $s[A,X]$ creates a line segment of length $(k+1)\cdot |CD| > |AB|$.
\end{proof}

Using \ref{DividingLemma} one can define multiplication of length of line segments by rational numbers:
$\frac{m}{n}|AB|$ is the length $L$ such that $n\cdot L=m\cdot |AB|$.
One can then extend it to multiplication by all positive real numbers $w$ as follows:
pick a natural number $n > w$, then pick a line segment $s[C,D]$ with $|CD| > n\cdot |AB|$
(see \ref{Archimedes' Axiom}),
and consider  two subsets of $s[C,D]$:
$S_1=\{X\in s[C,D]\mid t\cdot |CX| < |CD| \mbox{ for all rationals } t < w\}$ and 
$S_2=\{X\in s[A,B]\mid t\cdot |AX| > |CD| \mbox{ for all rationals } t > w\}$. Both are convex, non-empty and neither can be a closed line segment. Therefore their union is not the whole $s[C,D]$ and the point $X$ outside of that union is declared as the one that satisfies $w\cdot |AX|=|AB|$. 

The following is an extension of Archimedes' Axiom \ref{Archimedes' Axiom}:
\begin{Proposition}\label{RatiosOfLengthsLine}
Given points $A$, $B$, $C$, and $D$ of an Erlangen plane such that $C\ne D$ there
is a real number $t$ satisfying $t\cdot |CD|= |AB|$.
\end{Proposition}
\begin{proof}
Consider all rational numbers $q\ge 0$ satisfying $q\cdot |CD|\leq |AB|$. They form a convex subset $\Sigma$ of $[0,n]$ for some natural number $n$ such that $n\cdot |CD| > |AB|$. The real number $t$ such that $\Sigma=[0,t]$ or the complement of $\Sigma$ in $[0,n]$ is equal to $[t,n]$ is the number we were looking for.
\end{proof}

\section{Isometries and isomorphisms of Erlangen lines}

\begin{Definition}\label{IsometryLineDef}
Suppose $(l,\mathcal{I})$ is an Erlangen line. A function $f:l\to l$ is an \textbf{isometry}
if it preserves the length of closed line segments, i.e. $|AB|=|f(A)f(B)|$ for all $A,B\in l$.
\end{Definition}

Notice isometries preserve betweenness.

\begin{Theorem}\label{IsometriesAreIsomorphismsLine}
Let $(l,\mathcal{I})$ be an Erlangen line. The following conditions are equivalent
for any function $f:l\to l$:\\
a. $f\in \mathcal{I}$,\\
b. $f:l\to l$ is an isometry.
\end{Theorem}
\begin{proof}
a.$\implies$b. follows from the definition of congruency of line segments.\\
b.$\implies$a.
By composing $f$ with a translation and a reflection (if necessary) we may assume $f$
has two different fixed points $A$ and $B$. Suppose $C\in s(A,B)$. Notice $f(C)\in s(A,B)$ as otherwise either $|AC|\ne |Af(C)|$ or $|BC|\ne |Bf(C)|$. In that case $C=f(C)$.
Similarly, one can show $f(C)=C$ for points outside of $s[A,B]$.
\end{proof}

Since Erlangen lines are pairs of structures on a set, we need to define their isomorphisms accordingly.

\begin{Definition}\label{IsomorphismOfErlangenLinesDef}
Let $(l_1,\mathcal{I}_1)$ and $(l_1,\mathcal{I}_1)$ be two Erlangen lines.
An \textbf{isomorphism of Erlangen lines} is a bijection $f:l_1\to l_2$ preserving betweenness such that any bijection $g:l_2\to l_2$ belongs to $\mathcal{I}_2$
if and only if $f^{-1}\circ g\circ f$ belongs to $\mathcal{I}_1$.
\end{Definition}

\begin{Example}\label{IsoBetweenHypLines}
The function $f(x):=\frac{x+1}{1-x}$ from Hyperbolic line II \ref{The hyperbolic line II} 
to Hyperbolic line I \ref{The hyperbolic line I} is an isomorphism.

Indeed, $f^{-1}\circ g\circ f$ is a rational function for any rational function $g$
and $f^{-1}\circ g\circ f$ preserves/reverses $1$, $-1$ if and only if $g$ preserves/reverses $0$, $\infty$. The only rational functions $g$ on $(0,\infty)$ that preserve $0$, $\infty$ are dilations $x\to c\cdot x$ for some $c > 0$.
The only rational functions $g$ on $(0,\infty)$ that switch $0$ and $\infty$ are of the form $x\to c\cdot x^{-1}$ for some $c > 0$.
\end{Example}

\begin{Theorem}\label{CharOfIsoOfErlangenLines}
Let $l_1, l_2$ be two Erlangen lines and let $f:l_1\to l_2$ be a bijection preserving betweenness.
The following conditions are equivalent:\\
1. $t\cdot |AB|=|CD|$ in $l_1$ implies $t\cdot |f(A)f(B)|=|f(C)f(D)|$ in $l_2$ for all $A,B,C,D\in l_1$ and all real numbers $t\ge 0$,\\
2. $t\cdot |AB|=|CD|$ in $l_1$ implies $t\cdot |f(A)f(B)|=|f(C)f(D)|$ in $l_2$ for all $A,B,C,D\in l_1$ and all rational numbers $t\ge 0$,\\
3. $n\cdot |AB|=|CD|$ in $l_1$ implies $n\cdot |f(A)f(B)|=|f(C)f(D)|$ in $l_2$ for all $A,B,C,D\in l_1$ and all natural numbers $n$,\\
4. $|AB|=|CD|$ in $l_1$ implies $|f(A)f(B)|=|f(C)f(D)|$ in $l_2$ for all $A,B,C,D\in l_1$,\\
5. $f(M)$ is the midpoint of $f(A)$ and $f(B)$ in $l_2$ if $M$ is the midpoint of $A$ and $B$ in $l_1$,\\
6. $\tau_{f(X)(Y)}=f\circ \tau_{XY}\circ f^{-1}$ and
$ i_{f(X)}=f\circ i_X\circ f^{-1}$
for all points $X, Y$ of $l_1$,\\
7. $f$ is an isomorphism of Erlangen lines.
\end{Theorem}
\begin{proof}

1.$\implies$2.,  2.$\implies$3., 3.$\implies$4., and 4.$\implies$5. are obvious.

5.$\implies$1. Pick a non-trivial translation $\tau$ and $A(0)\in l_1$. Define $A(n)$ as $\tau^n(A(0))$ for all $n\in \mathbb{Z}$. Extend the definition of points $A(n)$ from integers to rational numbers $q$ of the form $q=\frac{m}{2}$, $m\in \mathbb{Z}$, by requiring $A(q)$ is the midpoint of $A(n)$ and $A(n+1)$ if $q$ is not an integer and $n < q < n+1$.
The next step is constructing, by induction, points $A(q)$ for all rational numbers
in set $\mathcal{D}$ which consists of rationals $q$ 
of the form $q=\frac{m}{2^k}$, $m\in \mathbb{Z}$ and $k\in \mathbb{N}$. This is done in a similar manner as above. Notice $|A(q)A(p)|=|p-q|\cdot |A(0)A(1)|$ and $|B(q)B(p)|=|p-q|\cdot |B(0)B(1)|$
if $B(q)$ is defined as $f(A(q))$. 

The final step is defining $A(t)$ and $B(t)$ for all other real numbers $t$. This is done using connectedness: we choose $n\in\mathbb{Z}$ with $n < t < n+1$
and consider the union $\Sigma$ of all closed line segments $s[A(n),A(q)]$, $q\in \mathcal{D}$
and $n \leq q < t$. $A(t)$ is the point $X$ such that either $\Sigma=s[A(n),X]$
or the complement of $\Sigma$ in $s[A(n),A(n+1)]$ is $s[X,A(n+1)]$.
$B(t)$ is constructed in analogous fashion. Notice $|A(t)A(w)|=|t-w|\cdot |A(0)A(1)$
and $|B(t)B(w)|=|t-w|\cdot |B(0)B(1)$ for all real numbers $t$ and $w$. Also, $f(A(t))=B(t)$ for all real $t$. Since every point in $l_1$ is of the form $A(t)$, Condition 1. follows.

4.$\implies$7. Given any $g\in \mathcal{I}_1$, the function $f\circ g\circ f^{-1}$
is an isometry of $l_2$, hence it belongs to $\mathcal{I}_2$.
Conversely, if $f\circ g\circ f^{-1}$
is an isometry of $l_2$ for some $g\in \mathcal{I}_1$, then $g$ is an isometry of $l_1$ and must belong to $\mathcal{I}_1$ by \ref{IsometriesAreIsomorphismsLine}.

7.$\implies$6. Condition 6. amounts to saying $f\circ g\circ f^{-1}$
is a translation (reflection) of $l_2$ if $g$ is a translation (reflection) of $l_1$.
That is the same as saying $f\circ g\circ f^{-1}$
has no fixed points (has a fixed point) in $l_2$ if $g$ has no fixed points (has a fixed point) in $l_1$. That is certainly true.

6.$\implies$4. If $|AB|=|CD|$, we may assume $\tau_{AB}=\tau_{CD}$. Hence
$\tau_{f(A)f(B)}=\tau_{f(C)f(D)}$ and $|f(A)f(B)|=|f(C)f(D)|$. 
\end{proof}

\begin{Theorem}\label{IsoOfErlangenLines}
Let $l_1, l_2$ be two Erlangen lines. Given two different points $A_0$ and $A_1$ of $l_1$ and two different points $B_0$ and $B_1$ of $l_2$
there is a unique bijection $f:l_1\to l_2$ preserving betweenness such that
$f(A_0)=B_0$, $f(A_1)=B_1$, and 
$|AB|=|CD|$ in $l_1$ implies $|f(A)f(B)|=|f(C)f(D)|$ in $l_2$. In particular, lines $l_1$ and $l_2$ are isomorphic.
\end{Theorem}
\begin{proof}
Given a point $X\in ray[A_0,A_1]$ consider the real number $t_X$ such that $|A_0X|=t_X\cdot |A_0A_1|$ (see \ref{RatiosOfLengthsLine}). Find $Y\in ray[B_0,B_1]$ so that $|B_0Y|=t_X\cdot |B_0B_1|$ and declare $f(X)=Y$.

Given a point $X\in ray[A_1,A_0]$ consider the real number $t_X$ such that $|A_1X|=t_X\cdot |A_0A_1|$ (see \ref{RatiosOfLengthsLine}). Find $Y\in ray[B_1,B_0]$ so that $|B_1Y|=t_X\cdot |B_0B_1|$ and declare $f(X)=Y$.

If two bijections $f_1,f_2:l_1\to l_2$ preserve betweenness,
$f_i(A_0)=B_0$, $f_i(A_1)=B_1$, and 
$|AB|=|CD|$ in $l_1$ implies $|f_i(A)f(B)|=|f_i(C)f_i(D)|$ in $l_2$ and $i=1,2$,
then $f_2\circ f_1^{-1}$ must be the identity by \ref{IsometriesAreIsomorphismsLine} as it fixes two points.
\end{proof}

\section{Length functions}

\begin{Definition}\label{LengthFunctionDef}
Given an Erlangen line $(l,\mathcal{I})$ by a \textbf{length function}
$\mu$ we mean a function from the set of line segments of $l$ to the set of non-negative reals such that the following conditions hold:\\
1. $\mu(s[A,B])=0$ if and only if $s[A,B]$ contains at most one point,\\
2. if $s[A,B]$ is the union of two line segments $s[C,D]$ and $s[E,F]$, then 
$$\mu(s[A,B])=\mu(s[C,D])+\mu(s[E,F])-\mu(s[C,D]\cap s[E,F])$$
3. $\mu(s[f(A),f(B)])=\mu(s[A,B])$ for all $f\in \mathcal{I}$ and all $A,B\in l$.
\end{Definition}

\begin{Example} 
In the case of the real line \ref{The real line}
the function $\mu(s[t,w]):=|t-w|$ is a length function.
\end{Example}

\begin{Example}
In the case of the hyperbolic line I \ref{The hyperbolic line I}
the function $\mu(s[t,w]):=|\ln(t/w)|$ is a length function since $f(x)=\ln(x)$
is an isomorphism from the hyperbolic line I to the real line.
Indeed, $f^{-1}\circ g\circ f(x)$ for $g(t)=m\cdot t+b$ amounts
to $x\to x^m\cdot e^b$.
\end{Example}

\begin{Example}
In the case of the hyperbolic line II \ref{The hyperbolic line II} one can find a length function by using the isomorphism $f(x):=\frac{x+1}{1-x}$ from that line to the 
Hyperbolic line I (see \ref{IsoBetweenHypLines}).
\end{Example}

\begin{Proposition}\label{RelationBetweenLengthFunctionsProp}
Let $(l,\mathcal{I})$ be an Erlangen line.
Given two different length functions $\mu$ and $\lambda$ there is a constant $c > 0$ such that $\lambda=c\cdot \mu$.
\end{Proposition}
\begin{proof}
Since every Erlangen line is isomorphic to the real line (see \ref{IsoOfErlangenLines}), it suffices to show \ref{RelationBetweenLengthFunctionsProp} for the real line and that is left to the reader.
\end{proof}

\begin{Example}
The hyperbolic line III \ref{The hyperbolic line III}
is isomorphic to the real line via the function $f(t,x)=arcsinh(x)$. Therefore
$\mu(s[(t_1,x_1),(t_2,x_2)])$ defined as $$|arcsinh(x_1)-arcsinh(x_2)|$$
is a length function of the hyperbolic line III.

Recall that $\mathcal{I}$
consists of linear transformations
$$ g(t,x)=(a\cdot t+ b\cdot x, \pm (a\cdot x+ b\cdot t))$$
where $a > 0$ and $a^2-b^2=1$. To see the form of $f\circ g\circ f^{-1}$
we notice $a=\cosh(c)$, $b=\sinh(c)$ for some $c$, and $f^{-1}(s)=(\cosh(s),\sinh(s))$.
Therefore
$$g\circ f^{-1}(s)=g(\cosh(s),\sinh(s))=$$
$$(\cosh(c)\cdot \cosh(s)+ \sinh(c)\cdot \sinh(s), \pm (\cosh(c)\cdot \sinh(s)+ \sinh(c))\cdot \cosh(s)))=$$
$$(\cosh(\pm (s+c)),\sinh(\pm (s+ c)))$$
resulting in $f\circ g\circ f^{-1}(s)=\pm (s+ c)$.
\end{Example}

\section{Euclidean lines}\label{Euclidean lines}

In this section we outline how the ideas of Euclid lead to Erlangen lines.

The Elements of Euclid include the following five "common notions":
\par\noindent
    1. Things that are equal to the same thing are also equal to one another (\textbf{Transitive property of equality}).\\
2.
    If equals are added to equals, then the wholes are equal (\textbf{Addition property of equality}).\\
3.
    If equals are subtracted from equals, then the remainders are equal (\textbf{Subtraction property of equality}).\\
4.
    Things that coincide with one another are equal to one another (\textbf{Reflexive Property}).\\
5.
    The whole is greater than the part.

The above common notions are usually applied to length of line segments and measures of angles. Based on the common notions we outline how to create Erlangen lines using Euclid's concepts.

\begin{Definition}
A \textbf{Euclidean line} is a line $l$ with the notion of congruency of line segments and a linear order on equivalence classes of the congruency satisfying the conditions below.
The congruence of line segments is indicated by $s[A,B]\equiv s[C,D]$
and the equivalence class of $s[A,B]$ is denoted by $|AB|$:\\
a. If $C\in s[A,B]$, $C'\in s[A',B']$ and $|AC|=|A'C'|$, $|CB|=|C'B'|$, then $|AB|=|A'B'|$, \\
b. If $C\in s[A,B]$, $C'\in s[A',B']$ and $|AC|=|A'C'|$, $|AB|=|A'B'|$, then $|CB|=|C'B'|$, \\
c. If If $C\in s[A,B)$, then $|AC| < |AB|$,\\
d. If $A\ne B$, then for each maximal ray emanating from a point $C$ there is a point $X$ satisfying $|AB|=|CX|$.
\end{Definition}
Notice that a. corresponds to Common Notion 2. b. corresponds to Common Notion 3.
c. corresponds to Common Notion 5. and d. corresponds to the intersection of the line $l$ and the circle centered at $C$ with radius $|AB|$ which Euclid simply constructed using a straightedge and compass.

Obviously, the results of the previous section say that every Erlangen line is a Euclidean line. The point of this section is to show the converse.

\begin{Theorem}
For every Euclidean line $l$ there is a unique group $\mathcal{I}$ of isomorphisms of $l$
such that $(l,\mathcal{I})$ is an Erlangen line and $s[A,B]\equiv s[C,D]$
means exactly that there is $f\in \mathcal{I}$ sending $s[A,B]$ onto $s[C,D]$.
\end{Theorem}
\begin{proof}
Given two maximal rays $r_1$ and $r_2$ in $l$ emanating from $A_1$ and $A_2$, respectively, define $f:l\to l$ as follows:\\
1. If $X\in r_1$, then $f(X)$ is the unique point on $r_2$ satisfying $|A_1X|=|A_2f(X)|$.\\
2. If $X\notin r_1$, then $f(X)$ is the unique point on the opposite maximal ray to $r_2$ satisfying $|A_1X|=|A_2f(X)|$.\\
The remainder of the proof is left as an exercise.
\end{proof}

\section{Circles}

Euclid understood angles intuitively. Since Euclid started from the concept of the length, he thought in terms of circles (defined as having a center and a radius) and angles were essentially reduced to circles as ancient Greeks preferred to think in terms of objects that are finite in some sense.  In our approach we see angles as line segments in the boundary at infinity of Pasch spaces and we see circles as the simplest boundaries at infinity that do not consist of finitely many points. Also, circles and lines are Pasch spaces whose boundaries at infinity consist of two points.

\begin{Definition}\label{CircleDef}
A \textbf{circle} is a set $c$ with a relation of betweenness satisfying the following conditions:\\
a. $c$ contains at least three points,\\
b. each point $A$ of $c$ has the antipodal point $a(A)$ such that the line segment
$s[A,a(A)]$ contains exactly two points,\\
c. each closed line segment $s[A,B]$ is connected if $A$ is not antipodal to $B$,\\
d. for each $A\in c$ the complement $c\setminus\{A,a(A)\}$ can be expressed as the union of two disjoint non-empty convex sets $l_1$ and $l_2$ that are lines when inheriting betweenness from $c$.
\end{Definition}

\begin{Example}[The circle of angles]
Reals modulo $360$ form a circle. Two different (mod $360$) numbers are antipodal if their sum equals $0$ mod $180$. If two classes of numbers mod $360$ are not antipodal, we pick their representatives $x$ and $y$ in the interval $[0,360)$. Assume $x < y$. If $|x-y| < 180$,
then the line segment $s[x,y]$ is defined as $\{t \mid x\leq t\leq y\}$. Otherwise,
$s[x,y]$ is defined as the union of $\{t \mid 0\leq t\leq x\}$ and $\{t \mid y\leq t\leq 360\}$.
\end{Example}

\begin{Example}
$\mathbb{R}\cup\{\infty\}$ is a circle if the antipodal of $0$ is $\infty$ and the antipodal of $x\ne 0$ is $\frac{-1}{x}$. $s[\infty,x]$ for $x > 0$
is defined as $[x,\infty)\cup\{\infty\}$ and $s[\infty,x]$ for $x < 0$
is defined as $(-\infty,x]\cup\{\infty\}$. If $x,y$ are reals and $[x,y]$ does not contain a pair of antipodal points, then $s[x,y]$ is defined as $[x,y]$.
Otherwise, $s[x,y]$ is defined as $s[x,\infty]\cup s[\infty,y]$. 
\end{Example}

\begin{Proposition}
Each point of a circle is inside a closed connected segment.
\end{Proposition}
\begin{proof}
Given a point $A$ of $c$ there is a point $B$ different from both $A$ and $a(A)$. $A$ is in one component $l_1$ of $c\setminus\{B,a(B)\}$ which is a line,
hence $A$ is inside a closed line segment contained in $l_1$.
\end{proof}

Notice circles have two maximal rays emanating from each point.

\begin{Proposition}\label{CircleIsMaxRaysSpace}
If $c$ is a circle, then it is a space with maximal rays
and for each $A\in c$ there are two maximal rays emanating from $A$;
the union of $A$ and each of the components of 
$c\setminus\{A,a(A)\}$. 
\end{Proposition}
\begin{proof} 
According to Definition \ref{SpaceWithMaximalRays}
we need to show that each connected closed line segment is contained in the interior of another connected closed line segment, and if connected line segments $s[A,B]$ and $s[A,C]$ have a common interior point, then one of them is contained in the other.

Given a connected line segment $s[A,B]$ choose the component of $c\setminus\{A,a(A)\}$ containing $s(A,B]$ then choose $C\in s(B,a(A))$. Notice $s(A,B]\subset s(A,C)$.
Now, for any point $D\in s(A,a(C))$ one has $s[A,B)\subset s(D,B)$ resulting in
$s[A,B]\subset s(C,D)$.

Suppose connected line segments $s[A,B]$ and $s[A,C]$ have a common interior point
$D$. Pick $E\in s(A,D)$, $G\in s(E,D)$ and notice $s[E,B]$ and $s[E,C]$ have $G$ as a common interior point. As they both lie in a line (one of the components of $c\setminus\{A,a(A)\}$), one of them is contained in the other.

To show $c$ is a Pasch space we need (see \ref{PaschCondition}) to show that
if there are two maximal rays $r_1,r_2$ emanating from $A$ and points $B,C\in r_1$ and $D\in c$, with $r_2$ intersecting $s[B,D]$, then $r_2$ also intersects $s[C,D]$. It is clearly so if $r_1=r_2$. If $r_1\ne r_2$, then they are antipodal and $D$ must be in $r_2$, so $r_2$ does intersect $s[C,D]$.
\end{proof}

\begin{Proposition}\label{AntipodalMapOnCirclePreservesSegments}
For every circle $c$ the antipodal map $a:c\to c$ preserves closed line segments.
\end{Proposition}
\begin{proof}
First, let us show that $a$ preserves betweenness. Suppose $C$ is between $A$ and $B$. Clearly, if $C=A$ or $C=B$, then $a(C)$ is between $a(A)$ and $a(B)$, so assume $C\in s(A,B)$, in particular $a(B)\ne A$. Let $l_1$ be the component of $c\setminus\{B,a(B)\}$ containing $C$. $l_1$ contains $s[B,A]$, hence it does not contain $a(C)$. $a(C)$ cannot be in $s[C,a(A)]$, so the only possibility is that $a(C)$ is in $s[a(A),a(B)]$.

$a$ clearly preserves closed line segments whose endpoints are antipodal, so assume $s[A,B]$ is a connected line segment. As $a$ preserves betweenness,
$a(s[A,B])\subset s[a(A),a(B)]$ and $a(s[a(A),a(B)])\subset s[A,B]$. Apply $a$ to the first inclusion and conclude $a(s[A,B])=s[a(A),a(B)]$.
\end{proof}

\begin{Corollary}
Every circle is a spherical Pasch space with boundary at infinity consisting of two points.
\end{Corollary}
\begin{proof}
Use \ref{CircleIsMaxRaysSpace} and \ref{AntipodalMapOnCirclePreservesSegments}.
\end{proof}

\begin{Proposition}\label{CircleasPaschChar}
A spherical Pasch space $\Pi$, whose boundary at infinity at some point $A$ consists of two points, is a circle.
\end{Proposition}
\begin{proof}
Consider the two maximal rays $r_1$ and $r_2$ emanating from $A$. Obviously, $l_1:=r_1\setminus \{A\}$ and $l_2:=r_2\setminus \{A\}$ are lines, hence are convex. Suppose $B\in l_1$, $C\in l_2$, and $A\notin s[B,C]$. 
In that case $a(A)$ must be in $s[B,C]$ as otherwise we pick $D\in s(B,C)$
and look at $ray[A,D]$. It contains both $B$ and $C$, hence must be different from two given rays, a contradiction.

For a similar reason, both $\{a(A)\}\cup l_1$ and $\{a(A)\}\cup l_2$ must be the only maximal rays emanating from $a(A)$.

Given any other point of $\Pi$, say $B\in l_1$, the two maximal rays emanating from $B$ are $s[B,A]\cup s[A,a(B))$ and $s[B,a(A)]\cup s[a(A),a(B))$. The above proof shows $\{B,a(B)\}$ separates $\Pi$ into two components of convexity, each being a line. Thus, $\Pi$ is a circle.
\end{proof}

\begin{Corollary}\label{SubCirclesOfPaschSpaces}
Let $\Pi$ be a spherical Pasch space. If $B\ne a(A)$, then 
$$c:=s[A,B]\cup s[B,a(A)]\cup s[a(A),a(B)] \cup s[a(B),A]$$
is a circle.
\end{Corollary}
\begin{proof}
Notice $a(c)=c$ and $c$ is convex as $ray[C,X]$ is antipodal to $ray[C,a(X)]$
for all $C\ne X$. Therefore $c$ is a Pasch space with boundary at infinity consisting of two points, hence $c$ is a circle by \ref{CircleasPaschChar}.
\end{proof}

\section{Erlangen circles}

\begin{Definition}
An \textbf{Erlangen circle} is a pair $(c,\mathcal{I})$ consisting of a circle $c$ and a subgroup $\mathcal{I}$ of isomorphisms of $c$
such that Condition \ref{Homogeneity and rigidity of the circle} is satisfied.
\end{Definition}

\begin{Condition}[Homogeneity and rigidity of the Erlangen circle]\label{Homogeneity and rigidity of the circle}
For every ordered pair of maximal rays in $c$ there is exactly one isomorphism in $\mathcal{I}$ sending the first ray onto the other.
\end{Condition}

\begin{Example}[The hyperbolic circle]
$c=\mathbb{R}\cup\{\infty\}$, $\mathcal{I}$ is the set of all non-constant rational functions $f(x)$ such that
$f(\frac{-1}{x})=\frac{-1}{f(x)}$ for all $x\in l$.
\end{Example}

\begin{Example}[The circle of angles]
Reals mod 360 with isomorphisms in the form $f(x)=m\cdot x+b$, where $m=\pm 1$.
\end{Example}

Notice that we do not assume the antipodal map $a$ belongs to $\mathcal{I}$. However, we plan to prove it later on (see \ref{AntipodalMapIsInI}). At this moment notice that $a$ commutes with all elements of $\mathcal{I}$.

\begin{Lemma}\label{AntipodalMapCommutes}
If $f\in \mathcal{I}$, then $f\circ a=a\circ f$, where $a$ is the antipodal map of the circle $c$.
\end{Lemma}
\begin{proof}
Since $f$ preserves closed line segments, $f(s[X,a(X)])$ has to have exactly two points for any $X\in c$. One of them is $f(X)$, so the other must be $a(f(X))$ which means $f(a(X))=a(f(X))$.
\end{proof}

\begin{Corollary}\label{FixedPointsOnCircle}
 Let $(c,\mathcal{I})$ be an Erlangen circle. An isomorphism in $\mathcal{I}$ that has a fixed point is an involution. 
An isomorphism in $\mathcal{I}$ that has three fixed points is the identity. 
\end{Corollary}
\begin{proof}
Same as in \ref{InvolutionAndIdLines}.
\end{proof}

\begin{Corollary}\label{ErlangenCirclesAreRM}
If $c,\mathcal{I})$ is an Erlangen circle, then $(c,\mathcal{I})$ is a space with rigid motions.
\end{Corollary}
\begin{proof}
Let $A\ne B$ be two points of $c$.
Suppose $f(s[A,B])\subset s[A,B]$ for some $f\in \mathcal{I}$. In order to apply \ref{NilpotentCaseOfRM} we need to show $f$ restricted to $s[A,B]$ is of finite order.
It is clearly so if $A$ and $B$ are antipodal. Therefore assume
$s[A,B]$ is connected.
By \ref{FixedPointTheorem} $f$ has a fixed point
and by \ref{FixedPointsOnCircle}
$f^2=id$. 
By \ref{NilpotentCaseOfRM} $(c,\mathcal{I})$ is a space with rigid motions.
\end{proof}

\begin{Definition}\label{SymmetryCircleDef}
Let $(c,\mathcal{I})$ be an Erlangen circle and let $A\in c$. The \textbf{reflection} $i_A$ of $c$ in $A$ is the element of $\mathcal{I}$ that switches the two maximal rays at $A$.
\end{Definition}

\begin{Definition}\label{TranslationCircleDef}
Let $(c,\mathcal{I})$ be an Erlangen circle and let $A,B\in c$. The \textbf{translation} $\tau_{A,B}$ of $c$ from $A$ to $B$ is the element of $\mathcal{I}$ defined as follows:\\
1. If $A=B$, then $\tau_{A,B}=id$,\\
2. If $A$ and $B$ are antipodal, then $\tau_{A,B}$ sends each maximal ray $r$ at $A$
to the maximal ray at $B$ disjoint with $r$,\\
3. If $A$ and $B$ are not antipodal, then $\tau_{A,B}$ sends $ray[A,B]$ to the ray at $B$ not containing $A$.
\end{Definition}

Axiom 4 of Euclid says that all right angles are congruent. For Erlangen circles right angles can be defined very precisely.

\begin{Definition}\label{RightAngleDef}
Let $(c,\mathcal{I})$ be an Erlangen circle and let $A\in c$. A point $B$ forms a \textbf{right angle} with $A$ if the reflection $i_A$ in $A$ sends $B$ to its antipodal point $a(B)$.
\end{Definition}

\begin{Proposition}
Let $(c,\mathcal{I})$ be an Erlangen circle and let $A\in c$.
There are exactly two points in $c$ that form a right angle with $A$ and they are antipodal.
\end{Proposition}
\begin{proof}
Consider the composition $a\circ i_A$ of the reflection in $A$ and the antipodal map $a$. Notice that it flips the rays on $l$, a component of $c\setminus \{A,a(A)\}$.
Therefore $a\circ i_A$ has a fixed point in $l$ by \ref{TranslationLemma} and its antipodal point is a fixed point of $a\circ i_A$ as well by \ref{AntipodalMapCommutes}. There cannot be any more fixed points of $a\circ i_A$
as $i_A$ flips $s[A,B]$ with $s[A,a(B)]$ and it flips $s[a(A),B]$ with $s[a(A),a(B)]$.
\end{proof}

\subsection{Congruence}
Since Erlangen circles are spaces with rigid motions by \ref{ErlangenCirclesAreRM} we have the notion of congruence for line segments.

\begin{Corollary}\label{AntipodalSegmentsAreCongruent}
If $A$ and $B$ are two points of an Erlangen circle $c$ with the antipodal map $a$, then
the line segments $s[A,a(A)]$ and $s[B,a(B)]$ are congruent.
\end{Corollary}
\begin{proof}
It is clearly so if $A=B$ or $A=a(B)$. Otherwise, let $r_1$ be the ray emanating from $A$ and containing $B$ and let $r_2$ to be the ray emanating from $B$ and containing $A$.
Pick an isomorphism $\rho$ sending $r_1$ onto $r_2$. Obviously, $\rho(A)=B$.
Also $\rho(a(A))$ must be a point such that $s[\rho(A),\rho(a(A))]$ contains only two points as it is equal to $\rho(s[A,a(A)]$. Therefore, $\rho(a(A))=a(B)$.
\end{proof}

\begin{Proposition}\label{AntipodalMapIsAnIsometry}
If $c$ is an Erlangen circle with the antipodal map $a$, then $s[a(A),a(B)]$ is congruent to $s[A,B]$ for any points $A,B\in c$.
\end{Proposition}
\begin{proof}
Suppose $s[a(A),a(B)]$ is not congruent to $s[A,B]$ for some points $A,B\in c$.
In view of
\ref{AntipodalSegmentsAreCongruent} points $A$ and $B$ are not antipodal.
Since we can send $ray[a(A),a(B)]$ onto $ray[A,B]$ via an isomorphism $f$, $f(B)$ cannot be $a(B)$. Enlarge $\mathcal{I}$ to the group $\mathcal{J}$ of isomorphisms of $c$ by adding isomorphisms of the form $a\circ h$,
$h\in \mathcal{I}$. Notice that in the congruence induced by $\mathcal{J}$ we arrived at a line segment congruent to its proper subset: either $(f\circ a)(s[A,B])$ is a proper subset of $s[A,B]$ or vice versa. If we show $(c,\mathcal{J})$ is a space with rigid motions, we arrive at a contradiction.

Suppose we have $g\in \mathcal{J}$ such that $g(s[A,B])$ is a proper subset of $s[A,B]$. That can only happen if $s[A,B]$ is not finite, i.e. $s[A,B]$ is connected and $A\ne B$.
Therefore $g$ has a fixed point and $g=a\circ f$ for some $f\in \mathcal{I}$.
Since $g^2=f^2$, $g^2$ is an involution by \ref{FixedPointsOnCircle}. Thus $g^4=id$
and $g$ is of finite order on $s[A,B]$.

\end{proof}

The following is a strengthening of \ref{AntipodalMapIsAnIsometry}.
\begin{Corollary}\label{AntipodalMapIsInI}
If $(c,\mathcal{I})$ is an Erlangen circle, then the antipodal map $a$ belongs to $\mathcal{I}$.
\end{Corollary}
\begin{proof}
Pick $A$ in $c$ and consider $f\in \mathcal{I}$ sending $A$ to $a(A)$ and sending each maximal ray $r$ at $A$ to the maximal ray at $a(A)$ that is disjoint with $r$. We need to show $f=a$.
Suppose $f(B)\ne a(B)$ for some $B\in c$. Either $f(B)\in s(a(B),a(A))$
or $f(B)\in s(A,a(B))$. In both cases we arrive at a contradiction in the form of detecting a line segment congruent to its proper subset. 
\end{proof}

\begin{Proposition}[Axiom 4 of Euclid]\label{Axiom 4 of Euclid}
All right angles in an Erlangen circle are congruent.
\end{Proposition}
\begin{proof}
Suppose pairs $A$, $B$ and $C$, $D$ form right angles. Choose an isomorphism $\rho$ sending $A$ to $C$. Put $E=\rho(B)$. Notice $\rho\circ i_A\circ \rho^{-1}$ is an involution fixing $C$ and not equal identity, hence equal to $i_C$. Equivalently, $ \rho\circ i_A= i_C\circ \rho$. Since $i_C(D)=a(D)$
and $i_A(B)=a(B)$, 
$$i_C(E)= i_C\circ \rho(B)= \rho\circ i_A (B)=\rho(a(B))=a(\rho(B))=a(E)$$
That means either $E=D$ or $E=a(D)$. If $E=D$, then congruence of $s[A,B]$
with $s[C,D]$ is established by $\rho$. If $E=a(D)$, then congruence of $s[A,B]$
with $s[C,D]$ is established by $i_C\circ \rho$. 
\end{proof}

\begin{Proposition}\label{PreservingRightAngles}
Any isomorphism $f\in \mathcal{I}$ of an Erlangen circle sends right angles to right angles.
\end{Proposition}
\begin{proof}
Suppose $A,B\in c$ form a right angle. Put $C=f(A)$ and $D=f(B)$. As above
$f\circ i_A\circ f^{-1}=i_C$ and $i_C(D)=f\circ i_A(B)=f(a(B))=a(f(B))=a(D)$.
\end{proof}

\begin{Corollary}
If $A$ and $B$ form a right angle in an Erlangen circle, then $B$ and $A$ also form a right angle.
\end{Corollary}
\begin{proof}
Let $C=i_B(a(A))$. If $C\ne A$, we may assume $C\in s(A,B]$.
Now, $s[a(B),C]$ is congruent to its proper subsegment, namely $i_A(s[B,C])$, a contradiction by
\ref{NonCongruenceOfProperSegmentsCircle}.
\end{proof}

\begin{Proposition}\label{RightAngleAlgebraically}
Points $A$ and $B$ of an Erlangen circle form a right angle if and only if the square of the translation $\tau_{AB}$ from $A$ to $B$ equals the antipodal map $a$.
\end{Proposition}
\begin{proof}
If $A$ and $B$ form a right angle, then $B$ and $\tau_{AB}(B)$ form a right angle and $\tau_{AB}(B)=a(A)$ as $\tau_{AB}(B)=A$ is not possible.
Therefore $\tau_{AB}^2$ is the translation from $A$ to $a(A)$ which is $a$.

Suppose $\tau_{AB}^2=a$. Hence $\tau_{AB}(B)=\tau_{AB}^2(A)=a(A)$
and $s[A,B]$ is congruent to $s[B,a(A)]$.
Choose $M$ on $ray[A,B]$ forming a right angle with $A$.
$M$ cannot be in $s(A,B)$ as in that case $i_M$ sends $s[a(A),B]$ to a proper subset of $s[A,B]$ which is congruent to $s[a(A),B]$. For a similar reason $M$ cannot be in $s(a(A),B)$. Concluding: $M=B$ and $A$ forms a right angle with $B$.
\end{proof}

\subsection{The structure of isomorphisms of Erlangen circles}

\begin{Proposition}\label{StructureOfIsoCircle}
Let $(c,\mathcal{I})$ be an Erlangen circle and let $f\in \mathcal{I}$ be not equal to the identity.\\
1. $f$ is a reflection if and only if it has a fixed point. \\
2. $f$ is either a reflection or a translation but not both.\\
3. $f$ is a translation if and only if it has no fixed points.
\end{Proposition}
\begin{proof}
1. Obviously, reflections have fixed points. Suppose $f(A)=A$. In this case the two maximal rays emanating from $A$ are either sent to themselves and $f=id$ (which is not possible) or the rays are swapped which means $f=i_A$.\\
2. If $f$ has a fixed point, it is a reflection by 1. Suppose $f$ has no fixed points
and $B=f(A)$ for some $A\in c$. The ray $ray[A,B]$ cannot be sent onto the ray $ray[B,A]$
as in that case $f(B)=A$ (if $f(B)\ne A$, then one detects a segment congruent to its proper subset) and there is a fixed point of $f$ in $s[A,B]$, a contradiction.
Therefore $f$ is a translation.\\
3. If $f=\tau_{AB}$ and has a fixed point $C$, then $f=i_C$ by 1. and $A$ belongs to the other maximal ray at $C$ than $B$. In that case $s[A,B]$ contains a fixed point ($C$ or $a(C)$) which is a contradiction as the ray $ray[A,B]$ is sent by $f$ to the maximal ray at $B$ not containing $s(A,B)$ where a fixed point is.

If $f$ has no fixed points, it must be a translation by 2.
\end{proof}

\begin{Corollary}
Let $(c,\mathcal{I})$ be an Erlangen circle and let $f\in \mathcal{I}$.
If $f^2$ is the antipodal map $a$, then $f$ is a translation and $f(A)$ forms a right angle with $A$ for all $A\in C$.
\end{Corollary}
\begin{proof}
Notice $f$ has no fixed points, so it is a translation from $A$ to $f(A)$ for all $A\in c$ by
\ref{StructureOfIsoCircle}. Using
\ref{RightAngleAlgebraically} one gets $f(A)$ forms a right angle with $A$ for all $A\in C$.
\end{proof}

\section{Measure of line segments in Erlangen circles}

This section is very similar to Section \ref{LengthOfLineSegmentsSec}
and part of Section \ref{Spaces with rigid motions}. However, there are some differences worth pointing out.
Since our basic example of an Erlangen circle is the circle of rays (see \ref{CircleOfAnglesErlangenPlane}), we will use the term of the measure of a line segment in an Erlangen circle instead of the length of it.

Traditionally, students are taught that using degrees to measure angles is interchangeable with using radians. It is actually only true in Euclidean geometries as radians refer to the circumference of circles.

Also, we are not considering angles of $270$ degrees or $-90$ degrees. Those arise only when the circle has a base point and is oriented.

\begin{Definition}
Given two line segments $s[A,B]$ and $s[C,D]$ of an Erlangen circle we say the \textbf{measure} of 
$s[A,B]$ is at most the \textbf{measure} of $s[C,D]$ (notation: $|AB|\leq |CD|$) if there is an isometry $\tau$ such that
$\tau(s[A,B])\subset s[C,D]$ or $C$ is antipodal to $D$. If, in addition, $s[A,B]$ is not congruent to $s[C,D]$, then we say the \textbf{measure} of 
$s[A,B]$ is smaller than the \textbf{measure} of $s[C,D]$ (notation: $|AB| < |CD|$).
\end{Definition}

\begin{Lemma}\label{LengthInequalityTransitivityCircle}
If $|AB|\leq |CD|$ and $|CD\leq |EF|$, then $|AB|\leq |EF|$.
\end{Lemma}
\begin{proof}
Similar to \ref{LengthInequalityTransitivityLine}.
\end{proof}

The following lemma implies that congruence of line segments is the same as equality of their measures.
\begin{Lemma}
If $|AB|\leq |CD|$ and $|CD|\leq |AB|$ in an Erlangen circle, then the line segments $s[A,B]$ and $s[C,D]$ are congruent.
\end{Lemma}
\begin{proof} In case of both line segments $s[A,B]$ and $s[C,D]$ being connected the proof is the same as for Erlangen lines.

Assume $s[A,B]$ is not connected, i.e. $A$ is antipodal to $B$.
Pick isomorphism $\rho$ such that $\rho(s[C,D])\subset s[A,B]$. Notice $s[C,D]$ cannot be connected as in such case $\rho(s[C,D])$ is a point. Therefore $C$ is antipodal to $D$ and $s[A,B]$ is congruent to $s[C,D]$.
\end{proof}

\begin{Lemma}
Given two line segments $s[A,B]$ and $s[C,D]$ of an Erlangen circle one has
$|AB|\leq |CD|$ or $|CD|\leq |AB|$.
\end{Lemma}
\begin{proof}
Similar to that for spaces with rigid motions in Section \ref{Spaces with rigid motions}.
\end{proof}

\subsection{Algebra of measures of line segments}

\begin{Definition}\label{SumOfLengthsDef}
Let $c$ be an Erlangen circle with the antipodal map $a$. The measure of a line segment $s[A,B]$
is equal to the sum of measures $|CD|$ and $|EF|$ if one of the following conditions is satisfied:\\
a. there is a point $M$ in $s[A,B]$
such that $|AM|=|CD|$ and $|MB|=|EF|$,\\
b. $A=a(B)$ and there is a point $M$
such that $|AM|=|CD|$ and $|MB|=|EF|$.
\end{Definition}

\begin{Lemma}
For every two line segments whose measures are at most the measure of a right angle there is a line segment whose measure equals the sum of measures.
\end{Lemma}
\begin{proof}
Find line segments congruent to the given two line segments so that they have the common endpoint $A$. If they have a common interior point, flip one of them via $i_A$. The union of the two line segments has the measure equal to the sum of measures of the original line line segments.
\end{proof}

\begin{Lemma}
If the measure of a line segment $s[A,B]$
is equal to the sum of measures $|CD|$ and $|EF|$ and the measure of a line segment $s[A',B']$
is equal to the sum of measures $|CD|$ and $|EF|$, then $|AB|=|A'B'|$.
\end{Lemma}
\begin{proof}
Similar to that for spaces with rigid motions in Section \ref{Spaces with rigid motions}.
\end{proof}

\begin{Lemma}
Suppose both $|AB|+|A'B'|$ and $|CD|+|C'D'|$ exist.
If $|AB|\leq |CD|$ and $|A'B'|\leq |C'D'|$, then $|AB|+|A'B'|\leq |CD|+|C'D'|$.
\end{Lemma}
\begin{proof}
Same as for spaces with rigid motions in Section \ref{Spaces with rigid motions}.
\end{proof}

\begin{Definition}\label{DifferenceOfLengthsDef}
Let $c$ be an Erlangen circle. The measure of a line segment $s[A,B]$
is equal to the \textbf{difference of measures} $|CD|$ and $|EF|$ if $|AB|+|EF|=|CD|$.
\end{Definition}

\begin{Lemma}
For every two measures $|AB|\ge |CD|$ there is a line segment whose measure equals the difference of measures.
\end{Lemma}
\begin{proof}
Same as for spaces with rigid motions in Section \ref{Spaces with rigid motions}.
\end{proof}

\begin{Corollary}
Translations form an Abelian subgroup of isomorphism of any Erlangen circle.
\end{Corollary}
\begin{proof}
Same as for Erlangen lines.
\end{proof}

\subsection{Dividing of measures of line segments}

\begin{Definition}
Given two line segments $s[A,B]$ and $s[C,D]$ of an Erlangen circle we say the \textbf{measure} of 
$s[A,B]$ is \textbf{twice the measure} of $s[C,D]$ (notation: $|AB|=2\cdot |CD|$) if
there is a point $M$ such that $|AM|=|MB|=|CD|$.
$M$ is called a \textbf{midpoint} between $A$ and $B$.
\end{Definition}

\begin{Lemma}\label{MidpointLemmaCircles}
Let $c$ be an Erlangen circle. \\
a. For every two non-antipodal points $A$ and $B$ of $c$ there
is exactly one midpoint of $s[A,B]$, i.e. a point $M\in s[A,B]$ satisfying
$|AM|=|MB|$.\\
b. For every two antipodal points $A$ and $B$ of $c$ there
are exactly two midpoints of $s[A,B]$, i.e. points $M\in c$ satisfying
$|AM|=|MB|$.
\end{Lemma}
\begin{proof}
Obvious if $A=B$, so assume $A\ne B$.

b. follows from the previous section as in the case of antipodal $A$ and $B$, $A$ and $M$ form a right triangle if $M$ is the midpoint of $A$ and $B$.

a. Pick an isomorphism $\rho$ sending $ray[A,B]$ onto $ray[B,A]$. It sends $s[A,B]$ to itself, so $\rho$ has a fixed point $M$ in $s(A,B)$. This is the midpoint we were looking for.
There are no more midpoints as they would be fixed points of $\rho$ which is not the identity.
\end{proof}

\begin{Lemma}\label{DividingLemmaCircles}
Let $c$ be an Erlangen circle and let $n \ge 1$ be a natural number.
For each measure $L$ of line segments there is a unique measure $L_n$ of line segments
such that $2^n\cdot L_n=L$.
\end{Lemma}
\begin{proof}
The proof for $L$ being the measure of a connected line segment follows from \ref{DividingLemma}, so assume $L$ is the measure of $s[A,a(A)]$ for some $A\in c$.
Existence of $L_n$ follows from \ref{MidpointLemmaCircles}. Uniqueness of $L_n$ follows from the fact $2^{n-1}\cdot L_n$ is the length of a right angle.
\end{proof}

\section{Isometries of Erlangen circles}

Following Euclid, we consider the measure of a right angle as the unit of measure for circles.
One can easily switch to degrees by multiplying by $90$ which we will quite often do.

\begin{Proposition}\label{IsomorphismBetweenCircles}
Let $f:c_1\to c_2$ be a bijection between Erlangen circles preserving closed line segments. If $f^{-1}\circ \rho\circ f\in \mathcal{I}_1$ for any isomorphism
$\rho\in \mathcal{I}_2$, then $f$ preserves measures of closed line segments.
\end{Proposition}
\begin{proof}
Notice $f^{-1}\circ a_2\circ f=a_1$, where $a_i$ is the antipodal map of the circle $c_i$, $i=1,2$. It is so since $f(s[A,a_1(A)])$ contains $f(A)$, is a closed line segment, and contains exactly two points.

Suppose $\rho$ is the translation from $A$ to $B$ in $c_2$ such that $A$ and $B$ form a right angle, i.e. $|AB|=1$. Notice $\rho^2=a_2$ (see \ref{RightAngleAlgebraically})
Therefore $(f^{-1}\circ \rho\circ f)^2= f^{-1}\circ \rho^2\circ f= f^{-1}\circ a_2\circ f=a_1$
and it means that $f^{-1}\circ \rho\circ f$ is a translation in $c_1$ by a right angle
(see \ref{RightAngleAlgebraically}).

Pick a right angle $A$, $B$ in $c_1$. Let $\rho$ be the translation from $f(A)$ to $f(B)$ in $c_2$. $ f^{-1}\circ \rho\circ f$ sends $A$ to $B$ and does not have any fixed point $C$ as in that case $f(C)$ would be a fixed point of $\rho$. 
Therefore $ f^{-1}\circ \rho\circ f$ is the translation from $A$ to $B$ (see 
\ref{StructureOfIsoCircle}) and its square is $a_1$. Therefore $\rho^2=a_2$
and $f(A)$, $f(B)$ form a right angle.

It remains to show that congruence of line segments $s[C,D]$ and $s[E,F]$ in $c_1$
implies congruence of line segments $s[f(C),f(D)]$ and $s[f(E),f(F)]$ in $c_2$.
Suppose line segments $s[f(C),f(D)]$ and $s[f(E),f(F)]$ are not congruent in $c_2$.
Without loss of generality, we may assume there is an isomorphism $\tau$ in $\mathcal{I}_2$ sending the line segment $s[f(C),f(D)]$ onto a proper subset of $s[f(E),f(F)]$. In that case $ f^{-1}\circ \rho\circ f$ sends $s[C,D]$ onto a proper subset of $s[E,F]$, a contradiction.
\end{proof}

\begin{Theorem}
Let $(c,\mathcal{I})$ be an Erlangen circle. The following conditions are equivalent
for any function $f:c\to c$:\\
a. $f\in \mathcal{I}$,\\
b. $f:c\to c$ preserves measures of closed line segments.
\end{Theorem}
\begin{proof}
Similar to that of \ref{IsometriesAreIsomorphismsLine}.
\end{proof}

\begin{Theorem}
Let $c_1, c_2$ be two Erlangen circles. Given a point $A$ of $c_1$ and a point $B$ of $c_2$
there is an isometry $f:c_1\to c_2$ such that
$f(A)=B$.
\end{Theorem}
\begin{proof}
Similar to \ref{CharOfIsoOfErlangenLines}.
\end{proof}

\section{Planes}

We define planes analogously to the lines. Another view of planes is that of Pasch spaces with connected ine segments whose all boundries at infinity are circles.

\begin{Definition}\label{PlaneDef}
A \textbf{plane} is a space $\Pi$ with maximal rays satisfying the following conditions:\\
a. $\Pi$ contains at least two points,\\
b. for each pair of different points $A$ and $B$ of $\Pi$, 
the line segment $s[A,B]$ is connected and the complement $\Pi\setminus l(A,B)$ can be expressed as the union of two disjoint non-empty convex sets $h_1$ and $h_2$ such that $s[C,D]$ intersects $l(A,B)$ at exactly one point whenever $C\in h_1$ and $D\in h_2$. 
\end{Definition}

\begin{Definition}
Each of the components of the complement of $l(A,B)$ is called a \textbf{half-plane}
and $l(A,B)$ is called the \textbf{boundary} of each of them.
\end{Definition}

\begin{Example}
Any two-dimensional real vector space $\Pi$ is a plane. Indeed, we define betweenness in $\Pi$ as follows: $C$ is between $A$ and $B$
if there is $t\ge 0$ such that $A-C=t(B-C)$. Notice $l(A,B)=\{A+t(B-A)\mid t\in \mathbb{R}\}$ if $A\ne B$ and $\Pi\setminus l(A,B)$ has two components:
$h_1=\{A+t(B-A)+s\cdot C\mid s > 0 \mbox{ and } t\in \mathbb{R}\}$
and $h_2=\{A+t(B-A)+s\cdot C\mid s < 0 \mbox{ and } t\in \mathbb{R}\}$, where $C$ is a point of $\Pi$ such that $\{B-A,C-A\}$ form a basis of $\Pi$.
\end{Example}

\begin{Example}[Descartes plane]\label{Descartes plane}
\textbf{Descartes plane} is a special case of a two-dimensional real vector space. It is the set of all pairs $(x,y)$ of real numbers with the standard addition and standard multiplication by scalars.
\end{Example}

\begin{Example}[Klein plane]\label{Klein plane}
\textbf{Klein plane} is the interior of the unit circle in the Descartes plane. It inherits betweenness from the Descartes plane and therefore lines in that model are intersections of Descartes lines with the interior of the unit circle.
\end{Example}

\subsection{Planes and Pasch spaces}

Let's start by showing the following:

\begin{Lemma} Let $\Pi$ be a regular Pasch space and $C\in s[A,B]$, $E\in s[A,D]$. If line segments $s[B,D]$ and $s[C,E]$ intersect at a point in their interior and $A$, $B$, $D$ do not lie on one line, then $B=C$ and $E=D$.
\end{Lemma}

\begin{proof} Assume $C\ne B$ and let $F$ be a point in the interior of $s[B,D]$ and in the interior of $s[C,E]$. Notice $C\ne A$ as in that case $A$, $B$, $D$ do lie on a ray emanating from $D$.

The line segment $s[A,F]$ intersects $s[C,D]$ at a point $G$ different from $F$ (otherwise $B=C$). Now, $s[A,G]$ intersects $s[C,E]$ at $H$ different from $F$. Therefore all five points $A$, $H$, $F$, $C$, $E$ lie one one line segment in the line $l(F,H)$. In conclusion, $B\in l(H,F)$ resulting in $D\in l(H,F)$, a contradiction.
\end{proof}

\begin{Lemma}
Let $\Pi$ be a space with maximal rays.
If it satisfies the Strong Pasch Condition (A maximal line intersecting one side of a triangle must intersect another side as well), then every line not containing any of vertices of a triangle that is not lying on a line either intersects two of its sides or none.
\end{Lemma}
\begin{proof}
Suppose $X$, $Y$, and $Z$ are different three points belonging to interior of each side of a triangle such that $Y$ is between $X$ and $Z$. Label the vertices of the triangle so that $X\in s[A,C]$, $Y\in s[B,C]$, and $Z\in s[A,B]$. Choose $D\in s[X,C]$ and $E\in s[B,Y]$.
The line $l(D,E)$ must intersect $s(X,Y)$ as it does not intersect $s[C,Y]$.

Therefore $l(D,E)$ must intersect $s(Z,B)$ as $l(D,E)$ does not intersect $s[Z,Y]$. $l(D,E)$ does intersect $s(X,Z)$ but it does not intersect any other side of the triangle $\Delta(AXZ)$, a contradiction.
\end{proof}

\begin{Lemma}
Let $\Pi$ be a space with maximal rays.
If it satisfies the Strong Pasch Condition (A maximal line intersecting one side of a triangle must intersect another side as well), then $\Pi$ is a plane.
\end{Lemma}
\begin{proof}
Suppose $l$ is a maximal line in $\Pi$. Given $A$ outside of $l$ define $\Sigma_A$ as the set of points $X\in \Pi\setminus \{A\}$ such that $s[A,X]$ does not intersect $l$. Notice $\Sigma_A$ is convex: $s[X,Y]$ misses $l$ if $X,Y\in \Sigma_A$ as $l$ cannot intersect only one side of the triangle $\Delta(AXY)$.
Also, if $B\in \Pi\setminus (l\cup \Sigma_A)$, then $\Sigma_A$ is disjoint from $\Sigma_B$ and $\Sigma_A\cup \Sigma_B=\Pi\setminus l$. Indeed, given $Y$ outside of $l$ exactly one of line segments $s[A,Y]$, $s[B,Y]$ intersects $l$ and $Y$ belongs to exactly one of the sets $\Sigma_A$, $\Sigma_B$. That proves $\Pi$ is a plane.
\end{proof}

\begin{Lemma}
If $\Pi$ is a plane, then it satisfies the Strong Pasch Condition (A maximal line intersecting one side of a triangle must intersect another side as well).
\end{Lemma}
\begin{proof}
Suppose $l$ is a line intersecting only one of the sides of a triangle $\Delta(ABC)$. In particular, none of the points $A$, $B$, $C$ lies on $l$, two of them lie on opposite sides of $l$, and the third one lies on the same side of $l$ as the other two. That is not possible. 
\end{proof}

\begin{Theorem}\label{CircleOfRaysThm}
Every plane $\Pi$ is a Pasch space and every boundary at infinity $\partial(\Pi,A)$ is a circle. 
\end{Theorem}
\begin{proof}
We need to verify Condition \ref{PaschCondition}: for every two maximal rays $r_1,r_2$ emanating from $A$ and every points $B,C\in r_1$ and $D\in \Pi$, if $r_2$ intersects $s[B,D]$, then it also intersects $s[C,D]$.

It is clearly so if $r_1=r_2$. If $r_2=a(r_1)$, then $D$ must belong to $r_2$ and $r_2$ intersects $s[C,D]$. Assume $r_2\ne r_1$, $r_2\ne a(r_1)$, and $D$ does not lie on $r_2$. $B$ and $D$ are on different sides of the line $r_2\cup a(r_2)$, and $B,C$ are on the same side of that line. Therefore $C$ and $D$
are on different sides and the worst case is that of $a(r_2)$ intersecting $s[C,D]$. That cannot happen. Indeed, in that case $A$ would be between two points, one on $s(B,D)$ and one on $s(C,D)$. However, all those points belong to one side of the line $r_1\cup a(r_1)$, a contradiction.

To conclude the proof we need to show that the complement of $\{r, a(r)\}$ in $\partial(\Pi,A)$
splits into two lines for any maximal ray $r$ at $A$.
The two components $\Sigma_1$ and $\Sigma_2$ of the complement of $r\cup a(r)$ in $\Pi$ lead naturally to splitting of the complement of $\{r, a(r)\}$ in $\partial(\Pi,A)$ into $l_1$ and $l_2$. Namely, $l_i$ is the set of maximal rays $r$ at $A$ such that $r\setminus\{A\}$ is contained in $\Sigma_i$ for $i=1,2$.

To show $l_1$ is a line, notice it is convex: given two points $B,C\in \Sigma_1$
the rays $ray[A,X]$, $X\in s[B,C]$ form the segment joining $ray[A,B]$ and $ray[A,C]$ and are contained in $l_1$. Each ray $r\in l_1$ splits $l_1$ into two components corresponding to the components of $\Pi\setminus (r\cup a(r))$.
\end{proof}

\begin{Theorem}\label{PaschSpacesAndPlanes1}
Let $\Pi$ be a Pasch space with all line segments being connected.
If the boundary at infinity $\partial(\Pi,A)$ is a circle for some $A$, then $\Pi$
is a plane.
\end{Theorem}
\begin{proof}
Suppose $l$ is a line in $\Pi$. We need to show $l$ separates $\Pi$ into two components of convexity. First, assume $A\notin l$. Define two disjoint subsets of $\Pi$
$$\Sigma_1:=\{B\in \Pi\setminus l \mid s[A,B]\cap l=\emptyset\}$$
and $$\Sigma_2:=\{B\in \Pi\setminus l \mid s[A,B]\cap l\ne\emptyset\}.$$
$\Sigma_2$ is convex.

 Indeed, if $C\in l$ is between $A$ and $B$,
and $C'\in l$ is between $A$ and $B'$, then for every $D\in s[B,B']$,
the line segment $s[A,D]$ intersects $s[C,C']$ by the Medium Pasch Condition, hence $D\in \Sigma_2$. 

$\Sigma_1$ is convex by contradiction: 
Suppose $B, C\in \Sigma_1$ but $l$ intersects $s(B,C)$ at $D$.  In the space $\partial(\Pi,A)$ the set
$ray[A,X]$, $X\in l$ forms a line that intersects the line segment
$s[ray[A,B],ray[A,C]]$ in its interior. Therefore we can find points $E, F\in s(B,C)$ and find points $E', F'\in l$ such that
$D\in s(E,F)$,
$ray[A,F]=ray[A,F']$, and $ray[A,E]=ray[A,E']$.

By what we already know about Pasch spaces we may assume $E'\in s(A,E)$ and $E\in s(D,C)$. However, in that case $l$ intersects $s[A,C]$ (by Pasch Condition), a contradiction.

Suppose $B\in \Sigma_2$ and $C\in \Sigma_1$. Let $D\in l$ lie on
$s[A,B]$. Suppose $s[B,C]$ does not intersect $l$. The rays $ray[A,X]$, $X\in l$ form a line $L$ in $\partial(\Pi,A)$ containing $ray[A,B]$.
The rays $ray[A,Y]$, $Y\in s[B,C]$ form a line segment $S$ in $\partial(\Pi,A)$ emanating from $ray[A,B]$, hence intersecting $L$. 

Hence there is a an interior point of $S$ of the form
$r'=ray[A,E]$, $E\in l$. In particular, $r'$ intersects $s[B,C]$ at some interior point $F$. If $F\in s[A,E]$, then the Weak Pasch Condition says $s[B,F]$ and $s[D,E]$ intersect implying $l$ intersects $s[B,C]$, a contradiction.

Thus, $E\in s(A,F)$. Look at the intersection $G$ of $s[D,C]$
and $s[A,F]$. All three possibilities: $G=E$, $G\in s[A,E)$,
$G\in s[F,E)$ lead to intersection of $l$ with either $s[A,C]$ 
or $s[B,C]$ by applying Pasch Condition. A contradiction!
The case of $A\in l$ is simpler. Look at two components of $\partial(\Pi,A)$ minus rays in $l$. They give rise to two components of $\Pi$ minus $l$.
\end{proof}

\begin{Theorem}\label{PaschSpacesAndPlanes2}
Let $X$ be a Pasch space with all line segments being connected.
Given two different maximal rays $r_1$ and $r_2$ emanating from $A$ and not antipodal, the union $\Pi(r_1,r_2)$ of all maximal rays contained in $$s[r_1,r_2]\cup s[a(r_1),r_2]\cup s[r_1,a(r_2)]\cup s[a(r_1),a(r_2)]$$
is a plane.
\end{Theorem}
\begin{proof}
It suffices to show $\Pi(r_1,r_2)$ is convex in a strong sense:
given $B,C\in \Pi(r_1,r_2)$, $l(B,C)\subset \Pi(r_1,r_2)$ if $B\ne C$.
Indeed, in that case $\Pi(r_1,r_2)$ is a Pasch space whose boundary at infinity is a circle by 
\ref{SubCirclesOfPaschSpaces}. Using \ref{PaschSpacesAndPlanes1}
one concludes $\Pi(r_1,r_2)$ is a plane.

First, consider the case of $B\in r_1$ and $C\in r_2$. If $B=A$ or $C=A$, then
clearly $l(B,C)\subset \Pi(r_1,r_2)$.
Assume $B\ne A$ and $C\ne A$.
Use 
 \ref{PaschBetweennessLemma} to conclude $l(B,C)\subset \Pi(r_1,r_2)$.

Applying  \ref{PaschBetweennessLemma} one can see $\Pi(r_1,r_2)=\Pi(r_1,r'_2)$
if $r_2'\in s[r_1,r_2]\cup s[a(r_1),r_2]\cup s[r_1,a(r_2)]\cup s[a(r_1),a(r_2)]$.
Therefore $\Pi(r_1,r_2)=\Pi(r'_1,r'_2)$ if
$r_1',r_2'\in
s[r_1,r_2]\cup s[a(r_1),r_2]\cup s[r_1,a(r_2)]\cup s[a(r_1),a(r_2)]$
are not antipodal. Since for every two different points $B,C\in \Pi(r_1,r_2)$,
that do not lie on the same line with $A$, there are $r_1',r_2'\in
s[r_1,r_2]\cup s[a(r_1),r_2]\cup s[r_1,a(r_2)]\cup s[a(r_1),a(r_2)]$
containing $B$ and $C$, respectively, we are done.
\end{proof}

\subsection{Parallel lines}

\begin{Definition}\label{ParallelLinesDef}
Two lines $l(A,B)$ and $l(C,D)$ on a plane $\Pi$ are \textbf{parallel} if either they are equal or they do not intersect each other.
\end{Definition}

\begin{Definition}\label{ViewingAngleDef}
Given a point $P$ and a line $l(Q,R)$ on a plane $\Pi$ consider the set of maximal rays $ray[P,X]$, where $X\in l(Q,R)$. If $P\in l(Q,R)$, it consists of two antipodal points $r_1, r_2$ in the circle $c_P$ of rays emanating from $P$. If $P\notin l(Q,R)$,
it forms a line in $c_P$ which is of the form $s(r_1,r_2)$ for some maximal rays $r_1$ $r_2$. In both cases we call $s[r_1,r_2]$ the \textbf{viewing angle} of $l(Q,R)$ from $P$.
\end{Definition}

\begin{Definition}
An angle $s[r_1,r_2]$ in the circle of rays $c_P$, where $P$ is a point of a plane $\Pi$ is \textbf{full} if $r_1$ is antipodal to $r_2$.
\end{Definition}

\begin{Proposition}\label{ExistenceOfParallelLinesProp}
For every point $P$ and for every line $l(Q,R)$ of a plane $\Pi$ there is a maximal line passing through $P$ and parallel to $l(Q,R)$. That line is unique if and only if the viewing angle of $l(Q,R)$ from $P$ is full.
\end{Proposition}
\begin{proof}
It is clearly so if $P$ belongs to $l(Q,R)$. Assume $P\notin l(Q,R)$ and let $s[r_1,r_2]$ be the viewing angle of $l(Q,R)$ from $P$. Notice the maximal line extending $r_1$ is parallel to $l(Q,R)$ and the same is true of $r_2$. The two lines are equal if and only if $r_1$ is antipodal to $r_2$, i.e. the viewing angle is full.
\end{proof}

\begin{Remark}
Given a plane $\Pi$ and given two different points $A,B\in \Pi$ there is a canonical map $f:\partial(\Pi,A)\to \partial(\Pi,B)$ defined as follows:
for every maximal ray $r$ emanating from $A$ not containing $B$ the rays $ray[B,X]$, $X\in r$, form a ray in $ \partial(\Pi,B)$, so it "ends" at some point
$f(r)$. In case of $r=ray[A,B]$ we declare $f(r)$ to be the subset of $r$
and we declare $f(a(r))$ to be $a(f(r))$. Notice $f$ preserves betweenness but
it preserves antipodal pairs of rays if and only if there is uniqueness of lines containing $B$ parallel to lines containing $A$. One can view this observation the following way: assigning $\partial(\Pi,A)$ to $A$ is functorial if and only if $\Pi$ satisfies Axiom 5 of Euclid.
\end{Remark}

\section{Erlangen planes}

\begin{Definition}
An \textbf{Erlangen plane} is a pair $(\Pi,\mathcal{I})$ consisting of a plane $\Pi$ equipped with a subgroup $\mathcal{I}$ of isomorphisms of $\Pi$
such that the Condition \ref{Homogeneity and rigidity of the plane} is satisfied.
\end{Definition}

\begin{Condition}[Homogeneity and rigidity of the Erlangen plane]\label{Homogeneity and rigidity of the plane}
For every two pairs $(h_1,r_1)$, $(h_2,r_2)$ consisting of a half-plane and a maximal ray in its boundary there is exactly one isomorphism sending $h_1$ onto $h_2$ and sending $r_1$ onto $r_2$.
\end{Condition}

\begin{Example}[Hilbert plane] 
The \textbf{Hilbert plane} is the Descartes plane with the set of isomorphisms consisting of functions of the form $f(x,y)=h(x,y)+(a,b)$, where $a$, $b$ are real constants
and $h$ is a linear transformation preserving the quadratic from $x^2+y^2$.

Indeed, any maximal ray emanating from $(x_0,y_0)$ can be expressed as the set of points
$(x_0,y_0)+t\cdot (x_1,y_1)$, where $t\ge 0$ and $(x_1,y_1)$ is the direction of the ray (in particular, the $x_1^2+y_1^2$ is $1$). One can map such a ray onto non-negative reals via the isomorphism $f(x,y):=(x_1\cdot x+y_1\cdot y,-y_1\cdot x+x_1\cdot y) -(x_0,y_0)$. The only two isomorphisms preserving non-negative reals are the identity and the map $(x,y)\to (x,-y)$.
\end{Example}

\begin{Example}[Complex plane] 
The \textbf{complex plane} is the set of all complex numbers $z$ with the betweenness relation induced by its structure of a $2$-dimensional real vector space.
The set of isomorphisms consists of functions of one of the two forms: $f(z)=(a+bi)\cdot z+w$
or $f(z)=(a+bi)\cdot \bar z+w$,
where $a$, $b$ are real constants, $a^2+b^2=1$, and $w$ is a complex constant.

Indeed, any maximal ray emanating from $z_0$ can be expressed as the set of points
$z_0+t\cdot z_1$, where $t\ge 0$ and $z_1$ is the direction of the ray (in particular, the modulus of $z_1$ is $1$). One can map such a ray onto non-negative reals via the isomorphism $f(z):=\frac{z-z_0}{z_1}$. The only two isomorphisms preserving non-negative reals are the identity and the conjugation.
\end{Example}

\begin{Corollary}\label{InvolutionAndIdPlanes} Let $(\Pi,\mathcal{I})$ be an Erlangen plane.\\
a. An isomorphism in $\mathcal{I}$ that has a fixed point $A$ and sends a ray emanating from $A$ to itself is an involution. \\
b. An isomorphism in $\mathcal{I}$ that has two fixed points is an involution. \\
c. An isomorphism in $\mathcal{I}$ that has three non-collinear fixed points is the identity. 
\end{Corollary}
\begin{proof}
a. and c. Suppose $\rho\in \mathcal{I}$ fixes $A$ and sends a ray at $A$ to itself.
Therefore it sends a maximal ray $r$ emanating from $A$ to itself. Indeed, $\rho(r)$ is a maximal ray at $A$ by \ref{BijectionsBetweenSpacesMaxRays} 
so it must be equal to $r$. Therefore $\rho(l)=l$ as well, $l$ being the maximal line containing $r$.
Consider the half-planes $h_1$ and $h_2$ that arise from components of $\Pi\setminus l$. Notice $\rho(l)=l$ and either $\rho(h_1)=h_1$ and $\rho(h_2)=h_2$ (this is so if $\rho$ has a fixed point in either $h_1$ or $h_2$) and $\rho$ must be the identity or $\rho(h_1)=h_2$ and $\rho(h_2)=h_1$ in which case $\rho^2$ preserves both $r$ and $h_1$ and must therefore be equal to the identity.\\
b. follows from a. as $\rho$ preserves $ray[A,B]$ if $A$, $B$ are fixed points of $\rho$.
\end{proof}

\begin{Corollary}\label{EPAreRM}
If $(\Pi,\mathcal{I})$ is an Erlangen plane, then $(\Pi,\mathcal{I})$ is a space with rigid motions.
\end{Corollary}
\begin{proof}
Let $A\ne B$ be two points of $\Pi$.
Suppose $f(s[A,B])\subset s[A,B]$ for some $f\in \mathcal{I}$. By \ref{FixedPointTheorem} $f$ has a fixed point $M$. If $M$ is one of the endpoints of 
$s[A,B]$, say $M=A$, then $f$ preserves the ray $ray[A,B]$
and by \ref{InvolutionAndIdPlanes}
$f^2=id$. If $M\in s(A,B)$, then either $f$ preserves $ray[M,A]$ and $f^2=id$
or $f$ flips $ray[M,A]$ onto $ray[M,B]$ and $f^2$ preserves $ray[M,A]$ resulting in $f^4=id$.
By \ref{NilpotentCaseOfRM} $(\Pi,\mathcal{I})$ is a space with rigid motions.
\end{proof}

\begin{Theorem}\label{EPVsRM}
Let $\Pi$ be a plane and let $\mathcal{I}$ be a subgroup of isomorphisms of $\Pi$.
$(\Pi,\mathcal{I})$ is an Erlangen plane if and only if the following conditions are satisfied:\\
a. $(\Pi,\mathcal{I})$ is a space with rigid motions,\\
b. there is a maximal line $l_0$ in $\Pi$ and $f_0\in \mathcal{I}$ fixing $l_0$
and flipping the half-planes whose boundary is $l_0$,\\
c. any $f\in \mathcal{I}$ fixing $l_0$ is either identity or is equal to $f_0$.
\end{Theorem}
\begin{proof}
In view of \ref{EPAreRM} any Erlangen plane satisfies conditions a.-c., hence it suffices to show the converse.

Given two maximal rays $r_1$, $r_2$ in $l_0$ there is $g\in \mathcal{I}$ sending
$r_1$ onto $r_2$ due to $(\Pi,\mathcal{I})$ being a space with rigid motions.
If $g$ preserves/flips the half-planes of $l_0$, then $f_0\circ g$ flips/preserves those half-planes. The same can be said of any two pairs $(h_1,r_1)$, $(h_2,r_2)$ consisting of a half-plane and a maximal ray in its boundary as $r_1$, $r_2$ can be sent to $l_0$
via some elements of $\mathcal{I}$. The only remaining issue is to show that
any $f\in \mathcal{I}$ preserving a pair $(h_1,r_1)$ must be the identity. As above we reduce it to $r_1\subset l_0$ in which case $f|l_0$ must be the identity. As $f$ cannot be equal to $f_0$, $f=id$.
\end{proof}

\begin{Definition}
Let $(\Pi,\mathcal{I})$ be an Erlangen plane. If $A,B\in\Pi$, then
the \textbf{translation} $\tau_{AB}$ is the identity if $A=B$. Otherwise,
it sends $ray[A,B]$ to the maximal ray emanating from $B$ contained in $ray[A,B]$ and it preserves both
half-planes of $l(A,B)$.

Given two maximal rays $r_1$ and $r_2$ at a point $A$, the \textbf{rotation}
from $r_1$ to $r_2$ is the identity if $r_1=r_2$. Otherwise, it is the isomorphism of $\mathcal{I}$ sending $r_1$ onto $r_2$ and sending the half-plane of $r_1$ containing $r_2$ onto the half-plane of $r_2$ not containing $r_1$.

Given a maximal line $l$ in $\Pi$ the \textbf{reflection} in $l$ is the isomorphism of $\mathcal{I}$ fixing $l$ and swapping the two half-planes of $l$.

A \textbf{glide reflection} is the composition of a non-trivial translation $\tau_{AB}$
with the reflection in the line $l(A,B)$.
\end{Definition}

\subsection{Length of line segments}

Since Erlangen planes are spaces with rigid motions, we have the concept of congruence of line segments and the concept of length of line segments.
Similarly as in the case of Erlangen lines one can see that lengths can be added, multiplied by positive reals, divided, and a smaller length can be subtracted from a larger length.
\subsection{Angles}

The boundary at infinity $\partial(\Pi,A)$ at $A$, i.e. the maximal rays emanating from a given point $A$ of a plane $\Pi$, form a circle by \ref{CircleOfRaysThm}. By an \textbf{angle} at $A$ we mean a line segment of $\partial(\Pi,A)$.

\begin{Theorem}\label{CircleOfAnglesErlangenPlane}
For any point $A$ of an Erlangen plane $(\Pi,\mathcal{I})$ the boundary at infinity $\partial(\Pi,A)$ is an Erlangen circle when equipped with isomorphisms induced by $\mathcal{I}$. All isomorphisms in $\mathcal{I}$ preserve measures of angles. 
\end{Theorem}
\begin{proof}
Consider all isomorphisms $\rho$ in $\mathcal{I}$ preserving $A$. 
$\rho$ induces the bijection of $\partial(\Pi,A)$, namely $r\to \rho(r)$, so we will denote it by the same letter
$\rho:\partial(\Pi,A)\to \partial(\Pi,A)$ and the set of those bijections of $\partial(\Pi,A)$ will be denoted by $\mathcal{I}_A$.
Notice $\rho:\partial(\Pi,A)\to \partial(\Pi,A)$ preserves closed line intervals. Observe that components of $\partial(\Pi,A)\setminus \{r,a(r)\}$ in $\partial(\Pi,A)$ correspond precisely to half-planes with boundary $r\cup a(r)$ in $\Pi$.
Therefore for every two maximal rays $r_1$, $r_2$ at $A$ and for every two maximal rays $l_1$ and $l_2$ in $c_A$ emanating from $r_1$ and $r_2$, respectively, there is a unique $\rho$ sending $l_1$ onto $l_2$. That proves $\partial(\Pi,A)$ is an Erlangen circle.

Suppose $f:\Pi\to\Pi$ belongs to $\mathcal{I}$ and $f(A)=B$. Our plan is to use \ref{IsomorphismBetweenCircles} in order to conclude that $f$ preserves measures of angles. To accomplish that simply realize that given $\rho\in \mathcal{I}_B$
the composition $f\circ \rho\circ f^{-1}:\partial(\Pi,B)\to \partial(\Pi,B)$ corresponds to
the isomorphism of $\partial(\Pi,B)$ induced by $f\circ \rho\circ f^{-1}\in \mathcal{I}$. As such it does belong to $\mathcal{I}_B$ and \ref{IsomorphismBetweenCircles} applies.
\end{proof}

\begin{Corollary}
Two rotations about a point $A$ of an Erlangen plane are equal or are inverse to each other if and only if the measures of angles between corresponding maximal rays are equal.
\end{Corollary}
\begin{proof}
Translations in the circle $\partial(\Pi,A)$ of angles at $A$ correspond to rotations at $A$ of the whole plane. As in \ref{DividingLengthViaTranslationsProp} equality of measures of angles means that the corresponding translations are either equal or are inverses of each other.
\end{proof}

\begin{Theorem}\label{CharErlangenPlanesViaBoundary}
Suppose $(\Pi,\mathcal{I})$ is a space with rigid motions. If $\Pi$ is a plane and for each point $A\in \Pi$ the induced space $(\partial(\Pi,A),\mathcal{I}_A)$ is an Erlangen circle,
where $ \mathcal{I}_A$ consists of all elements of $ \mathcal{I}$ fixing $A$,
then $(\Pi,\mathcal{I})$ is an Erlangen plane.
\end{Theorem}
\begin{proof}
Given a maximal ray $r$ emanating from $A$ there is $f\in \mathcal{I}_A$ preserving $A$ and flipping components of  $\partial(\Pi,A)\setminus \{r,a(r)\}$. That $f$ preserves $r$ and flips components of $\Pi\setminus (r\cup a(r))$. Any other $g\in \mathcal{I}_A$
doing the same results in $h:=g^{-1}\circ f$ preserving all rays emanating from $A$. Due to length considerations, $h$ must be the identity.
\end{proof}

\subsection{Translations of Erlangen planes}

\begin{Proposition}\label{TranslationsAreTwoSymmetries}
For any two points $A$ and $B$ of an Erlangen plane $\Pi$ the translation
$\tau_{AB}$ equals the composition of the reflection $i_M$ in the midpoint $M$ of $s[A,B]$ followed by the reflection $i_B$ in $B$.
\end{Proposition}
\begin{proof}
It is so if $A=B$, hence assume $A\ne B$. Notice $i_B\circ i_M$ preserves each component of $\Pi\setminus l(A,B)$ as each of $i_M$, $i_B$ flips them. The same reasoning applies to maximal rays of $l(A,B)$ if we orient $l(A,B)$ by picking $ray[A,B]$ as a positive maximal ray. Each of $i_M$, $i_B$ flips the orientation of $l(A,B)$, so their composition preserves it. In addition, $i_B\circ i_M(A)=B$, so $i_B\circ i_M$ must be equal to $\tau_{AB}$.
\end{proof}

\begin{Corollary}\label{NoFixedPointsForTranslations}
A translation of an Erlangen plane that has a fixed point is the identity.
\end{Corollary}
\begin{proof}
Suppose a non-trivial translation has a fixed point $C$. In view of \ref{TranslationsAreTwoSymmetries} it means existence of two different points $A$ and $B$ such that $i_A(C)=i_B(C)$. But $A$ is the midpoint of $s[C,i_A(C)]$ and $B$ is the midpoint of $s[C,i_B(C)]$, hence $A=B$, a contradiction.
\end{proof}

\subsection{Defect of triangles}

Given a triangle $\Delta(ABC)$ (i.e. three different points $A$, $B$, and $C$ that are called the \textbf{vertices} of the triangle) on an Erlangen plane we can talk about angles of the triangle at each of its vertices. In case of $A$ it means we are talking about the angle between maximal rays $ray[A,B]$ and $ray[A,C]$. 
\begin{Example}\label{AnglesOfFlatTriangle}
If all three points $A$, $B$, and $C$ lie on a line, then one of the angles of the triangle 
$\Delta(ABC)$ is $180$ degrees and the other two are $0$ degrees. The angle of $180$ degrees corresponds to the vertex that is between the other two.
\end{Example}

\begin{Proposition}\label{DefectIsNonNegativeProp}
Given three non-collinear points on an Erlangen plane, the sum of the angles of the triangle formed by the points is at most $180$ degrees.
\end{Proposition}
\begin{proof}
Suppose there is a triangle whose sum of angles is more than $180$ degrees.
By flipping it about the midpoint of the side corresponding to the smallest angle we create a quadrilateral and the
new triangle consisting of one diagonal and two sides that has the same sum of angles but the largest angle is the sum of the two largest angles of the original triangle.
That shows that particular sum can never be larger than $180$ degrees.

 By repeating the process a few times we reach the stage where the sum of two largest angles of a triangle is more than 180. This is not possible as we have already observed.
\end{proof}

\begin{Definition}\label{DefectOfTriangleDef}
Given three non-collinear points on an Erlangen plane, the \textbf{defect} of the triangle formed by the points is $180$ degrees minus the sum of angles of the triangle.
By \ref{DefectIsNonNegativeProp} the defect is always non-negative.
\end{Definition}

The fact the defect of a triangle is always non-negative allows for a standard proof of the following result.

\begin{Corollary}[Triangle Inequality]\label{TriangleInequality}
The largest side of a triangle $\Delta(ABC)$ is opposite to the largest angle of the triangle. Therefore $|AB|+|BC| > |AC|$.
\end{Corollary}
\begin{proof}
Suppose $|AC|$ is the largest side of the triangle. Pick $D\in s[A,C]$ satisfying $|AD|=|AB|$. Notice the angle at $B$ of $\Delta(ABD)$ is equal to the angle at $D$ of that triangle (use reflection in the ray bisecting the angle at $A$) which in turn is bigger than the angle of $\Delta(ABC)$ at $C$. Thus, the angle at $B$ is larger than any other angle.

To conclude $|AB|+|BC| > |AC|$ note that the only interesting case is that of $|AC|$ being the largest side. On the maximal ray at $B$ in the line $l(A,B)$ choose the point $D$ on the other side of $B$ such that $|BD|=|BC|$. Notice the angle at $C$ of $\Delta(ACD)$ is larger than the angle at $D$.
\end{proof}

\begin{Corollary}
Let $\Pi,\mathcal{I})$ be an Erlangen plane. If $|AB|+|BC|= |AC|$ for some points $A,B,C$ of $\Pi$, then $B$ is between $A$ and $C$.
\end{Corollary}

\begin{Proposition}\label{DefectZeroCase}
Let $\Pi,\mathcal{I})$ be an Erlangen plane. If there is one non-flat triangle in $\Pi$
with the defect equal $0$, then every triangle has defect equal $0$.
\end{Proposition}
\begin{proof}
By flipping a non-collinear triangle over the midpoints of its sides we create a triangle twice as big and with the defect $0$. We can repeat it until it swallows a given triangle. Since defect is additive, the given triangle must have the defect $0$ as well.
\end{proof}

\begin{Proposition}\label{DefectViewingAngleConnectionDef}
Let $\Pi,\mathcal{I})$ be an Erlangen plane. If $P\notin l(Q,R)$, then the supremum of defects of triangles $\Delta(PAB)$, $A,B\in l(Q,R)$, equals $180$ degrees minus the viewing angle of $l(Q,R)$ from $P$.
\end{Proposition}
\begin{proof}
The viewing angle (more precisely, the measure of it) of $l(Q,R)$ from $P$ equals 
the supremum of angles at $P$ of triangles $\Delta(PAB)$. Making the other angles approaching $0$ one arrives at the conclusion of \ref{DefectViewingAngleConnectionDef}.
\end{proof}

\begin{Proposition}\label{CompositionOfTranslationsProp}
Given three non-collinear points $A$, $B$, and $C$ on an Erlangen plane, the composition $\rho:=\tau_{CA}\circ \tau_{BC}\circ \tau_{AB}$ of translations is the rotation at $A$ by the defect of the triangle $\Delta(ABC)$ in the direction away from the triangle $\Delta(ABC)$.
\end{Proposition}
\begin{proof}
Look at the line $l(A,B)$ and see where it arrives after the composition of the first two translations $\tau_{AB}$ and $\tau_{BC}$. $\tau_{AB}$ preserves $l(A,B)$
and $\tau_{BC}$ sends $l(A,B)$ to the line at $C$ so that the image of $ray[B,A]$ forms the angle $\alpha+defect$ with $ray[C,A]$, where $\alpha$ is the angle of $\Delta(ABC)$ at $A$. Therefore $\tau_{CA}\circ \tau_{BC}\circ \tau_{AB}$ rotates $l(A,B)$ by the defect of the triangle in the direction away from the triangle $\Delta(ABC)$. Apply the same reasoning to $l(A,C)$ to see it is rotated by the defect in the direction towards $ray[A,B]$.
Since $\rho$ coincides with rotation by the defect on three non-collinear points, it is equal to that rotation.
\end{proof}

\subsection{Congruence of triangles in Erlangen planes}

When talking about congruence of triangles in an Erlangen plane one needs to think of triangles as ordered triples of points.
\begin{Definition}\label{CongruenceOfTrianglesDef}
Two triangles $\Delta(ABC)$ and $\Delta(A'B'C')$ on an Erlangen plane $(\Pi,\mathcal{I})$ are \textbf{congruent} if there is $f\in\mathcal{I}$ satisfying $f(A)=A'$, $f(B)=B'$, and $f(C)=C'$.
\end{Definition}

\begin{Theorem}[SAS in Erlangen planes]\label{SAS in Erlangen planes}
Two triangles $\Delta(ABC)$ and $\Delta(A'B'C')$ on an Erlangen plane $(\Pi,\mathcal{I})$ are congruent if and only if $|AB|=|A'B'|$, $|AC|=|A'C'|$, and the measure of the angle of $\Delta(ABC)$ at $A$ equals the measure of the angle of $\Delta(A'B'C')$ at $A'$.
\end{Theorem}
\begin{proof}
By applying an $f\in \mathcal{I}$ sending $s[A,B]$ onto $s[A',B']$ one reduces the proof to the case $A=A'$ and $B=B'$. By applying the reflection in $l(A,B)$, if necessary,
one reduces the proof to the case of $C$ and $C'$ being on the same side of $l(A,B)$. However, in that case $C'$ must be equal to $C$ and we are done.
\end{proof}

\begin{Theorem}[ASA in Erlangen planes]\label{ASA in Erlangen planes}
Two triangles $\Delta(ABC)$ and $\Delta(A'B'C')$ on an Erlangen plane $(\Pi,\mathcal{I})$ are congruent if and only if $|AB|=|A'B'|$, the measure of the angle of $\Delta(ABC)$ at $A$ equals the measure of the angle of $\Delta(A'B'C')$ at $A'$, and the measure of the angle of $\Delta(ABC)$ at $B$ equals the measure of the angle of $\Delta(A'B'C')$ at $B'$.
\end{Theorem}
\begin{proof}
Same as that of \ref{SAS in Erlangen planes}.
\end{proof}

\begin{Theorem}[SSS in Erlangen planes]\label{SSS in Erlangen planes}
Two triangles $\Delta(ABC)$ and $\Delta(A'B'C')$ on an Erlangen plane $(\Pi,\mathcal{I})$ are congruent if and only if $|AB|=|A'B'|$, $|AC|=|A'C'|$, and $|BC|=|B'C'|$.
\end{Theorem}
\begin{proof}
Assume $|AB|$ is the largest side of the triangle $\Delta(ABC)$.
By applying an $f\in \mathcal{I}$ sending $s[A,B]$ onto $s[A',B']$ one reduces the proof to the case $A=A'$ and $B=B'$. By applying the reflection in $l(A,B)$, if necessary,
one reduces the proof to the case of $C$ and $C'$ being on different sides of $l(A,B)$.

Let $M$ be the point on $s[C,C']$ such that the line $l(A,M)$ bisects the angle at $A$ of the triangle $\Delta(ACC')$. Let $N$ be the point on $s[CC']$ such that the line $l(B,N)$ bisects the angle at $B$ of the triangle $\Delta(BCC')$. The composition of reflections in lines $l(A,M)$ and $l(B,N)$ fixes both $C$ and $C'$, hence it is either the identity, in which case $M=N$ and the triangles $\Delta(ABC)$ and $\Delta(ABC')$ are congruent,
or it is the reflection in the line $l(C,C')$, in which case $M=N$ again (otherwise
the composition does not fix neither $M$ nor $N$).
\end{proof}

\begin{Theorem}[AAS in Erlangen planes]\label{AAS in Erlangen planes}
Two triangles $\Delta(ABC)$ and $\Delta(A'B'C')$ on an Erlangen plane $(\Pi,\mathcal{I})$ are congruent if and only if $|AB|=|A'B'|$, $\angle(BAC)=\angle(B'A'C')$,
and $\angle(ACB)=\angle(A'B'C')$.
\end{Theorem}
\begin{proof}
It suffices to consider the case of $A=A'$, $B=B'$, and $C'\in s[A,C]$.
If $\angle(C'BC)$ is non-zero, then $\angle(AC'B)$ is bigger than the angle $\angle(ACB)$ (due to the defect of $\Delta(BC'C)$ being non-negative), a contradiction.
Therefore, $C'=C$ and we completed the proof.
\end{proof}

\begin{Theorem}[AAA in Erlangen planes]\label{AAA in Erlangen planes}
Two triangles $\Delta(ABC)$ and $\Delta(A'B'C')$ on an Erlangen plane $(\Pi,\mathcal{I})$ with positive defects are congruent if and only if their corresponding angles are equal.
\end{Theorem}
\begin{proof}
It suffices to consider the case of $A=A'$ and $B'\in s[A,B]$. In that case $C'\in s[A,C]$.
Indeed, if $C\in s(A,C')$, then line segments $s[B,C]$ and $s[B',C']$
intersect at an interior point $M$ (use \ref{CircleOfRaysThm}) in which case $\angle(ACB) > \angle(A'C'B')$, a contradiction.
If $C'\ne C$, then the defect of the triangle $\Delta(BCC')$ is $0$ due to additivity of defects and the fact the defect of $\Delta(ABC)$ equals the defect of $\Delta(A'B'C')$.
That contradicts \ref{DefectZeroCase}.
Therefore, $C'=C$. Similarly, $B'=B$ and we completed the proof.
\end{proof}

\subsection{Characterizing rotations and reflections in lines of an Erlangen plane}

Here we prove a characterization of rotations and reflections in lines of an Erlangen plane in an exceptionally simple way when compared to the classical approaches. 

\begin{Proposition}\label{StructureOfIsoPlane}
Let $(\Pi,\mathcal{I})$ be an Erlangen plane and $f\in \mathcal{I}$ is not equal to the identity.\\
1. $f$ is a rotation about a point if and only if it has a unique fixed point. \\
2. $f$ is a reflection in a line if and only if it has at least two fixed points.
\end{Proposition}
\begin{proof}
1. Obviously, rotations about points have fixed points. Suppose $f(A)=A$ and look at the way $f$ acts on the circle $c_A$ of maximal rays at $A$. This action cannot have any fixed points, hence it is a translation and $f:\Pi\to\Pi$ is a rotation.\\
2. If $f$ has at leat two fixed points $A$ and $B$, it fixes the line $l(A,B)$. As $f$ is not the identity, it must swap the two components of the complement of $l(A,B)$ in $\Pi$.
Thus, $f$ is the reflection in $l(A,B)$.
\end{proof}

\section{Euclidean planes}

In contrast to Section \ref{Euclidean lines} we are not going to pursue explaining Euclidean planes using the concept of length. Instead, we will use Axiom 5 of Euclid \ref{Axiom 5 of Euclid} and seek its equivalent formulations. The reason we are using Euclid's axiom is mostly for historic reasons. From the modern point of view the fact Euclidean planes are characterized by commutativity of translations (see \ref{CharOfEuclideanPlanesThm}) is much more important. Non-Euclidean geometries are the first noncommutative geometries in that sense.

\begin{Definition}\label{EuclideanPlaneDef}
A \textbf{Euclidean plane} is an Erlangen plane $(\Pi,\mathcal{I})$ satisfying Condition \ref{Axiom 5 of Euclid}.
\end{Definition}

\begin{Condition}[Axiom 5 of Euclid]\label{Axiom 5 of Euclid}
For every line $l(A,B)$ of a plane and for every point $C$ of a plane there is a
unique maximal line containing $C$ and parallel to $l(A,B)$.
\end{Condition}

See \cite{Roe lectures} (pp.1--2) for an interesting discussion connecting large scale geometry to Legendre's attempt of proving that Axiom 5 of Euclid is valid.

Notice \ref{ExistenceOfParallelLinesProp} says parallel lines exist and it gives a necessary and sufficient condition for uniqueness of parallel lines.

If the defect of every triangle is $0$, then the plane satisfies 
\ref{Axiom 5 of Euclid} by \ref{DefectViewingAngleConnectionDef} and \ref{ExistenceOfParallelLinesProp}.

\begin{Lemma}\label{TranslationsAndParallelLinesLem}
Let $(\Pi,\mathcal{I})$ be an Erlangen plane, $\tau_{AB}$ is a non-trivial translation,
and $D=\tau_{AB}(C)$. If the defect of every triangle is $0$, then 
the line $l(C,D)$ is parallel to $l(A,B)$. Consequently, $\tau_{AB}(l)=l$
for every line $l$ parallel to $l(A,B)$.
\end{Lemma}
\begin{proof}
It is clearly true if $C\in l(A,B)$, so assume $C\notin l(A,B)$ and the two lines intersect at $B$. Consider $E=\tau_{CD}^{-1}(B)$.
Since $\tau_{AB}=\tau_{EB}\circ \tau_{AE}$ by \ref{CompositionOfTranslationsProp}
and $\tau_{CD}=\tau_{EB}$, we infer $\tau_{AE}(C)=C$. By \ref{NoFixedPointsForTranslations}
points $A$ and $E$ coincide resulting in $l(A,B)=l(C,D)$, a contradiction.

If $X\in \tau_{AB}(l)\setminus l$ for some point $X$, then $X=\tau_{AB}(Y)$
for some $Y\in l$ and $l(Y,X)$ is parallel to $l(A,B)$, hence equal to $l$, a contradiction.
\end{proof}

\begin{Theorem}\label{CharOfEuclideanPlanesThm}
Let $(\Pi,\mathcal{I})$ be an Erlangen plane. The following conditions are equivalent:\\
1. $(\Pi,\mathcal{I})$ is a Euclidean plane.\\
2. The defect of every triangle in $(\Pi,\mathcal{I})$ is $0$.\\
3. For any translation $\tau$ one has $\tau=\tau_{CD}$ if $D=\tau(C)$.\\
4. The defect of some non-flat triangle in $(\Pi,\mathcal{I})$ is $0$.\\
5. Every two translations commute.\\
6. Translations form an Abelian subgroup of $\mathcal{I}$.\\
7. Translations form a subgroup of $\mathcal{I}$.
\end{Theorem}
\begin{proof}
1.$\implies$2. Consider three non-collinear points $A$, $B$, and $C$. Let $M$ be the midpoint of $s[B,C]$ and let $N$ be the midpoint of $s[A,C]$. The image $i_M(l(A,B))$
is a line passing through $C$ and not intersecting $l(A,B)$ (if the lines had an intersection at $D$, $i_M(D)$ also would belong to both lines). For the same reason 
the image $i_N(l(A,B))$
is a line passing through $C$ and not intersecting $l(A,B)$. Hence the two lines are equal and looking at angles at $C$ one sees the defect is $0$.

2.$\implies$3. 
It is obviously so for all points $C\in l(A,B)$, where $\tau=\tau_{AB}$, so assume
$C$ is outside of the line $l(A,B)$. 

 $l(C,D)$ does not intersect $l(A,B)$ by \ref{TranslationsAndParallelLinesLem} and
\ref{TranslationsAndParallelLinesLem} shows every translation $\tau_{AB}$ sends every parallel line $l$ to $l(A,B)$ to itself. Therefore the same is true of $\tau_{CD}\circ \tau_{AB}^{-1}$.
However, this isomorphism has $D$ as a fixed point, hence the line passing through $D$ and perpendicular to $l(C,D)$ consists of fixed points of $\tau_{CD}\circ \tau_{AB}^{-1}$. In addition, $\tau_{CD}\circ \tau_{AB}^{-1}$ has $D$ as a fixed point, so $\tau_{CD}\circ \tau_{AB}^{-1}=id$.

3.$\implies$4. Consider two different points $A$ and $B$. Pick $C$ on the line perpendicular to $l(A,B)$ crossing $l(A,B)$ at $A$. Look at $D=\tau_{AB}(C)$.
The quadrangle $ABCD$ has the sum of angles equal to $360$, hence both trangles $ABD$ and $ACD$ have the defect equal to $0$.

4.$\implies$2. That is the content of \ref{DefectZeroCase}.\\
2.$\implies$1. Use \ref{Axiom 5 of Euclid}, \ref{DefectViewingAngleConnectionDef}, and \ref{ExistenceOfParallelLinesProp}.

2.$\implies$6. By 3. the composition of translations is a translation. Obviously, the inverse of a translation is a translation. 

Suppose $A$, $B$, and $C$ are three points of the plane. We need to show
$\tau_{AB}\circ \tau_{BC}=\tau_{BC}\circ \tau_{AB}$ and use 3) to conclude translations for an Abelian subgroup of isomorphisms of the plane. Let $D=\tau_{BC}(A)$.
Notice $D=i_M(B)$, where $M$ is the midpoint of $s[A,C]$. Therefore $C=\tau_{AB}(D)$
and $\tau_{AB}=\tau_{DC}$, $\tau_{AD}=\tau_{BC}$ by 2). Now $\tau_{AB}\circ \tau_{BC}=\tau_{DC}\circ \tau_{AD}=\tau_{AC}=\tau_{BC}\circ \tau_{AB}$
and we are done.

5.$\implies$2. Suppose $\Delta(ABC)$ is a non-flat triangle.
Since $\tau_{BC}\circ \tau_{AB}=  \tau_{AB}\circ \tau_{BC}$,
$$\tau_{CA}\circ \tau_{BC}\circ \tau_{AB}=  \tau_{CA}\circ \tau_{AB}\circ \tau_{BC}=
 \tau_{AB}\circ \tau_{CA}\circ \tau_{BC}.$$
The left side it a rotation at $A$ by \ref{CompositionOfTranslationsProp}
and the right side is the rotation at $B$, hence both are equal to the identity.
Since the rotations are by the defect of the triangle $\Delta(ABC)$, the defect is $0$.

6.$\implies$5. is obvious.

6.$\implies$7. is obvious.

7.$\implies$2. Apply \ref{CompositionOfTranslationsProp} and \ref{NoFixedPointsForTranslations}.
\end{proof}

\begin{Proposition}[Similarity of triangles]\label{SimilarityOfTrianglesProp}
Let $r_1$ and $r_2$ be two maximal rays emanating from a point $O$ of an Euclidean plane
and let $t > 0$ be a real number. If $A_1,A_2\in r_1$ and $B_1,B_2\in r_2$ are points such that $t\cdot |OA_1|=|OA_2|$ and $t\cdot |OB_1|=|OB_2|$, then $t\cdot |A_1B_1|=|A_2B_2|$.
\end{Proposition}
\begin{proof}
By proving \ref{SimilarityOfTrianglesProp} for $t$ being natural, we get \ref{SimilarityOfTrianglesProp} is true for $t$ being rational, hence also valid for all real $t$.

Suppose $t=n$ is natural. By \ref{CompositionOfTranslationsProp} one has
$\tau_{A_1A_2}=\tau_{OA_2}\circ \tau_{A_1O}$ and
$\tau_{B_1B_2}=\tau_{OB_2}\circ \tau_{B_1O}$.
\ref{DividingLengthViaTranslationsProp} combined with \ref{CharOfEuclideanPlanesThm}  gives
$\tau_{A_1A_2}^n= \tau_{OA_2}^n\circ \tau_{A_1O}^n=
\tau_{OB_2}\circ \tau_{B_1O}=\tau_{B_1B_2}$ and we are done.
\end{proof}

\begin{Corollary}[Pythagoras Theorem]\label{Pythagoras Theorem}
If the angle at $A$ of the triangle $\Delta(ABC)$ on an Euclidean plane is a right angle
and $|AB|=t\cdot |BC|$, $|AC|=w\cdot |BC|$, then $t^2+w^2=1$.
\end{Corollary}
\begin{proof}
Drop the height from $A$ on the side $s[B,C]$ and analyze similarity of triangles.
\end{proof}

\begin{Theorem}\label{CharOfEuclideanPlanesThm}
Each Euclidean plane $(\Pi,\mathcal{I})$ is isomorphic to the complex plane.
\end{Theorem}
\begin{proof}
Pick a line $l_1$ and two different points $O$ and $P$ on $l_1$. Rotate $l_1$ at $O$ by $90$ degrees and call the resulting line $l_2$. Let $Q\in l_2$ be the image of $P$ 
under rotation.
Let $x:l_1\to \mathbb{R}$ be
the isomorphism to the reals with $x(O)=0$ and $x(P)=1$. Let $y:l_2\to\mathbb{R}$ be the isomorphism with $y(O)=0$ and $y(Q)=1$. The function $x$ can be extended over the whole $\Pi$ as follows: given a point $C\in \Pi$ select the line $l_3$ passing through $C$ and parallel to $l_1$. That line intersects $l_2$ at a point $C'$
and the translation $\tau_{C'O}$ sends $C$ to a point $C''$ on $l_1$. We define $x(C)$ as equal to $x(C'')$. Similarly, one extends $y$ over the whole $\Pi$ and now one can
combine $x$ and $y$ to the function $z$, namely
$A$ is sent to $x(A)+i\cdot y(A)$, from $\Pi$ to the complex plane. Using the Pythagoras Theorem \ref{Pythagoras Theorem} one can see it has the property
that $|UV|=t\cdot |OP|$ implies $|z(U)-z(V)|=t$. Hence, for every $f\in \mathcal{I}$
the function $z\circ f\circ z^{-1}$ preserves distances on the complex plane and must belong to the isomorphisms of the complex plane. As in
 \ref{CharOfIsoOfErlangenLines} one can see $z$ is an isomorphism of Erlangen planes.
\end{proof}

\section{Models of hyperbolic geometry}

In a typical introduction to hyperbolic geometry one usually defines the hyperbolic length first and then investigates isometries (see \cite{Stah} for example). Our approach allows for a much simpler exposition. Namely, we start with the group of isomorphisms, choose a line, and propagate its betweenness relation all over to create a plane. This approach corresponds very well to the way measurements are made: one has a ruler and uses it to draw lines all over the plane.

\begin{Theorem}\label{PropagationTheorem}
Let $(\Pi,\mathcal{I})$ be a set $\Pi$ with a subgroup $\mathcal{I}$ of its bijections.
Suppose there is a line $l_0$ in $\Pi$ and a point $A_0\in l_0$ satisfying the following properties:\\
i. if $f\in \mathcal{I}$, $f(A_0)=A_0$, and $f(B)\in l_0$ for some $B\in l_0$ different from $A_0$, then $f(l_0)=l_0$ and $f_0$ preserves betweenness of $l_0$,\\
ii. $(l_0,\mathcal{J})$ is an Erlangen line, where $\mathcal{J}$ consists of restrictions to $l_0$ of functions in $\mathcal{I}$ that preserve $l_0$,\\
iii. for every two points $A$, $B$ of $\Pi$ there is $f\in \mathcal{I}$ such that $f(A),f(B)\in l_0$,\\
iv. $\Pi\setminus l_0$ is the union of two disjoint non-empty sets $h_+$ and $h_-$ with the property
that every $f\in \mathcal{I}$ preserving $l_0$ either preserves each of $h_+$, $h_-$, or flips them,\\
v. every function $f\in \mathcal{I}$ that fixes $A_0$ but does not preserve $l_0$
sends one of the two maximal rays in $l_0$ emanating from $A_0$ to  $h_+$ and the other to $h_-$,\\
vi. if $f(l_0)$ is disjoint from $l_0$ for some $f\in \mathcal{I}$, then $f(l_0)$ is contained in one of the sets $h_+$, $h_-$,\\
vii. there is exactly one function $f_0$ in $\mathcal{I}$ that fixes $l_0$ and flips sets $h_+$ and $h_-$.

If one defines betweenness on $\Pi$ by the condition that $C$ is between $A$ and $B$ whenever there is $f\in \mathcal{I}$ sending all of $A,B,C$ to $l_0$ such that $f(C)$ is between $f(A)$ and $f(B)$ in $l_0$, then this relation is well-defined, $\Pi$ is a plane, and $(\Pi,\mathcal{I})$ is an Erlangen plane.
\end{Theorem}
\begin{proof}
First, let's strengthen some conditions as follows:\\
i'. if $f\in \mathcal{I}$, $f(A)=A$, and $f(B)\in l_0$ for some $B\in l_0$ different from $A\in l_0$, then $f(l_0)=l_0$ and $f_0$ preserves betweenness of $l_0$,\\
iii'. for every two points $A$, $B$ of $\Pi$ there is $f\in \mathcal{I}$ such that $f(A),f(B)\in l_0$ and $f(A)=A_0$,\\
v'. every function $f\in \mathcal{I}$ that fixes some $A\in l_0$ but does not preserve $l_0$
sends one of the two maximal rays in $l_0$ emanating from $A$ to  $h_+$ and the other to $h_-$.

To show i'. choose, using iii., $g\in \mathcal{I}$ preserving $l_0$ and sending $A$ to $A_0$. Apply i. to $h=g\circ f\circ g^{-1}$. To show iii' use ii. To show v'. choose, using ii., $g\in \mathcal{I}$ preserving $l_0$ and sending $A$ to $A_0$. Apply v. to $h=g\circ f\circ g^{-1}$.

Suppose $f\in\mathcal{I}$ sends three points $A$, $B$, and $C$ of $\Pi$ to $l_0$
and $f(A)$ is between $f(B)$ and $f(C)$. What is needed to be shown is that if another element
$g\in\mathcal{I}$ sends $A$, $B$, and $C$ to the $l_0$, then $g(A)$ is between $g(B)$ and $g(C)$. We may assume $f(A)=A_0$ and $g(A)=A_0$ by composing $f$ and $g$
with appropriate elements $\mathcal{I}$ preserving $l_0$ and preserving betweenness of $l_0$. 
Let $h=g\circ f^{-1}$. It is an element of $\mathcal{I}$ that sends two different points of the $l_0$ to $l_0$ in addition to $h(A_0)=A_0$. By i. $h$ preserves $h_0$ and
by ii. it preserves betweenness of $l_0$. Therefore $h(f(A))$ is between $h(f(B))$ and $h(f(C))$, i.e. $g(A)$ is between $g(B)$ and $g(C)$.

Since every two points $A$, $B$ of $\Pi$ can be sent to $l_0$, the line segment $s[A,B]$ is isomorphic to a line segment in $l_0$ and must be connected. For the same reason each connected closed line segment is contained in the interior of another connected closed line segment. To demonstrate $\Pi$ is a space with maximal rays we need to show that
if two connected line segments $s[A,B]$ and $s[A,C]$ have a common interior point, then one of them is contained in the other. By applying an element of $\mathcal{I}$
we can reduce it to the case $A=A_0$ and the common interior point being $D\in l_0$.
Choose $f\in \mathcal{I}$ sending $A,D,B$ to $l_0$. We may assume $f(A_0)=A_0$
by composing $f$ with another element of $ \mathcal{I}$ (use ii.). Now $f$ preserves $l_0$ resulting in $B\in l_0$. For the same reason $C\in l_0$ and one of $s[A,B]$ and $s[A,C]$ is contained in the other since $l_0$ is a space with maximal rays.

In order to apply Theorem \ref{EPVsRM} we need to show $(\Pi,\mathcal{I})$
is a space with rigid motions and $\Pi$ is a plane. Given a maximal ray $r$ in $\Pi$
one can use iii'. and ii. to send $r$ to a particular maximal ray in $l_0$ emanating from $A_0$. That shows $\Pi$ is isotropic. The rigidity of elements of $\mathcal{I}$ can be shown as follows: given $f(s[A,B])\subset s[A,B]$, where $A\ne B$, find $g\in \mathcal{I}$ sending
$s[A,B]$ to $l_0$. Now, $h=g\circ f\circ g^{-1}$ preserves $l_0$ by i. 
and $h(s[g(A),g(B)]\subset s[g(A),g(B)]$. Consequently, $h(s[g(A),g(B)]= s[g(A),g(B)]$
leading to $f(s[A,B])= s[A,B]$.

Notice that maximal lines $l(A,B)$ in $\Pi$ are of the form $f(l_0)$ for some $f\in \mathcal{I}$. 

Observe each of $h_+$ and $h_-$ is convex in $\Pi$. Indeed, if $A,B\in h_+$
and $l(A,B)$ does not intersect $l_0$, then vi. says $l(A,B)\subset h_+$.
Consider $l(A,B)$ containing $C\in l_0$. Choose $f\in \mathcal{I}$
sending $l_0$ onto $l(A,B)$ and fixing $C$. By v'. $f$
sends one of the two maximal rays in $l_0$ emanating from $C$ to  $h_+$ and the other to $h_-$. That means $s[A,B]$ does not contain $C$ and $s[A,B]\subset h_+$.
For similar reasons $s[A,B]$ intersects
$l_0$ if $A\in h_+$ and $B\in h_-$. 

Given any maximal line $l$ in $\Pi$ we represent it as $f(l_0)$ for some $f\in \mathcal{I}$ and notice that $\Pi\setminus l$ has $f(h_+)$ and $f(h_-)$ as components.

Apply vii. and Theorem \ref{EPVsRM} to conclude the proof.
\end{proof}

If one wishes to apply some knowledge of topology, then Theorem \ref{PropagationTheorem} has a simpler version:

\begin{Theorem}\label{TopPropagationTheorem}
Let $(\Pi,\mathcal{I})$ be a topological space $\Pi$ with a subgroup $\mathcal{I}$ of its homeomorphisms.
Suppose there is a topological line $l_0$ which is closed in $\Pi$ and a point $A_0\in l_0$ satisfying the following properties:\\
i. if $f\in \mathcal{I}$, $f(A_0)=A_0$, and $f(B)\in l_0$ for some $B\in l_0$ different from $A_0$, then $f(l_0)=l_0$,\\
ii. $(l_0,\mathcal{J})$ is an Erlangen line, where $\mathcal{J}$ consists of restrictions to $l_0$ of functions in $\mathcal{I}$ that preserve $l_0$,\\
iii. for every two points $A$, $B$ of $\Pi$ there is $f\in \mathcal{I}$ such that $f(A),f(B)\in l_0$,\\
iv. $\Pi\setminus l_0$ is the union of two disjoint, non-empty, open, and connected sets $h_+$ and $h_-$,\\
v. there is exactly one function in $\mathcal{I}$ that fixes $l_0$ and flips sets $h_+$ and $h_-$.

If one defines betweenness on $\Pi$ by the condition that $C$ is between $A$ and $B$ whenever there is $f\in \mathcal{I}$ sending all of $A,B,C$ to $l_0$ such that $f(C)$ is between $f(A)$ and $f(B)$ in $l_0$, then this relation is well-defined, $\Pi$ is a plane, and $(\Pi,\mathcal{I})$ is an Erlangen plane.
\end{Theorem}
\begin{proof}
\ref{TopPropagationTheorem} can be proved directly following the ideas in the proof of \ref{PropagationTheorem}. On the other hand iv.--vi. of \ref{PropagationTheorem} follow from connectivity of $h_+$ and $h_-$.
\end{proof}

\begin{Example}[Poincare Model]\label{Poincare Model}
Let $\mathbb{PM}^2$ be the interior of the unit disk 
$$\{z=x+y\cdot i\mid x^2+y^2 < 1\}$$
in the complex plane and let $\mathcal{I}$
be the union of two sets:\\
a. one consisting of all functions $f(z)=\frac{a\cdot z+\bar c}{c\cdot z+\bar a}$, where $a,c$ are complex numbers satisfying $|a| > |c|$,\\
b. the other consisting of all functions $f(z)=\frac{-a\cdot \bar z+\bar c}{-c\cdot \bar z+\bar a}$, where $a,c$ are complex numbers satisfying $|a| > |c|$,\\
In both cases one has $|f(0)| < 1$ and $|z|=1$
implies $|f(z)|=1$. In other words, $f(\mathbb{PM}^2)= \mathbb{PM}^2$.

We consider the intersection $l_0$ of $\mathbb{PM}^2$ with the $x$-axis equipped with the standard relation of betweenness. This relation is propagated to the whole $\mathbb{PM}^2$ as follows: $A$ is between $B$ and $C$ if there is $f\in\mathcal{I}$ sending all points $A$, $B$, and $C$ to the $x$-axis and $f(A)$ is between $f(B)$ and $f(C)$. The result is an Erlangen plane that is not Euclidean.

Let $h_+=\{x+yi\in \mathbb{PM}^2\mid y > 0\}$ and let $h_-=\{x+yi\in \mathbb{PM}^2\mid y < 0\}$. Both are open and connected.

For simplicity, let's apply \ref{TopPropagationTheorem} (applying \ref{PropagationTheorem} is easy but a bit longer). Put $A_0=0$.
The restrictions $\mathcal{J}$ of functions in $\mathcal{I}$ which preserve $l_0$ are those from Hyperbolic line II, hence $(l_0,\mathcal{J})$ is an Erlangen line by \ref{The hyperbolic line II}. That means ii. of \ref{TopPropagationTheorem} holds. Among the functions in $\mathcal{I}$ fixing $l_0$ exactly one ($z\to \bar z$) flips sets $h_+$ and $h_-$ and v. of \ref{TopPropagationTheorem} holds. iv. is obvious.

Every point $A$ in $ \mathbb{PM}^2$ can be moved to $l_0$ by a function of the form $z\to c\cdot z$ for some $c$ of modulus $1$, and then sent to $A_0$ by an element of $\mathcal{J}$. Therefore iii. follows as for every point $B$ of $\Pi$ there is $f\in \mathcal{I}$ of the form $z\to c\cdot z$ for some $c$ of modulus $1$ such that $f(B)\in l_0$ and $f(A_0)=A_0$. 

Functions $f$ in $\mathcal{I}$ fixing $A_0$ are of one of two forms: $f(z)=a\cdot z$, where $a$ is of modulus $1$, and $f(z)=a\cdot \bar z$, where $a$ is of modulus $1$. Those which do not preserve $l_0$ must have $a$ not being real. Therefore i. of \ref{TopPropagationTheorem} holds.

To see that $(\mathbb{PM}^2,\mathcal{I})$ is not a Euclidean plane it suffices to find three parallel lines such that every two of them lie on the same side of the other one.

Look at $f(z)=\frac{a\cdot z+\bar c}{c\cdot z+\bar a}$ that sends $-1$ to itself and sends $1$ to $i$. $a=2+i$ and $c=-i$ works. If $z$ is real, the real part of $f(z)$ is negative. That means sending $z$ to the negative of the conjugate of $f(z)$ produces a line in $h_+$ disjoint with $f(l_0)$.
\end{Example}

\begin{Example}[Upper half-plane model]\label{Upper half-plane model}
Let $\mathbb{H}^2$ be the upper half-plane $$\{z=x+y\cdot i\mid y > 0\}$$ and let $\mathcal{I}$
be the union of two sets:\\
a. one consisting of all rational functions $f(z)=\frac{a\cdot z+b}{c\cdot z+d}$, where $a,b,c,d$ are real
and $a\cdot d-b\cdot c > 0$,\\
b. the other consisting of all functions of the form $f(z)=\frac{-a\cdot \bar z+b}{-c\cdot \bar z+d}$, where $a,b,c,d$ are real
and $a\cdot d-b\cdot c > 0$.

We consider the upper part of the $y$-axis with the standard relation of betweenness and that relation is propagated to the whole $\mathbb{H}^2$ as follows: $A$ is between $B$ and $C$ if there is $f\in\mathcal{I}$ sending all points $A$, $B$, and $C$ to the $y$-axis and $f(A)$ is between $f(B)$ and $f(C)$. The result is an Erlangen plane that is not Euclidean.

Define $l_0$ as the upper part of the $y$-axis and consider the quadrants $h_+=\{x+yi\in \mathbb{H}^2\mid x > 0\}$ and $h_-=\{x+yi\in \mathbb{H}^2\mid x < 0\}$. 

We will apply \ref{TopPropagationTheorem} (applying \ref{PropagationTheorem} is easy but a bit longer). Put $A_0=0+i$.
The restrictions $\mathcal{J}$ of functions in $\mathcal{I}$ which preserve $l_0$ are of the form $f(z)=c\cdot z$ for some real $c > 0$ or of the form $f(z)=\frac{-c}{\bar z}$ for some real $c > 0$. $(l_0,\mathcal{J})$ is basically the Hyperbolic line I, hence an Erlangen line by \ref{The hyperbolic line I}. That means ii. of \ref{TopPropagationTheorem} holds. Among the functions in $\mathcal{I}$ fixing $l_0$ exactly one ($z\to \frac{-1}{\bar z}$) flips sets $h_+$ and $h_-$ and v. of \ref{TopPropagationTheorem} holds. iv. is obvious.

Every point $A$ in $ \mathbb{H}^2$ can be moved to $l_0$ by a function of the form $z\to z-c$ for some real $c$, and then sent to $A_0$ by a dilation. Therefore, to show iii. it suffices to demonstrate that for every point $B$ of $\Pi$ there is $f\in \mathcal{I}$ such that $f(B)\in l_0$ and $f(A_0)=A_0$. Functions $f$ in $\mathcal{I}$ fixing $A_0$ are of one of two forms: $f(z)=\frac{a\cdot z+b}{-b\cdot z+a}$, where $a,b$ are real
and $a^2+b^2 > 0$, and $f(z)=\frac{-a\cdot \bar z+b}{b\cdot \bar z+a}$, where $a,b$ are real and $a^2+b^2 > 0$. Those which do not preserve $l_0$ must have $a\ne 0$ and $b\ne 0$, hence can be reduced to one of two forms: $f(z)=\frac{z+b}{-b\cdot z+1}$, where $b\ne $ is real, and $f(z)=\frac{-\bar z+b}{b\cdot \bar z+1}$, where $b\ne 0$ is real.
Notice in both cases $f(t\cdot i)$ being purely imaginary can happen only if $t=1$. Therefore i. of \ref{TopPropagationTheorem} holds.

Given any complex number $u+i\cdot v$ with $v > 0$ and $u\ne 0$ we find intersections with the $x$-axis
of the Euclidean circle centered on the $x$-axis and passing through $(0,1)$ and $(u,v)$.
That leads to the equation $x^2+1=(u-x)^2+v^2$ or $x=\frac{u^2+v^2-1}{2u}$
for the center of the circle and $r^2=1+x^2$ for its radius.
The function we are looking for is $f(z)=c\cdot \frac{z-x+r}{z+x+r}$.
For some positive $c$ it satisfies $f(i)=i$ and $f(u+i\cdot v)$ is purely imaginary for all positive $c$. Therefore iii. of \ref{TopPropagationTheorem} holds.

To see that $(\mathbb{H}^2,\mathcal{I})$ is not a Euclidean plane look at $f(z)=2\cdot z$. It is a translation from $1$ to $2$ but it is not a translation from $z$ to $2z$
for $z$ outside of the $y$-axis. Indeed, the set of points $\{(x,y) | \frac{y}{x}=c\}$
is not a line in $\mathbb{H}^2$ for any constant $c$.
\end{Example}

\begin{Example}[Beltrami Model]\label{Beltrami Model}
Let $\mathbb{BM}^2$ be the the upper part of the hyperboloid $\{(t,x,y)\mid t^2-x^2-y^2=1\}$ (i.e. $t > 0$) and let $\mathcal{I}$ be the orthochronous group $O^+(1,2)$
of Lorentz transformations. It is the group of linear transformations of the space $\mathbb{R}^3$ preserving the quadratic form $t^2-x^2-y^2$ and preserving the orientation of $t$.

We consider the intersection $l_0$ of $\mathbb{BM}^2$ with the $tx$-plane equipped with the relation of betweenness determined by the order on $x$. This relation is propagated to the whole $\mathbb{BM}^2$ as follows: $A$ is between $B$ and $C$ if there is $f\in\mathcal{I}$ sending all points $A$, $B$, and $C$ to the $tx$-plane and $f(A)$ is between $f(B)$ and $f(C)$. The result is an Erlangen plane that is not Euclidean.

Let $h_+=\{(t,x,y)\in \mathbb{BM}^2\mid y > 0\}$ and $h_-=\{(t,x,y)\in \mathbb{BM}^2\mid y < 0\}$. Both are open and connected in $ \mathbb{BM}^2$ which is topologically a plane.

We will apply \ref{TopPropagationTheorem} (applying \ref{PropagationTheorem} is easy but a bit longer). Put $A_0=(1,0,0)$.

The restrictions $\mathcal{J}$ of functions in $\mathcal{I}$ which preserve $l_0$ are those from Hyperbolic line III, hence $(l_0,\mathcal{J})$ is an Erlangen line by \ref{The hyperbolic line III}. That means ii. of \ref{TopPropagationTheorem} holds. Among the functions in $\mathcal{I}$ fixing $l_0$ exactly one ($(t,x,y)\to (t,x,-y)$) flips sets $h_+$ and $h_-$. Therefore, v. of \ref{TopPropagationTheorem} holds. iv. is obvious.

Functions $f$ in $\mathcal{I}$ fixing $A_0$ are the forms: $f(t,x,y)=(t,h(x,y)$, where 
$h(x,y)$ is a linear transformation preserving the quadratic form $x^2+y^2$. Those which do not preserve $l_0$ move all $l_0\setminus \{A_0\}$ outside of $l_0$. Therefore i. of \ref{TopPropagationTheorem} holds.

Every point $A$ in $ \mathbb{BM}^2$ can be moved to $l_0$ by a boost in $y$-direction, and then sent to $A_0$ by an element of $\mathcal{J}$. Therefore iii. follows as for every point $B$ of $\mathbb{BM}^2$ there is $f\in \mathcal{I}$ fixing $A_0$ such that $f(B)\in l_0$. 

Maximal lines in $\mathbb{BM}^2$ are non-empty intersections of $\mathbb{BM}^2$
with Euclidean planes in $\mathbb{R}^3$ passing through $(0,0,0)$. To see that $(\mathbb{BM}^2,\mathcal{I})$ is not a Euclidean plane it suffices to find three disjoint lines
such that every pair of them is on the same side of the third one. Let one of them be $l_0$, let the second be the intersection of $\mathbb{BM}^2$ with the line $t=x+y$,
and let the third one be the intersection of $\mathbb{BM}^2$ with the line $t=-x+y$.
Look at their projections onto the $xy$-plane.
\end{Example}

\section{Spherical geometry}

One can use our approach and create an axiomatization for spherical geometry. In contrast to typical textbooks, we do not use lines. Instead, we use circles.

\begin{Definition}\label{SphereDef}
A \textbf{sphere} is a space $\Sigma$ with maximal rays containing at least three points and equipped with an antipodal map $a$  such that each pair of non-antipodal points $A$, $B$ in $\Sigma$ is contained in a unique circle $c(A,B)$. $c(A,B)$ has the property that its complement $\Sigma\setminus c(A,B)$ can be expressed as the union of two disjoint non-empty convex sets $h_1$ and $h_2$ (called \textbf{half-spheres}) such that $s[C,D]$ intersects $c(A,B)$ at exactly one point whenever $C\in h_1$ and $D\in h_2$. 
Also, $s[A,B]$ consists of exactly two points $A$ and $B$ if they are antipodal.
\end{Definition}

One can show spheres are spherical Pasch spaces whose boundary at infinity is a circle. Conversely, every spherical Pasch space whose boundary at some point is a circle, must be a sphere.

\begin{Definition}
An \textbf{Erlangen sphere} is a pair $(\Sigma,\mathcal{I})$ consisting of a sphere $\Sigma$ equipped with a subgroup $\mathcal{I}$ of isomorphisms of $\Sigma$
such that the Condition \ref{Homogeneity and rigidity of the sphere} is satisfied.
\end{Definition}

\begin{Condition}[Homogeneity and rigidity of the Erlangen sphere]\label{Homogeneity and rigidity of the sphere}
For every two pairs $(h_1,r_1)$, $(h_2,r_2)$ consisting of a half-sphere and a maximal ray in its boundary there is exactly one isomorphism sending $h_1$ onto $h_2$ and sending $r_1$ onto $r_2$.
\end{Condition}

The basic distinction between Erlangen spheres and Erlangen planes is that in Erlangen spheres the defect of every triangle is negative. Indeed, each Erlangen sphere has a triangle with all angles being right. There cannot be a triangle with defect $0$ as that leads to arbitrarily large distances and each length in an Erlangen sphere is at most $2$ (recall we consider the right angle to be the unit of length on Erlangen circles). Therefore each triangle has the sum of angles bigger than $180$ degrees for continuity reasons. There is a formal similarity between Erlangen spheres and hyperbolic planes in the sense that the same congruence results hold for triangles.

Spherical geometry is simpler for the following reason:

\begin{Theorem}\label{SphericalPaschSpacesThm}
Let $(\Pi,\mathcal{I})$ be a space with rigid motions. If $\Pi$ is a spherical Pasch space, then for every $A\in \Pi$ the functions in $\mathcal{I}$ fixing $A$ induce rigid motions of $\partial(\Pi,A)$.
\end{Theorem}
\begin{proof}
Suppose $f\in \mathcal{I}$ fixes $A$ and preserves $ray[r_1,r_2]$ in $\partial(\Pi,A)$ for some rays $r_1, r_2$ emanating from $A$ such that $f(r_1)=r_1$ but $f(r_2)\ne r_2$. Therefore $f|r_1$ is the identity. Choose $B\in r_2$ different from $A$. Let $r_3:=ray[A,f(B)]$. We may assume $r_3$ is strictly between $r_1$ and $r_2$. Hence $a(r_1)$ is between $r_2$ and $ray[A,a(f(B))]$
which means $a(r_1)$ intersects $s[B,a(f(B))]$ at some point $C$. Now, $ray[C,B]$ is antipodal to $ray[C,a(f(B))]$, so $r_4:=ray[C,B]$ must contain $f(B)$. Since $f(C)=C$ we see $f(r_4)=r_4$, therefore $f|r_4$ is the identity, in particular $f(B)=B$, a contradiction.
\end{proof}

\begin{Corollary}\label{ErlangenSpheresViaRigidMotions}
$(\Pi,\mathcal{I})$ is an Erlangen sphere if and only if $\Pi$ is a sphere, $(\Pi,\mathcal{I})$ is a space with rigid motions, and $(\partial(\Pi,A),\mathcal{I}_A)$
is an isotropic space for some $A\in \Pi$.
\end{Corollary}

More generally, spherical geometry can be considered as the study of spherical Pasch spaces. More narrowly, spherical geometry can be considered as the study of boundaries at infinity of Pasch spaces whose all line segments are connected. Are these ideas the same?

\begin{Question}
Is every spherical Pasch space the boundary at infinity of some regular Pasch space?
\end{Question}

\section{Erlangen spaces}

We would like to generalize Erlangen lines, circles, planes, and spheres to Erlangen spaces in such a way that the boundary at infinity of $\Pi$ is also an Erlangen space if $\Pi$ is an Erlangen space. This idea makes the task actually quite simple and straightforward. Our plan is to generalize \ref{ErlangenLinesViaRigidMotionsCor} and \ref{CharErlangenPlanesViaBoundary} while using \ref{RigidMotionsViaRaysProp}. Also, we plan to follow the idea of what a Klein geometry ought to be as expressed by Heinrich Guggenheimer \cite{Gugg}  (on p. 139): \emph{A Klein geometry is the theory of geometric invariants of a transitive transformation group}. 

In the simplest language an Erlangen space is a Pasch space of rigid motions such that applying the operator of the boundary at infinity leads to spaces with rigid motions. That's the strongest definition possible. The weakest definition possible is to say that an Erlangen space is a Pasch space of rigid motions such that applying the operator of the boundary at infinity leads to isotropic spaces. We do not know if those are equivalent (it hinges on properties of Euclidean planes - see \ref{ErlangenPlanesViaRigidMotionsQuestion}), therefore our choice is to use the strongest definition possible.

Notice our definition of an Erlangen line $l$ involved choices of $A\in l$ and $r\in\partial(l,A)$.
Erlangen circles were defined similarly. In case of Erlangen planes $\Pi$ we had a choice of $A\in \Pi$, $r\in\partial(\Pi,A)$, and a choice of a half-plane can be viewed as choosing $h\in\partial(\partial(\Pi,A),r)$. We can generalize this observation as follows:
\begin{Definition}\label{ScaffoldingOfRaysDef}
Let $\Pi$ be a Pasch space and $n\ge 0$. An \textbf{$n$-scaffolding of rays} in $\Pi$
is a sequence $(\Pi_i,r_i)$, $0\leq i\leq n$, such that $\Pi_0= \Pi$, $r_0\in \Pi_0$ is a point,
$\Pi_{i+1}=\partial(\Pi_i,r_i)$ and $r_{i+1}\in \Pi_{i+1}$ for $i < n$.
\end{Definition}

The concept of an $n$-scaffolding of rays allows for a simple definition of dimension of Pasch spaces.

\begin{Definition}\label{DimensionOfPaschSpacesDef}
Let $\Pi$ be a Pasch space and $n\ge 0$. The \textbf{dimension} of $\Pi$
is at least $n$ if there is an $n$-scaffolding of rays in $\Pi$.
We say $\dim(\Pi)=n$ if $\dim(\Pi)\ge n$ holds but $\dim(\Pi)\ge n+1$
is false.
\end{Definition}

\begin{Example}
$\dim(\Pi)=1$ means $\Pi$ is either a line or a circle. $\dim(\Pi)=2$ means $\Pi$ is either a plane or a sphere.
\end{Example}

\begin{Definition}\label{ErlangenSpaceDef}
An \textbf{Erlangen space} is a pair $(\Pi,\mathcal{I})$ such that the following conditions are satisfied:\\
1. $\Pi$ is a Pasch space,\\
2. $(\Pi,\mathcal{I})$ is a space with rigid motions,\\
3. Given an $n$-scaffolding of rays $(\Pi_i,r_i)$, $0\leq i\leq n$, in $\Pi$,
the group of bijections $\mathcal{I}_n$ of $\Pi_n$ induced by $\mathcal{I}$
makes $(\Pi_n,\mathcal{I}_n)$ a space with rigid motions,\\
4. If $\Pi$ is a spherical Pasch space, then the antipodal map belongs to $\mathcal{I}$.
\end{Definition}

\begin{Example}
Erlangen lines, circles, planes, and spheres are Erlangen spaces.
\end{Example}

\begin{Corollary}
If $(\Pi,\mathcal{I})$ is an Erlangen space, then its boundaries at infinity (equipped with induced morphisms from $\mathcal{I}$) are Erlangen spaces.
\end{Corollary}

\begin{Theorem}\label{ScaffoldingToScaffoldingThm}
If $(\Pi,\mathcal{I})$ is an Erlangen space, then for any two $n$-scaffoldings of rays in $\Pi$ there is $f\in \mathcal{I}$ sending the first $n$-scaffolding of rays in $\Pi$ to the second $n$-scaffolding of rays in $\Pi$.
\end{Theorem}
\begin{proof}
By induction on $n$. For $n=1$ it corresponds to the definition of spaces with rigid motions. Suppose \ref{ScaffoldingToScaffoldingThm} holds for $n=k$
and suppose we have two $(k+1)$-scaffoldings of rays in $\Pi$.
Pick $g\in \mathcal{I}$ sending the basepoint of the initial element of the first scaffolding to the the basepoint $B$ of the initial element of the second scaffolding. Using $g$ create the third scaffolding as the image under $g$ of the first scaffolding.
By removing initial elements of the third and the second scaffolding we arrive at two $k$-scaffoldings of rays in $\partial(\Pi,B)$. By the inductive hypothesis, there is $h$ in $\mathcal{I}$ sending the modified third scaffolding to the modified second scaffolding. Now, $f:=h\circ g$ sends the first scaffolding to the second scaffolding.
\end{proof}

\begin{Corollary}\label{DimensionOfErlangenSpacesCorollary}
Let $(\Pi,\mathcal{I})$ be an Erlangen space and $n\ge 0$. If $\dim(\Pi)=n$,
then $\dim(\partial(\Pi,A))=n-1$ for all $A\in \Pi$.
\end{Corollary}
\begin{proof}
Any $k$-scaffolding of rays in $\partial(\Pi,A)$ extends to a $(k+1)$-scaffolding of rays in $\Pi$. Therefore $\dim(\partial(\Pi,A))\ge \dim(\Pi,A)-1$. By \ref{ScaffoldingToScaffoldingThm} existence of an $n$-scaffolding of rays in $\Pi$ implies existence of an $(n-1)$-scaffolding of rays in $\partial(\Pi,A)$. Therefore $\dim(\partial(\Pi,A))=n-1$.
\end{proof}

\begin{Theorem}\label{RigidityOfErlangenSpacesThm}
Let $(\Pi,\mathcal{I})$ be an Erlangen space and $n\ge 0$.
If $\dim(\Pi)=n$, then for any two $n$-scaffoldings of rays in $\Pi$ there is a unique $f\in \mathcal{I}$ sending the first $n$-scaffolding of rays in $\Pi$ to the second $n$-scaffolding of rays in $\Pi$.
\end{Theorem}
\begin{proof}
It suffices to show any $f\in \mathcal{I}$ fixing an $n$-scaffolding of rays in $\Pi$ is the identity function.
By induction on $n$. It is so for $n=1$ by the results on Erlangen lines and Erlangen circles. Suppose \ref{RigidityOfErlangenSpacesThm} holds for $n=k$
and suppose we have $g\in \mathcal{I}$ fixing an $(k+1)$-scaffolding of rays in $\Pi$ starting at $A$. The induced map $f$ on $\partial(\Pi,A)$ fixes
a $k$-scaffolding of rays. By \ref{DimensionOfErlangenSpacesCorollary}
and by the inductive assumption, $f$ is identity on $\partial(\Pi,A)$.
That means $g$ preserves all maximal rays in $\Pi$ emanating from $A$.
Since $\Pi$ is a space with rigid motions, $g$ must be equal to the identity function.
\end{proof}

\begin{Theorem}\label{SphericalErlangenSpacesThm}
Let $(\Pi,\mathcal{I})$ be a space with rigid motions such that $\Pi$ is a spherical Pasch space and the antipodal map belongs to $\mathcal{I}$.
If for every two $n$-scaffoldings of rays in $\Pi$ there is $f\in \mathcal{I}$ sending the first $n$-scaffolding of rays in $\Pi$ to the second $n$-scaffolding of rays in $\Pi$, then $(\Pi,\mathcal{I})$ is an Erlangen space.
\end{Theorem}
\begin{proof}
Apply \ref{SphericalPaschSpacesThm} to conclude that, given an $n$-scaffolding of rays $(\Pi_i,r_i)$, $0\leq i\leq n$, in $\Pi$,
the group of bijections $\mathcal{I}_n$ of $\Pi_n$ induced by $\mathcal{I}$
makes $(\Pi_n,\mathcal{I}_n)$ a space with rigid motions.
\end{proof}

Ideally, we would like \ref{SphericalErlangenSpacesThm} to hold for all Pasch spaces. However, we face some obstacles in proving such generalization.
Right now we can prove a generalization of \ref{SphericalErlangenSpacesThm} under an additional condition.

\begin{Theorem}\label{RegularErlangenSpacesThm}
Let $(\Pi,\mathcal{I})$ is a space with rigid motions, where $\Pi$ is a regular Pasch space.
$(\Pi,\mathcal{I})$ is an Erlangen space if the following conditions are satisfied:\\
1. $(\partial(\Pi,A),\mathcal{I}_A)$ is a space with rigid motions for some $A\in \Pi$,\\
2. For every two $n$-scaffoldings of rays in $\Pi$ there is $f\in \mathcal{I}$ sending the first $n$-scaffolding of rays in $\Pi$ to the second $n$-scaffolding of rays in $\Pi$.
\end{Theorem}
\begin{proof}
Use \ref{SphericalPaschSpacesThm} to conclude $(\partial(\Pi,A),\mathcal{I}_A)$ is an Erlangen space for all $A\in \Pi$.
\end{proof}

The reason we cannot drop Condition 1. in \ref{RegularErlangenSpacesThm} is because we do not know if \ref{ErlangenSpheresViaRigidMotions} can be generalized to all planes.

\begin{Question}\label{ErlangenPlanesViaRigidMotionsQuestion}
Let $(\Pi,\mathcal{I})$ be a space with rigid motions such that $\Pi$ is a plane.
Is $(\Pi,\mathcal{I})$
 an Erlangen plane if $(\partial(\Pi,A),\mathcal{I}_A)$
is an isotropic space for some $A\in \Pi$?
\end{Question}

However, we can answer \ref{ErlangenPlanesViaRigidMotionsQuestion} for hyperbolic planes.

\begin{Theorem}\label{ErlangenHypPlaneCharThm}
Let $(\Pi,\mathcal{I})$ be a space with rigid motions such that $\Pi$ is a hyperbolic plane.
If $(\partial(\Pi,A),\mathcal{I}_A)$
is an isotropic space for some $A\in \Pi$, then $(\Pi,\mathcal{I})$
is an Erlangen plane.
\end{Theorem}
\begin{proof}
In the context of \ref{ErlangenHypPlaneCharThm} by a hyperbolic plane we mean a plane such that there exist two intersecting lines $l_1$ and $l_2$, both parallel to a third line $l_3$.

Suppose $f\in \mathcal{I}$ fixes $A$, preserves a maximal ray $r$ emanating from $A$, preserves both half-planes with boundary $l_A:=r\cup a(r)$, and is not equal to the identity function. Therefore $f|l_A$ is the identity. Choose and $B$ outside of $l_A$. Observe $f(B)\ne B$. Indeed, applying \ref{PaschSpacesAndPlanes2} one deduces $f$ is the identity function.
Notice $l(B,f(B))$ cannot intersect $l_A$ for the same reason.
Applying the same argument again one obtains that the visual angle of $l(B,f(B))$ from every point of $_A$ must be $180$ degrees (otherwise there are rays preserved by $f$, the rays on the boundary of the visual angle). Since $(\Pi,\mathcal{I})$ is isotropic, we conclude that for every two points $A$ and $B$ of $\Pi$ and for every line $l_A$ containing $A$ but missing $B$ there is a line $l_B$ passing through $B$ whose viewing angle from every point on the line $l_A$ is $180$ degrees. 
Therefore every line passing through $A$ but different from $l_A$ must intersect $l_B$. That means, if we choose the line for $l_B$, it must be exactly $l_A$.
Since $(\Pi,\mathcal{I})$ is isotropic, we can 
assume $l_B=l_3$ and $A$ is in $l_1\cap l_2$. That leads to a contradiction:
the line chosen for $l_B$ has to be different from either $l_1$ or $l_2$ in which case the other line intersects $l_B$ and is parallel to $l_B$ at the same time.
\end{proof}

\section{Philosophical and pedagogical implications}

Philosophically, this paper shows topology to be at the foundation of geometry in agreement with modern physics (gauge theory, string theory, etc). Also, this paper contributes to compactification of knowledge: our ideas provide a shorter path to modern science than typical expositions of foundations of geometry. Moreover, its roots are deeply connected to the original concepts of Euclid and are far removed from the formalism of Hilbert or Tarski.

From the pedagogical point of view, our approach corresponds well to the Standards for Mathematical Practice
\cite{StandardsMathPractice}, especially the following two:\\
MP7. Look for and make use of structure.\\
MP8. Look for and express regularity in repeated reasoning.

Our approach can be easily taught at the university level to students who took an Introduction to Abstract Mathematics course and have some exposure to group theory. It introduces them to an application of a powerful tool in mathematics, namely a Fixed Point Theorem, it gives them a taste of topology, and it introduces them to group actions, a fundamental concept in modern mathematics and physics (see \cite{BH}). In addition, it is a better fit for future physicists wishing to master gauge theory eventually (see \cite{Fram}). Also, it gives math majors a quicker way to understand the connection of special relativity (boosts) to hyperbolic geometry.

Notice we spent a lot of time on properties of lines. The simplicity of proofs there is useful in helping students in understanding of mathematical reasoning. Classical Euclidean geometry is like juggling many objects at the same time. Our approach allows for a slower exposure to mathematical juggling of concepts and proofs.

One could create a version of our approach for high schools by formulating the Fixed Point Theorem \ref{FixedPointTheorem} as an axiom.

Notice that some proofs in Section \ref{Convexity and connectedness} are really saying that connected line segments are compact. Usually, the concepts of connectedness and compactness are considered to be independent. Having the two concepts emerge from planar geometry is of some educational value.

Marek Kordos \cite{Kord} has an interesting account (in Polish) of dilemmas in teaching geometry in high schools in Europe.

Hung-Hsi Wu (see \cite{WuPre}) promotes a very interesting
\textbf{Fundamental Assumption of School Mathematics} (\textbf{FASM}), which states that
if an identity or an inequality $\leq$  among numbers is valid for all fractions (respectively, all rational numbers), then it is also valid for all nonnegative real numbers (respectively, all real numbers.)
Notice that in our approach to geometry it acquires a simpler form: if an identity or an inequality $\leq$  among lengths of line segments is valid for all natural numbers, then it is also valid for all nonnegative real numbers (see the proofs of \ref{CharOfIsoOfErlangenLines} and \ref{SimilarityOfTrianglesProp}). 
It is highly likely that the reason math education has such enormous problems is that math is detached from its physical roots. Teaching of fractions is detached from measurements and students do not understand that physical units are chosen arbitrarily. One should not manipulate fractions abstractly following some canned formulae. There is always a unit of measure in whatever we do. $\frac{3}{4}$ means a quantity $z$ related to two units $x$ and $y$ such that $x=4y$ and $z=3y$.
$z$ expressed as $\frac{3}{4}$ means it is expressed in terms of unit $x$. That reasoning allows for applying formulae for natural numbers to deducting of formulae for fractions.

\begin{Example}
Consider showing $\frac{a}{b}+ \frac{c}{d}= \frac{a\cdot d+b\cdot c}{b\cdot d}$ as follows. $x=\frac{a}{b}$ is expressed in Marsian feet. It is the quantity of $a$ Pluto feet and each Marsian foot is $b$ Pluto feet. The quantity $y=\frac{c}{d}$ is also expressed in Marsian feet. It is $c$ Saturn feet and each Marsian foot is 
$d$ Saturn feet. One Marsian foot is $b\cdot d$ Lunar feet, hence $x$ is $a\cdot d$ Lunar feet and $y$ is $b\cdot c$ Lunar feet.
Measured in Lunar feet the quantity $x+y$ is $a\cdot d+b\cdot c$ which is what we wanted to show because
$\frac{a\cdot d+b\cdot c}{b\cdot d}$ Marsian feet equals $a\cdot d+b\cdot c$ Lunar feet.
\end{Example}

The above example is another illustration of noncommutativity of teaching and learning (see \cite{Dyda}). Another view of this paper is that it promotes applying of Common Core ideas at the university level. Indeed, one way to see Common Core is as a movement to promote the mainstream of knowledge. This paper fits the bill very well. It shows one can combine ideas from other STEM fields (Science, Technology, Engineering, and Mathematics) in teaching geometry. In particular, one can see the importance of complex numbers here. One should wish that studying complex numbers in depth (beyond multiplying and dividing them) ought to be a basic staple of high school education with the resulting abandonment of classical trigonometry in favor of applications of the Moivre Formula.

Notice Hung-Hsi Wu (see \cite{WuGeo} and \cite{WuCCSS}) has his own set of axioms for planar geometry that is well-suited for teaching geometry in high schools and which can be considered as a useful bridge for teaching our system of axioms at the university level.

\end{document}